\documentclass[11pt, reqno, oneside]{amsart}

\usepackage{times}

\usepackage{color}
\usepackage[all,arc]{xy}
\usepackage{enumerate}
\usepackage{mathrsfs}
\usepackage{upref}

\usepackage{amsmath}
\usepackage{amsthm}
\usepackage{amsfonts}
\usepackage{amssymb}
\usepackage{amsbsy}
\usepackage{graphicx}
\usepackage{enumitem}
\usepackage{hyperref}
\usepackage{bbm}
\usepackage[mathscr]{euscript}

\usepackage{tikz}
\tikzstyle{startstop} = [rectangle, rounded corners, minimum width=2cm, minimum height=1cm,text centered, draw=black]
\tikzstyle{arrow} = [thick,->,>=stealth]

\newtheorem{theorem}{Theorem}
\newtheorem{prop}[theorem]{Proposition}
\newtheorem{lemma}[theorem]{Lemma}
\newtheorem{cor}[theorem]{Corollary}

\theoremstyle{definition}
\newtheorem{definition}[theorem]{Definition}
\newtheorem{remark}[theorem]{Remark}

\newcommand{\ms}[1]{\mathscr{#1}}
\newcommand{\mc}[1]{\mathcal{#1}}
\newcommand{\ovl}[1]{\overline{#1}}

\newcommand{\varep}{ \varepsilon }

\newcommand{\sse} {\subseteq}

\newcommand{\Z}{\mathbb{Z}}
\newcommand{\R}{\mathbb{R}}

\newcommand{\C}{\mathbb{C}}
\newcommand{\E}{\mathbb{E}}
\newcommand{\N}{\mathbb{N}}
\newcommand{\icomplex}{\textup{i}}

\newcommand{\ra}{\rightarrow}
\newcommand{\toinf}{\ra \infty}

\newcommand{\beq}{\begin{equation}}
\newcommand{\eeq}{\end{equation}}
\newcommand{\mbf}[1]{\mathbf{#1}}
\newcommand{\mbb}[1]{\mathbb{#1}}

\newcommand{\p}{\mathbb{P}}

\numberwithin{equation}{section}
\numberwithin{theorem}{section}
\newcommand{\ind}{\mathbbm{1}}
\newcommand{\ptl}{\partial}
\DeclareMathOperator*{\esssup}{ess\,sup}
\DeclareMathOperator{\Tr}{Tr}

\newcommand{\torus}{\mathbb{T}}
\newcommand{\rade}{R\aa de}
\newcommand{\groupid}{\textup{id}}

\newcommand{\rconn}{\mathbf{A}}

\newcommand{\liegroup}{G}
\newcommand{\lalg}{\mathfrak{g}}

\newcommand{\const}{C}

\newcommand{\threetorus}{\torus^3}

\newcommand{\gaction}[2]{#1^{#2}}

\usepackage{xparse}
\DeclareDocumentCommand{\toporbitspace}{o}{\ms{T}_{q\IfValueT{#1}{, #1}}}
\DeclareDocumentCommand{\ctoporbitspace}{o}{\ms{T}_{c\IfValueT{#1}{, #1}}}
\DeclareDocumentCommand{\borelorbitspace}{o}{\mc{B}_{q\IfValueT{#1}{, #1}}}
\DeclareDocumentCommand{\cborelorbitspace}{o}{\mc{B}_{c\IfValueT{#1}{, #1}}}
\newcommand{\wloop}{\ell}

\newcommand{\lalgdim}{d_\lalg}
\newcommand{\onbasislalg}{S}

\newcommand{\vectorspace}{V}
\newcommand{\oneform}{$1$-form}
\newcommand{\oneforms}{$1$-forms}

\newcommand{\nonlineardistspace}{\mc{X}}

\DeclareDocumentCommand{\nucutoff}{o o o}{\nu_{\IfValueT{#2}{#2} \IfValueT{#3}{, #3}; #1}}
\DeclareDocumentCommand{\mucutoff}{o o o}{\mu_{\IfValueT{#2}{#2} \IfValueT{#3}{, #3}; #1}}

\newcommand{\unitary}{\textup{U}}

\newcommand{\gt}{\sigma}
\newcommand{\metricspace}{d_{\torus^d}}

\newcommand{\curv}[1]{F_{#1}}

\newcommand{\dchain}{d}
\newcommand{\FT}{F}

\newcommand{\consttb}{{\const_{\textup{B}}}}
\newcommand{\exptb}{{\beta_{\textup{B}}}}
\newcommand{\constgf}{{\const_{\textup{D}}}}
\newcommand{\constfourp}{{\const_{\textup{E}}}}

\usepackage{geometry} 
\renewcommand{\tilde}{\widetilde}
\renewcommand{\hat}{\widehat}

\begin{document}

\title[YM heat flow with random distributional initial data]{The Yang--Mills heat flow with random distributional initial data}
\author{Sky Cao}
\author{Sourav Chatterjee}
\address{Department of Statistics, Stanford University, Sequoia Hall, 390 Jane Stanford Way, Stanford, CA 94305}
\email{skycao@stanford.edu}
\email{souravc@stanford.edu}
\thanks{S.~Cao was partially supported by NSF grant DMS RTG 1501767}
\thanks{S.~Chatterjee was partially supported by NSF grants DMS 1855484 and DMS 2113242}

\begin{abstract}
We construct local solutions to the Yang--Mills heat flow (in the DeTurck gauge) for a certain class of random distributional initial data, which includes the 3D Gaussian free field. The main idea, which goes back to work of Bourgain as well as work of Da Prato--Debussche, is to decompose the solution into a rougher linear part and a smoother nonlinear part, and to control the latter by probabilistic arguments. In a companion work, we use the main results of this paper to propose a way towards the construction of  3D Yang--Mills measures. 
\end{abstract}

\maketitle

\tableofcontents

\section{Introduction}\label{section:introduction}

Take any dimension $d \geq 2$. Let $\liegroup$ be a compact Lie group and let $\lalg$ denote the Lie algebra of $\liegroup$. We assume that $\liegroup \sse \unitary(N)$ for some $N \geq 1$. A {\it connection} on the trivial principal $G$-bundle $\torus^d \times G$ is a function $A: \torus^d \ra \lalg^d$, that is, a $d$-tuple of functions $A = (A_1, \ldots, A_d)$, with $A_i : \torus^d \ra \lalg$ for $1 \leq i \leq d$. Note that $A$ can also be viewed as a $\lalg$-valued {\oneform} on $\torus^d$. Thus, in this paper, we will use `connection' and `{\oneform}' interchangeably. 

The main object of this paper is to prove local existence of solutions to the Yang--Mills heat flow with random distributional initial data. The Yang--Mills heat flow (also often called the Yang--Mills gradient flow, or Yang--Mills heat equation) is the following PDE on time-dependent connections $A(t)$ (in the following, we omit the time parameter $t$):
\beq\label{eq:YM}\tag{$\textup{YM}$} \begin{split}
\ptl_t A_i &= \Delta A_i + \sum_{j=1}^d \biggl(-\ptl_{ji} A_j + [A_j, 2\ptl_j A_i - \ptl_i A_j + [A_j, A_i]] \\
&\qquad \qquad \qquad \qquad \qquad + [\ptl_j A_j, A_i]\biggr), ~~ 1 \leq i \leq d. 
\end{split}\eeq
This equation can be obtained as the gradient flow of a certain action on the space of connections, analogous to how the heat equation can be obtained as the gradient flow of the Dirichlet energy. The Yang--Mills heat flow has played a central role in various areas of mathematics, starting with the paper by Atiyah and Bott \cite{AB1983}.  See \cite[Section 1]{Feehan2016} for a historical overview of this equation and its many applications in mathematics and physics, as well as for an encyclopedic account of existing results. See also \cite{CG2013, CG2015, CG2017, Gross2016, Gross2017, OT2017, Waldron2019} for some newer results.

Actually, in this paper, we will not directly work with \eqref{eq:YM}, but rather a certain well-known variant which is often treated as equivalent to \eqref{eq:YM}. This variant is the following PDE:
\beq\label{eq:ZDDS}\tag{$\textup{ZDDS}$} 
\ptl_t A_i = \Delta A_i + \sum_{j = 1}^d [A_j, 2\ptl_j A_i - \ptl_i A_j + [A_j, A_i]], ~~ 1 \leq i \leq d. \eeq
We will refer to this as the ZDDS equation, named after the authors associated with this equation --- Zwanziger \cite{Zwan1981}, DeTurck \cite{DeT1983}, Donaldson \cite{Don1985}, Sadun \cite{Sad1987}. For a discussion as to why \eqref{eq:ZDDS} is equivalent to \eqref{eq:YM}, see, e.g., \cite[Section 1]{CCHS2020}. The advantage of \eqref{eq:ZDDS} is that it is a parabolic equation, and thus local existence is often easier to establish for \eqref{eq:ZDDS} than \eqref{eq:YM}. Indeed, one of the main methods for showing local existence of \eqref{eq:YM} for various types of initial data is to first show it for \eqref{eq:ZDDS}, and then use a well-known procedure to obtain solutions to \eqref{eq:YM} out of solutions to \eqref{eq:ZDDS} (see, e.g., \cite[Section 1.3]{CG2013}).

By now, local existence of solutions to \eqref{eq:ZDDS} has been established for various classes of initial data --- again, see the survey \cite{Feehan2016}, as well as \cite{CG2013, Gross2016}. However, as far as we can tell, there are no results for distributional initial data. In particular, there are no results which consider random distributional initial data that is too rough to be handled purely by deterministic methods. The present paper seeks to address this case. Our motivation is twofold. For one, we think that this case is of intrinsic interest --- random initial data has been studied for a variety of PDEs, such as the nonlinear Schr\"{o}dinger equation \cite{Bourgain1996, Bourgain1999, DNY2019, LRS1988, OST2021}, the nonlinear wave equation \cite{BT2008a, BT2008b}, and the Navier--Stokes equations \cite{NPS2013}. (This list of references is woefully incomplete. See, e.g., \cite{DNY2019, OST2021} for more.) 

Second, we originally came upon this problem through the companion work \cite{CaoCh2021}. In that paper, we give a proposal for constructing 3D Euclidean Yang--Mills theories (following a suggestion of Charalambous and Gross \cite{CG2013}), and in particular, we construct and study a state space that may potentially support 3D Yang--Mills measures. As evidence of this possibility, in \cite{CaoCh2021} we apply the results of the present paper to give nontrivial elements of the state space, and additionally to give a road map for completing the program and actually constructing 3D Yang--Mills measures. See \cite{CaoCh2021} for more background and discussion.

\subsection{The main result}\label{section:main-result}

We begin to build towards the statement of the main result, Theorem \ref{thm:main}. In this paper, we will often deal with functions $A(t, x)$ of both $t \in [0, \infty)$ and $x \in \torus^d$. In an abuse of notation, for $t \in [0, \infty)$, we will write $A(t)$ to denote the function on $\torus^d$ given by $x \mapsto A(t, x)$.

\begin{definition}
Let $0 < T \leq \infty$. We say that $A$ is a solution to \eqref{eq:ZDDS} on $[0, T)$ if $A \in C^\infty([0, T) \times \torus^d, \lalg^d)$ and $A$ satisfies \eqref{eq:ZDDS} on $(0, T) \times \torus^d$. Similarly, we say that $A$ is a solution to \eqref{eq:ZDDS} on $(0, T)$ if $A \in C^\infty((0, T) \times \torus^d, \lalg^d)$ and $A$ satisfies \eqref{eq:ZDDS} on $(0, T) \times \torus^d$.
\end{definition}

We next state the following classical theorem and lemma, which give local existence and uniqueness for solutions to \eqref{eq:ZDDS} with smooth initial data. Since these results concern smooth initial data, they are not new and quite classical. For instance, the local existence and regularity results stated below can be obtained by combining the various general results of \cite[Sections 17.4, 17.5, and 20.1]{Feehan2016}.  We provide proofs in Appendix \ref{appendix:deterministic-pde-proofs} if the reader is interested.


\begin{theorem}\label{thm:zdds-existence}
Let $A^0$ be a smooth {\oneform}. There exists $T > 0$ and a solution $A \in C^\infty([0, T) \times \torus^d, \lalg^d)$ to \eqref{eq:ZDDS} on $[0, T)$ with initial data $A(0) = A^0$. 
\end{theorem}

\begin{lemma}\label{lemma:zdds-uniqueness}
Let $T > 0$. Suppose that $A, \tilde{A} \in C^\infty([0, T) \times \torus^d, \lalg^d)$ are solutions to \eqref{eq:ZDDS} on $[0, T)$ such that $A(0) = \tilde{A}(0)$. Then $A = \tilde{A}$. 
\end{lemma}

Because of Lemma \ref{lemma:zdds-uniqueness}, in circumstances where the smooth initial data has been specified, we will usually say ``the solution to \eqref{eq:ZDDS}" rather than ``a solution to \eqref{eq:ZDDS}". 

Before we proceed, we need some notation. For integers $n \geq 1$, let $[n] := \{1, \ldots, n\}$. For vectors $v \in \R^n$ for some $n$, we write $|v|$ for the Euclidean norm of $v$, and we write $|v|_\infty := \max_{1 \leq i \leq n} |v_i|$ for the $\ell^\infty$ norm of~$v$. Next, since we assumed that $\liegroup \sse \unitary(N)$, the Lie algebra $\lalg$ is a real finite-dimensional Hilbert space, with inner product given by $\langle S_1, S_2 \rangle = \Tr(S_1^* S_2) = -\Tr(S_1  S_2)$ (note that $S^* = -S$ for all $S \in \lalg$, because $\liegroup \sse \unitary(N)$).

\begin{definition}\label{def:lalg-and-X}
Let $\lalgdim$ be the dimension of $\lalg$. Throughout this paper, fix an orthonormal basis $(\onbasislalg^a, a \in [\lalgdim])$ of $\lalg$.
\end{definition} 

We may thus equivalently view a ($\lalg$-valued) {\oneform} $A : \torus^d \ra \lalg^d$ as a collection $(A^a_j, a \in [\lalgdim], j \in [d])$ of functions $A^a_j : \torus^d \ra \R$, satisfying the relation
\beq\label{eq:one-form-R-valued-functions-relation} A_j = \sum_{a \in [\lalgdim]} A^a_j \onbasislalg^a, ~~ j \in [d].\eeq
Next, we recall the notation for Fourier coefficients. Let $\{e_n\}_{n \in \Z^d}$ be the Fourier basis on $\torus^d$. Explicitly, if we identify functions on $\torus^d$ with 1-periodic functions on $\R^d$, then $e_n(x) = e^{\icomplex 2\pi n \cdot x}$. Given $f \in L^1(\torus^d, \R)$, define the Fourier coefficient
\[ \hat{f}(n) := \int_{\torus^d} f(x) \ovl{e_n(x)} dx \in \C, ~~ n \in \Z^d.  \]
Note (since $f$ is $\R$-valued) that $\hat{f}(-n) = \ovl{\hat{f}(n)}$ for all $n \in \Z^d$. 
Also, if $f \in C^\infty(\torus^d, \R)$, then for any $k \geq 1$,
\beq\label{eq:fourier-coefficients-rapid-decay-general-dim} \sup_{n \in \Z^d} (1 + |n|^2)^{k/2} |\hat{f}(n)| < \infty, \eeq
where $|n|$ is the Euclidean norm of $n$. (See, e.g., \cite[Chapter 3, equation (1.7)]{T2011a}.) 

We note that this all generalizes to the case where $f$ takes values in some finite-dimensional normed linear space $(V, |\cdot|)$, in which case $\hat{f}(n) \in V^\C$, where $V^\C := \{v_1 + \icomplex v_2 : v_1, v_2 \in V\}$ is the ``complexified" version of $V$, with norm $|v_1 + \icomplex v_2| := (|v_1|^2 + |v_2|^2)^{1/2}$. Moreover, defining $\ovl{v_1 + \icomplex v_2} := v_1 - \icomplex v_2$, we have that $\hat{f}(-n) = \ovl{\hat{f}(n)}$ for all $n \in \Z^d$. 

Throughout this paper, given a normed linear space $(V, |\cdot|_V)$, we will abuse notation and write $|v|$ instead of $|v|_V$ for the norm of $v \in V$. Similarly, when $(V, \langle \cdot, \cdot \rangle_V)$ is an inner product space, we will write $\langle v_1, v_2 \rangle$ instead of $\langle v_1, v_2 \rangle_V$. The main examples of this are when $V$ is one of $\lalg, \lalg^\C, \lalg^d, (\lalg^d)^\C$ (note that the inner product that we defined on $\lalg$ induces inner products on the latter three spaces). 

\begin{definition}\label{def:fourier-truncation-operator}
For $N \geq 0$, define the Fourier truncation operator $\FT_N$ on distributions as follows. Given a distribution $\phi$ on $\torus^d$, define $\FT_N \phi \in C^\infty(\torus^d)$ by
\[ \FT_N \phi := \sum_{\substack{n \in \Z^d \\ |n|_\infty \leq N}} \hat{\phi}(n) e_n, \]
where $|n|_\infty$ is the $\ell^\infty$ norm of $n$. (Here $\hat{\phi}(n) := (\phi, e_{-n})$ for $n \in \Z^d$.)
\end{definition}

\begin{definition}[Quadratic forms]\label{def:quadratic-forms}
Let $A^0$ be a smooth {\oneform}. Let $\mbb{I} := [\lalgdim] \times [d] \times \torus^d$. Let $\lambda$ be the measure on $\mbb{I}$ defined by taking the product of counting measure on $[\lalgdim]$, counting measure on $[d]$, and Lebesgue measure on $\torus^d$. We say that $K : \mbb{I}^2 \ra \R$ is a smooth function if for any $a_1, a_2 \in [\lalgdim]$, $j_1, j_2 \in [d]$, the function on $(\torus^d)^2$ defined by $(x, y) \mapsto K((a_1, j_1, x), (a_2, j_2, y))$ is smooth. In this case, we write $K \in C^\infty(\mbb{I}^2, \R)$. Given a smooth function $K \in C^\infty(\mbb{I}^2, \R)$, define 
\beq\label{eq:quadratic-form-def} (A^0, K A^0) := \int_{\mbb{I}} \int_{\mbb{I}}  A^{0, a_1}_{j_1}(x) K(i_1, i_2) A^{0, a_2}_{j_2}(y) d\lambda(i_1) d\lambda(i_2) \in \R, \eeq
where $i_1 = (a_1, j_1, x)$, $i_2 = (a_2, j_2, y)$. 
\end{definition}

We next begin to state the assumptions that are involved in the statement of our main result, Theorem \ref{thm:main}. First, we make some definitions.

\begin{definition}
By a random $\lalg^d$-valued distribution $\rconn^0$, we mean a stochastic process $((\rconn^0, \phi), \phi \in \C^\infty(\torus^d, \R))$ of $\lalg^d$-valued random variables such that for all $\phi_1, \phi_2 \in C^\infty(\torus^d, \R)$, $c_1, c_2 \in \R$, we have that 
\beq\label{eq:random-distribution-linear} (\rconn^0, c_1 \phi_1 + c_2 \phi_2) \stackrel{a.s.}{=} c_1 (\rconn^0, \phi_1) + c_2 (\rconn^0, \phi_2). \eeq
\end{definition}

\begin{remark}\label{remark:random-distribution-different-viewpoints}
There are several different ways one can view $\rconn^0$. By linearity, we may also view $\rconn^0$ as a $(\lalg^d)^\C$-valued stochastic process $((\rconn^0, \phi), \phi \in C^\infty(\torus^d, \C))$ indexed by $\C$-valued test functions, which satisfy \eqref{eq:random-distribution-linear} for $\phi_1, \phi_2 \in C^\infty(\torus^d, \C)$ and $c_1, c_2 \in \C$, and which also satisfy $\ovl{(\rconn^0, \phi)} = (\rconn^0, \bar{\phi})$ for all $\phi \in C^\infty(\torus^d, \C)$.

Also, as is the case with {\oneforms}, we may equivalently view $\rconn^0$ as a $\R$-valued process $((\rconn^{0, a}_j, \phi), \phi \in C^\infty(\torus^d, \R), a \in [\lalgdim], j \in [d])$. Then again by linearity, we may  view $\rconn^0$ as a $\C$-valued process $((\rconn^{0, a}_j, \phi), \phi \in C^\infty(\torus^d, \C), a \in [\lalgdim], j \in [d])$.

We will use these different viewpoints (i.e., $\lalg^d$-valued, $(\lalg^d)^\C$-valued, $\R$-valued, $\C$-valued) interchangeably. As much as possible, we will try to take the vector-valued (i.e. $\lalg^d$-valued or $(\lalg^d)^\C$-valued) viewpoint, but we will find it convenient later on (in particular in Section \ref{section:technical-proofs-nonlinear-part}) to take the scalar-valued viewpoint for certain arguments.
\end{remark}

In what follows, let $\rconn^0$ be a random $\lalg^d$-valued distribution.

\begin{definition}[Fourier truncations]\label{def:A-0-N}
Let $N \geq 0$ be a finite integer. Define the Fourier truncation $\rconn^0_N := \FT_N \rconn^0$, which is a $\lalg^d$-valued stochastic process with smooth sample paths.
\end{definition}

\begin{remark}
Since $\rconn^0_N$ is a random smooth {\oneform}, we may (recalling \eqref{eq:one-form-R-valued-functions-relation}) equivalently view $\rconn^0_N$ as a $\R$-valued stochastic process with smooth sample paths $\rconn^0_N = (\rconn^{0, a}_{N, j}(x), a \in [\lalgdim], j \in [d], x \in \torus^d)$.
\end{remark}

\begin{definition}[Fourier coefficients]\label{def:fourier-truncations}
Define the Fourier coefficients of $\rconn^0$ by
\[ \hat{\rconn}^{0}(n) := (\rconn^{0}, e_{-n}), ~~ a \in [\lalgdim], j \in [d], n \in \Z^d.\]
Note that $\ovl{\hat{\rconn}^0(n)} = \hat{\rconn}^0(-n)$ for all $n \in \Z^d$.
\end{definition}

\begin{definition}
For $\phi \in C^\infty(\torus^d, \R)$, define $\sigma_\phi := (\E[|(\rconn^0, \phi)|^2])^{1/2}$. For $K \in C^\infty(\mbb{I}^2, 
\R)$, let $\sigma_{N, K} := (\E[(\rconn^0_N, K \rconn^0_N)^2])^{1/2}$ (recall Definition \ref{def:quadratic-forms}).
\end{definition}


\begin{definition}\label{def:G-alpha}
Let $\alpha \in (0, d)$, and define the bivariate distribution
\[ G^\alpha(x, y) := \sum_{\substack{n \in \Z^d \\n \neq 0}} \frac{1}{|n|^\alpha} e_n(x - y). \]
Define $G^\alpha_0(x) := G^\alpha(0, x)$.
\end{definition}

\begin{remark}
Note that $G^\alpha$ is the Green's function for the fractional negative Laplacian $(-\Delta)^{\alpha / 2}$ on $\torus^d$. In particular, $G^2$ is the Green's function for $-\Delta$ (which is the covariance function of the GFF, to be introduced a bit later).
\end{remark}


We quote the following lemma giving properties of $G^\alpha$. See \cite[Theorem 2.17]{SteinWeiss1971} for a proof.

\begin{lemma}\label{lemma:fractional-greens-function-properties}
Let $\alpha \in (0, d)$. The distribution $G_0^\alpha$ is smooth on $\torus^d - \{0\}$, with the following properties.  
\begin{enumerate}
    \item $G_0^\alpha$ is bounded from below.
    \item As $x \ra 0$, we have that $G_0^\alpha(x) \sim \metricspace(0, x)^{-(d-\alpha)}$. Here ``$\sim$" means that the ratio tends to a positive constant.
\end{enumerate}
\end{lemma}

\noindent We now make the following assumptions on $\rconn^0$. One should think of these assumptions as saying that $\rconn^0$ qualitatively behaves like a Gaussian field. Indeed, the assumptions were all abstracted from properties of the Gaussian free field, which will be introduced later.

\begin{enumerate}[label = (\Alph*), ref=(\Alph*)]
    \item\label{assumption:l2-regularity}($L^2$ regularity). For all $\phi \in C^\infty(\torus^d, \R)$, we have $\E[|(\rconn^0, \phi)|^2] < \infty$. Moreover, we have that as $N \toinf$,
    $\E[|(\rconn^{0}, \phi) - (\rconn^0_N, \phi)|^2] \ra 0$.
    \item\label{assumption:tail-bounds} (Tail bounds). There exist constants $\consttb > 0$ and $\exptb > 0$ such that the following hold. For any $\phi \in C^\infty(\torus^d, \R)$, we have that
    \[ \p(|(\rconn^{0}, \phi)| > u) \leq \consttb \exp(- (u / \sigma_\phi)^{\exptb} / \consttb), ~~ u \geq 0. \]
    \noindent Additionally, for any $N \geq 0$, and for any smooth function $K \in C^\infty(\mbb{I}^2, \R)$ such that $K((a, j_1, x), (a, j_2, y)) = 0$ for all $a \in [\lalgdim]$, $j_1, j_2 \in [d]$, $x, y \in \torus^d$, we have that
    \[ \p\big(|(\rconn^0_N, K \rconn^0_N)| > u\big) \leq \consttb \exp(-(u / \sigma_{N, K})^{\exptb} / \consttb), ~~ u \geq 0. \] 
    \item\label{assumption:translation-invariance}(Translation invariance of covariance function) There is an integrable function $\rho : (\torus^d)^2 \ra L(\lalg^d, \lalg^d)$ (here $L(\lalg^d, \lalg^d)$ is the space of linear maps $\lalg^d \ra \lalg^d$) such that for any test functions $\phi_1, \phi_2 \in C^\infty(\torus^d, \R)$, and any linear map $K : \lalg^d \ra \lalg^d$,
    \[ \E[\langle (\rconn^0, \phi_1), K (\rconn^0, \phi_2)\rangle] = \int_{\torus^d} \int_{\torus^d} \phi_1(x) \phi_2(y) \Tr(K \rho(x, y)^t) dx dy. \]
    Moreover, we assume that $\rho$ is translation invariant, i.e. $\rho(x, y) = \rho(x - y, 0) = \rho(0, y - x)$ for $x, y \in \torus^d$. Here, $\Tr(K \rho(x, y)^t)$ is (for instance) computed by representing $M, \rho(x, y)$ as matrices with respect to the basis of $\lalg^d$ induced by the orthonormal basis $(S^a, a \in [\lalgdim])$ of $\lalg$. Similarly, ``integrable" in this context can be taken to mean that all matrix entry functions of $\rho$ are integrable.
    \item\label{assumption:bounded-by-fractional-greens-function}(Covariance is only as singular as $G^\alpha$). For some $\alpha \in (0, d)$, there is some constant $\constgf$ such that for any $x, y \in \torus^d$, $x \neq y$,
    \[ |\Tr(\rho(x, y))| \leq \constgf (G^\alpha(x, y) + \constgf).\]
    We assume without loss of generality that $G^\alpha + \constgf \geq 1$ (this is possible since $G_0^\alpha$ is bounded from below, by Lemma \ref{lemma:fractional-greens-function-properties}). 
    \item\label{assumption:four-product-assumption} (Four product assumption). 
    There is some constant $\constfourp \geq 0$ such that the following holds. Let $a_1, a_2 \in [\lalgdim]$, $j_1, j_2 \in [d]$, $\phi_1, \phi_2, \phi_3, \phi_4 \in C^\infty(\torus^d, \C)$. Assume that $a_1 \neq a_2$. Let $Z_1 = (\rconn^{0, a_1}_{j_1}, \phi_1)$, $Z_2 = (\rconn^{0, a_2}_{j_2}, \phi_2)$, $Z_3 = (\rconn^{0, a_1}_{j_1}, \phi_3)$, $Z_4 = (\rconn^{0, a_2}_{j_2}, \phi_4)$. Then
    \[ |\E[Z_1 Z_2 \ovl{Z_3 Z_4}]| \leq \constfourp \big( |\E[Z_1 \ovl{Z_3}] \E[Z_2 \ovl{Z_4}]| + |\E[Z_1 \ovl{Z_4}] \E [Z_2 \ovl{Z_3}]| \big) .\]
\end{enumerate}

\begin{remark}
Assumption \ref{assumption:translation-invariance} is motivated by the following fact. Let $X, Y$ be random vectors in $\R^n$, and let $\Sigma := \E[X Y^t]$ (so $\Sigma$ is an $n \times n$ matrix). Then for any $n \times n$ matrix $M$, we have that $\E[X^t M Y] = \Tr(M \Sigma^t)$.

Also, to be more concrete, instead of working with the function $\rho$ in Assumption \ref{assumption:translation-invariance}, one may work with scalar functions $\rho^{a_1 a_2}_{j_1 j_2} : (\torus^d)^2 \ra \R$ for $a_1, a_2 \in [\lalgdim]$, $j_1, j_2 \in [d]$, which are defined by requiring that
\[ \E[(\rconn^{0, a_1}_{j_1}, \phi_1) (\rconn^{0, a_2}_{j_2}, \phi_2)] = \int_{\torus^d} \int_{\torus^d} \phi_1(x) \rho^{a_1 a_2}_{j_1 j_2}(x, y) \phi_2(y) dx dy.\]
The function $\rho^{a_1 a_2}_{j_1 j_2}$ is then interpreted as the $((a_1, j_1), (a_2, j_2))$ matrix entry of $\rho$.
\end{remark}

\begin{remark}
In fact, we don't really need to assume that $\alpha < d$; $G_0^\alpha$ can be defined for $\alpha \geq d$ as well. However, we will just assume that $\alpha \in (0, d)$, because this will simplify our proofs later on (the cases $\alpha = d$ and $\alpha > d$ would have to be handled somewhat separately).
In any case, the regime $\alpha \in (0, d)$ is the more nontrivial setting, since $G_0^\alpha$ becomes less singular as $\alpha$ increases (see, e.g., \cite[Section 6.1]{SteinWeiss1971}).
\end{remark}

\begin{remark}
In assumption \ref{assumption:four-product-assumption}, it is important that we don't have the term $\E[Z_1 Z_2] \E[\bar{Z}_3 \bar{Z}_4]$, because this term will lead to divergences in Section \ref{section:technical-proofs-nonlinear-part}.
\end{remark}

We now state the main result of this paper.

\begin{theorem}\label{thm:main}
Let $\rconn^0$ be a random $\lalg^d$-valued distribution that satisfies Assumptions \ref{assumption:l2-regularity}--\ref{assumption:four-product-assumption}. Moreover, suppose that Assumption \ref{assumption:bounded-by-fractional-greens-function} is satisfied with $\alpha \in (\max\{d - 4/3, d/2\}, d)$. Then there exists a $\lalg^d$-valued stochastic process $\rconn = (\rconn(t, x), t \in (0, 1), x \in \torus^d)$, and a random variable $T \in (0, 1]$, such that the following hold. The function $(t, x) \mapsto \rconn(t, x)$ is in $C^\infty((0, T) \times \torus^d, \lalg^d)$, and moreover, it is a solution to \eqref{eq:ZDDS} on $(0, T)$. Also, $\E[T^{-p}] <  \infty$ for all $p \geq 1$. 

The process $\rconn$ relates to $\rconn^0$ in the following way. There exists a sequence $\{T_N\}_{N \geq 0}$ of $(0, 1]$-valued random variables such that the following hold. First, for any $p \geq 1$, we have that $\sup_{N \geq 0} \E[T_N^{-p}] < \infty$, and that $\E[|T_N^{-1} - T^{-1}|^p] \ra 0$. 
Also, let $\{\rconn^0_N\}_{N \geq 0}$ be the sequence of Fourier truncations of $\rconn^0$ (defined in Definition \ref{def:fourier-truncations}). Then there is a sequence $\{\rconn_N\}_{N \geq 0}$ of $\lalg^d$-valued stochastic processes $\rconn_N = (\rconn_N(t, x), t \in [0, 1), x \in \torus^d)$ such that for each $N \geq 0$, the function $(t, x) \mapsto \rconn_N(t, x)$ is in $C^\infty([0, T_n) \times \torus^d, \lalg^d)$ and is the solution to \eqref{eq:ZDDS} on $[0, T_n)$ with initial data $\rconn_N(0) = \rconn^0_N$. 

Finally, for any $k \in \{0, 1\}$, $p \geq 1$, $\delta \in (0, 1), \varep > 0$, we have that
\[\begin{split} 
\lim_{N \toinf} \E\bigg[\sup_{\substack{t \in (0, (1 - \delta)T) }} t^{p((k/2) + (1/4)(d-\alpha) + \varep)}\|\rconn_N(t) - \rconn(t)\|_{C^k}^p \bigg] &= 0. \end{split}\]
\end{theorem}

\begin{remark}
A closely related result was obtained by Chandra et al. as part of their recent work \cite{CCHS2022} -- see Section 1.2 and Remarks 2.8 and 3.15 in their paper for some similarities and differences. While there is some variation in the ways the results and the proofs are phrased, ultimately we are both (as far as we can tell) exploiting the same phenomenon, which is probabilistic smoothing, which we describe next. 
\end{remark}

\begin{remark}
The assumption that $\alpha > \max\{d - 4/3, d/2\}$ ensures that there is no need for renormalization when defining the solution to \eqref{eq:ZDDS} with initial data $\rconn^0$.
\end{remark}

We now give a quick overview of the proof of the local existence part of Theorem \ref{thm:main}. The proof of the other part of the theorem has a similar main idea. As usual with local existence for parabolic PDEs, we would like to try to realize the solution $\rconn$ as the fixed point of some contraction map $W$. Then, we could for instance obtain $\rconn$ by taking the limit of $W^{(n)} (\rconn^1)$, where $W^{(n)}$ is the $n$-fold composition of $W$, and $\rconn^1$ is the linear part of $\rconn$, i.e., $\rconn^1$ is the solution to the heat equation with initial data $\rconn^0$. However, the problem is that the initial data $\rconn^0$ is too rough, so that deterministic arguments break down already in the first step of the Picard iteration --- that is, when trying to obtain estimates on $W(\rconn^1)$ by deterministic (worst-case) methods, we get divergent integrals.

The saving grace is that $W(\rconn^1)$ behaves better than the worst-case. So instead of bounding $W(\rconn^1)$ deterministically, we bound it probabilistically, which allows us to take advantage of probabilistic cancellations which occur. To give an analogy with an elementary example, note that if $\{X_n\}_{n \geq 1}$ is a sequence of i.i.d.~random variables with  mean $0$ variance $1$, then the series $\sum_{n=1}^\infty X_n / n$ converges a.s.~(by Kolmogorov's two series theorem). However, if we were to try to bound this deterministically, the sum $\sum_{n=1}^\infty n^{-1} = \infty$ would inevitably appear. Analogously, we show that $W(\rconn^1)$ can be defined in a probabilistic sense, and in fact the difference $W(\rconn^1) - \rconn^1$ is more regular than the linear part $\rconn^1$. Once this regularity gain is established, we can then obtain the local existence of $\rconn$ by a deterministic fixed point argument (i.e., Picard iteration). 

This general idea to exploit the effects of probabilistic smoothing was (as far as we can tell) first used by Bourgain \cite{Bourgain1996, Bourgain1999} to analyze the nonlinear Schr\"{o}dinger equation with GFF initial data. A similar idea was later used by Da Prato and Debussche \cite{DaDeb2002, DaDeb2003} in the stochastic PDE setting. There is by now a wide body of work building on this idea in many different settings -- see \cite[Section 1.2.2]{DNY2019} for a much more complete list of references.

We next introduce the Gaussian free field (GFF), which will be the main example of random distributional initial data in this paper. Standard references are \cite{BP2021, She2007, WP2020}. A $d$-dimensional mean zero GFF on $\torus^d$ is a mean zero  Gaussian process $h = ((h, \phi), \phi \in C^\infty(\torus^d, \R))$ such that for all test functions $\phi_1, \phi_2 \in C^\infty(\torus^d, \R)$, the covariance is given by
\beq\label{eq:gff-covariance-condition} \E[ (h, \phi_1)(h, \phi_2)] = \sum_{\substack{n \in \Z^d \\ n \neq 0}} \frac{1}{|n|^2} \hat{\phi}_1(n) \ovl{\hat{\phi}_2(n)}.\eeq
We proceed to give an explicit construction of $h$. Fix a subset $I_\infty \sse \Z^d$ such that $0 \notin I_\infty$, and such that for each $n \in \Z^d \setminus \{0\}$, exactly one of $n, -n$ is in $I_\infty$. Let $(Z_n, n \in I_\infty)$ be an i.i.d.~collection of standard complex Gaussian random variables. For $n \notin I_\infty$, define $Z_n := \ovl{Z_{-n}}$. We can then interpret $h$ as the random Fourier series
\beq\label{eq:GFF-fourier-series} h = \sum_{\substack{n \in \Z^d \\ n \neq 0}} \frac{Z_n}{|n|} e_n. \eeq
There are several interpretations of the above; one is that the sum a.s.~converges in a negative Sobolev space (see, e.g., \cite[Theorem 1.24]{BP2021}). The interpretation that we will use is that the above identity means that for any $\phi \in C^\infty(\torus^d, \R)$, we have that
\beq\label{eq:GFF-phi-def} (h, \phi) \stackrel{a.s.}{=} \sum_{\substack{n \in \Z^d \\ n \neq 0}} \frac{Z_n}{|n|} \ovl{\hat{\phi}(n)}.\eeq
(Using standard concentration arguments, we have that a.s., $|Z_n| = O(\sqrt{d \log |n|})$. Combining this with \eqref{eq:fourier-coefficients-rapid-decay-general-dim}, we have that a.s., for all $\phi \in C^\infty(\torus^d, \R)$, the sum on the right hand side above is absolutely summable.) Using that $\E[Z_n \ovl{Z_m}] = \ind(n = m)$ (i.e., $1$ if $n=m$ and $0$ otherwise), it can be verified that the condition \eqref{eq:gff-covariance-condition} is satisfied, given the above.

For $N \geq 0$, let the Fourier truncation $h_N = (h_N(x), x \in \torus^d)$ be the mean zero Gaussian process with smooth sample paths defined as $h_N := \FT_N h$.
By standard properties of $\FT_N$ and the GFF, we have that for any $\phi \in C^\infty(\torus^d, \R)$, $(h_N, \phi) \ra (h, \phi)$ both a.s. and in $L^2$ (actually, the a.s.~convergence holds simultaneously for all $\phi$).
Therefore $h_N$ converges to $h$ in a natural sense. (Another viewpoint is that if we view $h$ as a random element of a negative Sobolev space, then $h_N$ a.s.~converges to $h$ in that space.)

Since \eqref{eq:ZDDS} is a PDE on {\oneforms}, the initial data we take must also be a {\oneform}. Recalling that we may view a {\oneform} $A$ as a collection of functions $(A^a_j, a \in [\lalgdim], j \in [d])$ satisfying \eqref{eq:one-form-R-valued-functions-relation}, this motivates the following definition of the $d$-dimensional $\lalg^d$-valued GFF.  We say that $\rconn^0$ is a $d$-dimensional $\lalg^d$-valued GFF if it is a collection of stochastic processes
\[ \rconn^0 = (\rconn^{0, a}_j,  a \in [\lalgdim], j \in [d]), \]
where $\rconn^{0,  a}_j, a \in [\lalgdim], j \in [d]$ are independent $d$-dimensional GFFs. This is $\lalg^d$-valued, because given $\phi \in C^\infty(\torus^d, \R)$, we may define a $\lalg^d$-valued random variable $(\rconn^0, \phi) = ((\rconn^0, \phi)_j, j \in [d])$ through the relation 
\[ (\rconn^0, \phi)_j  := \sum_{a \in [\lalgdim]} (\rconn^{0, a}_j, \phi) \onbasislalg^a, ~~ j \in [d].\]


As previously mentioned, the assumptions of Theorem \ref{thm:main} were abstracted from properties of the GFF. Thus, naturally, we will be able to obtain the following corollary of Theorem \ref{thm:main}.

\begin{cor}\label{cor:gff}
Let $d = 3$, and let $\rconn^0$ be a $\lalg^3$-valued GFF on $\torus^3$. Then the statement of Theorem \ref{thm:main} applies to $\rconn^0$.
\end{cor}

\begin{remark}
The above corollary will also hold for $d = 2$ by a simpler deterministic argument. However, once $d \geq 4$, the same result does not necessarily apply, because Assumption \ref{assumption:bounded-by-fractional-greens-function} will only be satisfied for $\alpha$ small (e.g., $\alpha = 2$ when $d = 4$), which is to say that the GFF becomes too rough once $d \geq 4$.
\end{remark}

\subsection{Additional notation}\label{section:notation}
We introduce some additional notation. Throughout this paper, $\const$ will denote a generic constant that may depend only on $\liegroup$. It may change from line to line, and even within a line. To express dependence on some additional parameter, say $\alpha$, we will write $\const_\alpha$. In these situations, we always understand $\const_\alpha$ to also depend on $\liegroup$.  Similarly, if we say that $\const_\alpha$ depends only on $\alpha$, we really mean that $\const_\alpha$ depends only on $\alpha$ and $\liegroup$.

The metric on $\torus^d$ (that is, the metric induced by the standard Euclidean metric of $\R^d$) will be denoted by $\metricspace$. Explicitly, if $\Pi : \R^d \ra \torus^d$ is the canonical projection map, then $\metricspace(x, y) := \inf \{|x_0 - y_0| : \Pi(x_0) = x, \Pi(y_0) = y\}$. Here $|x_0 - y_0|$ is the Euclidean distance between $x_0, y_0 \in \R^d$. 

Fix a real finite-dimensional normed linear space $(\vectorspace, |\cdot|)$ (for instance, we may take $V = \lalg$ or $\lalg^d$). 
For a tuple $\alpha = (\alpha_1, \ldots, \alpha_d) \in \N^d$ of nonnegative integers and a function $f : \torus^d \ra V$, let  
\[ \ptl^\alpha f := \ptl_1^{\alpha_1} \cdots \ptl_d^{\alpha_d} f, \]
provided that the mixed partial derivative on the right exists. 
We write $|\alpha|_1 := \sum_{i=1}^d \alpha_i$. For an integer $m \geq 0$, let $C^m(\torus^d, V)$ be the space of $m$-times continuously differentiable functions, that is, the space of functions $f : \torus^d \ra V$ such that for all multi-indices $\alpha$ with $|\alpha|_1 \leq m$, the derivative $\ptl^\alpha f$ exists and is continuous. For $f \in C^0(\torus^d, V)$, define
\[ \|f\|_{C^0(\torus^d, V)} := \sup_{x \in \torus^d} |f(x)|,\]
and for $f \in C^m(\torus^d, V)$, let 
\[ \|f\|_{C^m(\torus^d, V)} := \max_{|\alpha|_1 \leq m} \|\ptl^\alpha f\|_{C^0}.\]
We will usually write $\|f\|_{C^m}$ for brevity. 
Let $C^\infty(\torus^d, V)$ be the space of smooth functions from $\torus^d \ra V$, that is, $C^\infty(\torus^d, V) = \bigcap_{m=0}^\infty C^m(\torus^d, V)$. Note in particular $C^\infty(\torus^d, \lalg^d)$ is the space of smooth {\oneforms}. 

For a real number $0 < r < 1$, define the H\"{o}lder space $C^r(\torus^d, V)$ to be the space of continuous functions $f : \torus^d \ra V$ with finite H\"{o}lder norm: 
\[ \|f\|_{C^r(\torus^d, V)} := \max\bigg\{\|f\|_{C^0}, \sup_{\substack{x, y \in \torus^d \\ x \neq y}} \frac{|f(x) - f(y)|}{d_{\torus^d}(x, y)^r}\bigg\} < \infty. \]
We will usually write $\|f\|_{C^r}$ for brevity. 
Observe that for $f, g \in C^r(\torus^d, \R)$,
\beq\label{eq:C-r-norm-product}\|f g\|_{C^r} \leq 2\|f\|_{C^r} \|g\|_{C^r}.\eeq
For a real number $r \geq 0$, let $m \geq 0$, $0 \leq s < 1$ be such that $r = m + s$, where $m$ is the integer part of $r$. Define the H\"{o}lder space $C^r(\torus^d, V)$ to be the space of functions $f : \torus^d \ra V$ such that $f \in C^m(\torus^d, V)$ and such that for each multi-index $\alpha$ with $|\alpha|_1 = m$, the derivative $\ptl^\alpha f$ is in $C^s(\torus^d, V)$. Define the H\"{o}lder norm
\[ \|f\|_{C^r(\torus^d, V)} := \max\bigg\{\|f\|_{C^m},  \max_{|\alpha|_1 = m} \|\ptl^\alpha f\|_{C^s}\bigg\}. \]
Again, we will usually write $\|f\|_{C^r}$ for brevity. Note that when $r$ is integer, the above definition coincides with our previous definition of $\|\cdot\|_{C^r}$. We have that $(C^r(\torus^d, V), \|\cdot\|_{C^r})$ is a Banach space for any $r \in [0, \infty)$ (see, e.g.,  \cite[Appendix A]{T2011a}).

We next define $L^p$ spaces of functions. Let $1 \leq p \leq \infty$. Given a (measurable) function $f : \torus^d \ra \vectorspace$, let
\[
\|f\|_{L^p(\torus^d, \vectorspace)} := \begin{cases}  \big(\int_{\torus^d} |f(x)|^p dx\big)^{1/p} & 1 \leq p < \infty, \\ 
\esssup_{x \in \torus^d} |f(x)| & p = \infty .
\end{cases}
\]
For $1 \leq p \leq \infty$, let $L^p(\torus^d, V)$ be the real Banach space (of a.e.~equivalence classes) of measurable functions $\torus^d \ra V$ with finite $L^p$ norm. We will write $\|\cdot\|_p$ instead of $\|\cdot\|_{L^p(\torus^d, V)}$ for brevity.



Let $(e^{t \Delta})_{t > 0}$ be the semigroup generated by the Laplacian $\Delta$. Explicitly, given $f \in L^1(\torus^d, V)$ and $t > 0$, we have that (see e.g. \cite[Chapter 6.1]{T2011a}) 
\beq\label{eq:heat-semigroup-def} e^{t \Delta} f = \sum_{n \in \Z^d} e^{-4\pi^2|n|^2 t} \hat{f}(n) e_n. \eeq
Using that $|\hat{f}(n)| \leq \|f\|_{L^1}$ for all $n \in \Z^d$, it can be shown that the map $(t, x) \mapsto (e^{t\Delta} f)(x)$ is in $C^\infty((0, \infty) \times \torus^d, V)$, and it is a solution to the heat equation on $(0, \infty) \times \torus^d$.
Additionally, $e^{t \Delta} f$ has an explicit representation in terms of convolution with the heat kernel $\Phi$, i.e., for all $t > 0$, $x \in \torus^d$, we have that
\beq\label{eq:heat-kernel-semigroup-convolution} (e^{t \Delta} f)(x) = \int_{\torus^d} f(y) \Phi(t, x-y) dy, \eeq
where
\beq\label{eq:heat-kernel-def} \Phi(t, x) := \sum_{n \in \Z^d} e^{-4\pi^2 |n|^2 t} e_n(x), ~~ t > 0, x \in \torus^d.\eeq
Since $\Phi(t, \cdot)$ is also given by a probability density (see \cite[Chapter 6.1]{T2011a}), it is non-negative, and thus we have the following monotonicity property for integrable $\R$-valued functions $f, g$:
\beq\label{eq:heat-semigroup-monotonicity} |f| \leq g \implies |e^{t \Delta} f| \leq e^{t \Delta} g. \eeq
(In the above, $|f| \leq g$ means $|f(x)| \leq g(x)$ for all $x \in \torus^d$ --- actually, a.e.~$x \in \torus^d$ suffices --- and similarly for $|e^{t\Delta}f| \leq e^{t \Delta} g$.)

Recall the orthonormal basis $(\onbasislalg^a, ~a \in [\lalgdim])$ of $\lalg$ from Definition \ref{def:lalg-and-X}.  Let $(f^{abc}, ~a, b, c \in [\lalgdim])$ be the corresponding structure constants, i.e.,
\beq\label{eq:structure-constants-def} [\onbasislalg^a, \onbasislalg^b] = \sum_{c \in [\lalgdim]} f^{abc} \onbasislalg^c. \eeq
By starting from the definition above (and using that the inner product is given by $\langle S_1, S_2\rangle = \mathrm{Tr}(S_1^* S_2) = -\mathrm{Tr}(S_1 S_2)$), we obtain for $a, b, c \in [\lalgdim]$,
\[\begin{split}
\langle [\onbasislalg^a, \onbasislalg^b], \onbasislalg^c\rangle &= -\mathrm{Tr}([\onbasislalg^a, \onbasislalg^b] \onbasislalg^c) \\
&=  -\mathrm{Tr}((\onbasislalg^a \onbasislalg^b - \onbasislalg^b \onbasislalg^a) \onbasislalg^c) \\
&= - \mathrm{Tr}((\onbasislalg^b \onbasislalg^c - \onbasislalg^c \onbasislalg^b) \onbasislalg^a) \\
&=-\mathrm{Tr}([\onbasislalg^b, \onbasislalg^c] \onbasislalg^a) = \langle [\onbasislalg^b, \onbasislalg^c], \onbasislalg^a \rangle.
\end{split} \]
This shows that $f^{abc} = f^{bca}$. Proceeding similarly, we may obtain
\beq\label{eq:structure-constants-permute-indices} f^{abc} = f^{cab} = -f^{acb} = -f^{bac} = -f^{cba} = f^{bca}.\eeq

\begin{remark}
Even though we are introducing structure constants here, the results of this paper do not really rely on the specific bracket structure of \eqref{eq:ZDDS}. Indeed, we expect that the arguments could be adapted to the case where $\lalg$ is replaced by a finite-dimensional normed algebra, and \eqref{eq:ZDDS} is replaced by an equation of the form
\[ \ptl_t A = \Delta A + A (\nabla A) + A^3. \]
\end{remark}

\subsection{Organization of the paper}

We now give a summary of the rest of the paper. In Section~\ref{section:deterministic-results}, we state Theorem \ref{thm:rough-distributional-zdds-local-existence}, which is a deterministic result that gives local existence of solutions to \eqref{eq:ZDDS} with distributional intial data, assuming certain conditions are met. We also state various other useful deterministic lemmas in Section \ref{section:useful-lemmas}. Given Theorem \ref{thm:rough-distributional-zdds-local-existence}, the remainder of the paper is then concerned with showing that the conditions of the theorem are indeed met, for random distributional initial data with certain properties, as listed just before Theorem \ref{thm:main}.
Sections \ref{section:linear-part} and \ref{section:nonlinear-part} collect the main intermediate steps towards the proof of Theorem \ref{thm:main}. Given these intermediate steps, Theorem \ref{thm:main} is proven in Section \ref{section:proofs-main}. Corollary \ref{cor:gff} is obtained as an application of Theorem \ref{thm:main} in the same section. 
In Section \ref{section:useful-variant}, we state and prove Proposition \ref{prop:main-useful-variant}, which is a variant of Theorem \ref{thm:main} that will be used in \cite{CaoCh2021}. Sections \ref{section:technical-proofs-linear-part} and \ref{section:technical-proofs-nonlinear-part} contain the technical arguments needed to prove the intermediate results of Sections \ref{section:linear-part} and \ref{section:nonlinear-part}. 


Here is a flowchart illustrating the structure outlined above: \\

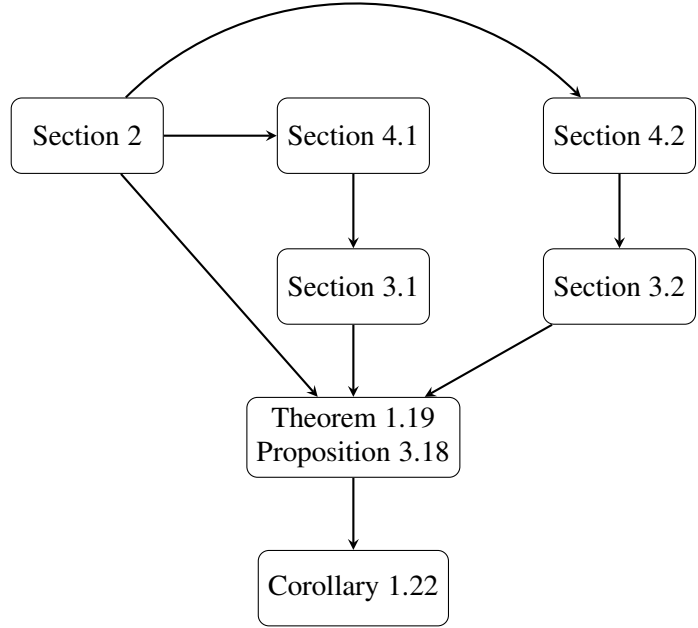
\begin{center}
\begin{tikzpicture}[node distance=2cm]


\node (linear-part) [startstop, xshift=1.5cm] {Section \ref{section:linear-part}};

\node (nonlinear-part) [startstop, right of=linear-part, xshift=1.5cm] {Section \ref{section:nonlinear-part}};

\node[align=center] (proof-linear-part) [startstop, above of=linear-part] {Section \ref{section:technical-proofs-linear-part}} ;

\node[align=center] (proofs-det) [startstop, left of=proof-linear-part, xshift=-1.5cm] {Section \ref{section:deterministic-results}};


\node[align=center] (proof-nonlinear-part) [startstop, above of=nonlinear-part] {Section \ref{section:technical-proofs-nonlinear-part}} ;

\node[align=center] (main) [startstop, below of=linear-part] {Theorem \ref{thm:main} \\ Proposition \ref{prop:main-useful-variant}};

\node (main-gff) [startstop, below of=main] {Corollary \ref{cor:gff}};

\draw [arrow] (proofs-det) -- (proof-linear-part) ;
\draw [arrow] (proofs-det) to [out=45, in=135] (proof-nonlinear-part) ;
\draw [arrow] (linear-part) -- (main) ;
\draw [arrow] (nonlinear-part) -- (main) ;
\draw [arrow] (main) -- (main-gff) ;
\draw [arrow] (proof-linear-part) -- (linear-part) ;
\draw [arrow] (proof-nonlinear-part) -- (nonlinear-part) ;
\draw [arrow] (proofs-det) -- (main) ;

\end{tikzpicture}
\end{center}

\subsection{Acknowledgements}
We are grateful to  Nelia Charalambous, Persi Diaconis, Len Gross, and Phil Sosoe for various helpful comments and discussions. We are also grateful to the referees for many helpful and detailed comments.

\section{Deterministic results}\label{section:deterministic-results}

In this section, we collect the deterministic results that are needed later on in the paper. We emphasize here that the results of this section may be read independently of the rest of the paper (although of course the main reason for these results is to use them to deduce Theorem \ref{thm:main}). The main result of this section (Theorem \ref{thm:rough-distributional-zdds-local-existence}) shows local existence of solutions to \eqref{eq:ZDDS} with distributional initial data.
The proofs of most results in this section are small variations of proofs of classical results in the theory of local existence for nonlinear parabolic PDEs. Thus, unless otherwise given, all proofs in this section are in Appendix \ref{appendix:deterministic-pde-proofs}.
We first define the notation that will be needed in Theorem \ref{thm:rough-distributional-zdds-local-existence} and in other parts of this paper.

\begin{definition}\label{def:P-T-r}
For $r \geq 0$, $T > 0$, define the path space $\mc{P}_T^r$ to be the space of continuous functions $A : [0, T] \ra C^r(\torus^d, \lalg^d)$. Define the norm $\|\cdot\|_{\mc{P}_T^r}$ on $\mc{P}_T^r$ by
\[ \|A\|_{\mc{P}_T^r} := \sup_{0 \leq t \leq T} \|A\|_{C^r}, ~~ A \in \mc{P}_T^r. \]
Note that $(\mc{P}_T^r, \|\cdot\|_{\mc{P}_T^r})$ is a Banach space, because $C^r(\torus^d, \lalg^d)$ is a Banach space (as noted in Section \ref{section:notation})
\end{definition}

\begin{definition}\label{def:Q-T-gamma}
Let $\gamma \geq 0$, $T > 0$. Define the path space $\mc{Q}_T^\gamma$ to be the space of continuous functions $A : (0, T] \ra C^1(\torus^d, \lalg^d)$ such that
\[ \|A\|_{\mc{Q}_T^\gamma} :=\sup_{t \in (0, T]}  t^\gamma  \|A(t)\|_{C^0} + \sup_{t \in (0, T]} t^{(1/2) + \gamma} \|A(t)\|_{C^1} < \infty. \]
For $R \geq 0$, define $\mc{Q}_{T, R}^\gamma := \{A \in \mc{Q}_T^\gamma : \|A\|_{\mc{Q}_T^\gamma} \leq R\}$.
\end{definition}

\begin{remark}
We thank one of the referees for pointing out here that the regularity parameter $\gamma$ is flipped, in that larger $\gamma$ allows for more irregularity.
\end{remark}


We will later show that $\mc{Q}_T^\gamma$ is a Banach space under this norm (Lemma~\ref{lemma:Q-T-gamma-banach}).

\begin{definition}\label{def:X}
Given a {\oneform} $A \in C^1(\torus^d, \lalg^d)$, define $X(A) \in C^0(\torus^d, \lalg^d)$ by $X(A) = (X_i(A), i \in [d])$, where
\[ X_i(A) := \sum_{j \in [d]} [A_j, 2\ptl_j A_i - \ptl_i A_j + [A_j, A_i]], ~~ i \in [d]. \]
Define $X^{(2)}(A), X^{(3)}(A) \in C^0(\torus^d, \lalg^d)$ as follows. For $i \in [d]$, let
\[ X^{(2)}_i(A) := \sum_{j \in [d]} [A_j, 2\ptl_j A_i - \ptl_i A_j], ~~ X^{(3)}_i(A) := \sum_{j \in [d]} [A_j, [A_j, A_i]]. \]
Note by construction that $X(A) = X^{(2)}(A) + X^{(3)}(A)$.
\end{definition}


\begin{definition}\label{def:rho}
Let $T > 0$. Let $j \in \{2, 3\}$. Let $A : [0, T] \ra C^1(\torus^d, \lalg^d)$ be a continuous function. Suppose that
\beq\label{eq:rho-finite-integral}  \int_0^T \|e^{(t-s)\Delta} X^{(j)}(A(s))\|_{C^1} ds < \infty. \eeq
Define $\rho^{(j)}(A) : [0, T] \ra C^1(\torus^d, \lalg^d)$ by
\[ \rho^{(j)}(A)(t) := \int_0^t e^{(t-s)\Delta} X^{(j)}(A(s)) ds, ~~ t \in [0, T].\]
We say that $\rho^{(j)}(A)$ is well-defined for $A$ if \eqref{eq:rho-finite-integral} holds. Now if $\rho^{(j)}(A)$ is well-defined for $A$ for $j = 2, 3$, define $\rho(A) : [0, T] \ra C^1(\torus^d, \lalg^d)$ by $\rho(A) := \rho^{(2)}(A) + \rho^{(3)}(A)$. In this case, we say that $\rho(A)$ is well-defined for $A$. Note since $X = X^{(2)} + X^{(3)}$, we have that \[ \rho(A)(t) = \int_0^t e^{(t-s)\Delta} X(A(s)) ds, ~~ t \in [0, T].\]
Finally, if the domain of $A$ is $(0,T]$ instead of $[0,T]$ (all other conditions remaining the same), we define $\rho^{(j)}(A) : (0, T] \ra C^1(\torus^d, \lalg^d)$ and $\rho(A) : (0, T] \ra C^1(\torus^d, \lalg^d)$ by replacing $[0, T]$ with $(0, T]$ everywhere above, provided that the integrals are absolutely convergent.
\end{definition}



The next lemma shows that $\rho(A)$ is well-defined for $A$ if $A \in \mc{P}_T^1$, and moreover that $\rho(A) \in \mc{P}_T^1$. 

\begin{lemma}\label{lemma:rho-in-P-T-0}
Let $T \in (0, 1]$. Let $A \in \mc{P}_T^1$. Then $\rho^{(j)}(A)$ is well-defined for $A$ for $j \in \{2, 3\}$, and moreover $\rho^{(2)}(A), \rho^{(3)}(A) \in \mc{P}_T^1$. Thus also $\rho(A) \in \mc{P}_T^1$.
\end{lemma}

\begin{remark}\label{remark:integral-equation}
The reason for the preceding definitions is that (as is usual for parabolic PDEs) to prove local existence, we will cast \eqref{eq:ZDDS} into its corresponding integral form
\[ A(t) = e^{t \Delta} A(0) + \int_0^t e^{(t-s)\Delta} X(A(s)) ds = e^{t \Delta} A(0) + \rho(A)(t). \]
(Note that $X(A(s))$ is the nonlinear part of \eqref{eq:ZDDS}.) We will first construct a solution $A$ to the above integral equation, and then show that the integral equation implies \eqref{eq:ZDDS}. 
\end{remark}

\begin{definition}\label{def:nonlinear-part-of-linear-part}
Let $A^1 : (0, 1] \ra C^1(\torus^d, \lalg^d)$ be such that $A^1(t) = e^{(t-s)\Delta} A^1(s)$ for all $s, t \in (0, 1]$, $s < t$. Let $B^1 : (0, 1] \ra C^1(\torus^d, \lalg^d)$ be a continuous function. We say that $B^1$ is a first nonlinear part for $A^1$ if the following holds. For $t_0 \in (0, 1)$, let $\tilde{A}^1(t) := A^1(t_0 + t)$, $\tilde{B}^1(t) := B^1(t_0 + t)$, $t \in [0, 1 - t_0]$. Then for all $t_0 \in (0, 1)$ and all $t \in [0, 1 - t_0]$, we have that $\tilde{B}^1(t) = e^{t \Delta} \tilde{B}^1(0) + \rho(\tilde{A}^1)(t)$.
\end{definition}


\begin{remark}\label{remark:first-nonlinear-part-c1-example}
To see where Definition \ref{def:nonlinear-part-of-linear-part} comes from, suppose $A^0$ is a smooth {\oneform}, and let $A^1(t) = e^{t \Delta} A^0$, $t \in (0, 1]$. Let $B^1 = \rho(A^1)$. Then for $t_0 \in (0, 1)$, $t \in [0, 1 - t_0]$, we have that
\[ B^1(t_0 + t) = e^{t\Delta} \int_0^{t_0} e^{(t_0 - s)\Delta} X(A^1(s)) ds + \int_{t_0}^{t_0 + t} e^{(t_0 + t - s)\Delta} X(A^1(s)) ds. \]
Changing variables $s \mapsto s - t_0$ in the second integral above, we then obtain
\[ B^1(t_0 + t) = e^{t\Delta} B^1(t_0) + \int_0^t e^{(t-s)\Delta} X(A^1(t_0 + s)) ds, \]
and thus we see that $B^1$ is a first nonlinear part of $A^1$.
Definition \ref{def:nonlinear-part-of-linear-part} abstracts this relation to the setting where $\rho(A^1)$ is not necessarily well-defined, which will be the case for us, because we are considering (random) distributional initial data.
\end{remark}

We can now state the main result of this section. This theorem is the deterministic part of the argument outlined just after the statement of Theorem \ref{thm:main}. In essence, this theorem says the following. In usual local existence arguments via contraction mapping, given $A^1$ as in Definition \ref{def:nonlinear-part-of-linear-part}, we would want to bound $\rho(A^1)$, and moreover, show that $\rho(A^1)$ is more regular than $A^1$. However, for us, $\rho(A^1)$ will not even be well-defined, because $A^1$ will be too rough. On the other hand, if we are able to obtain a proxy $B^1$ for $\rho(A^1)$, such that $B^1$ is more regular than $A^1$, then we can still run a fixed point argument to obtain a solution to \eqref{eq:ZDDS}. If it helps, one can think of this strategy as running a fixed point argument on an ``enhanced space" consisting of pairs $(A^1, B^1)$, instead of just $A^1$.

\begin{theorem}
\label{thm:rough-distributional-zdds-local-existence}
Let $\gamma_1 \in [0, 1/2)$, $\gamma_2 \in [0, 1/4)$ be such that $\gamma_1 + \gamma_2 < 1/2$. Then, there is a continuous non-increasing function $\tau_{\gamma_1 \gamma_2} : [0, \infty) \ra (0, 1]$ (which only depends on $\gamma_1, \gamma_2, d$) such that the following holds. Let $A^1 : (0, 1] \ra C^1(\torus^d, \lalg^d)$ be such that $A^1(t) = e^{(t-s)\Delta} A^1(s)$ for all $s, t \in (0, 1]$, $s < t$. Suppose $A^1 \in \mc{Q}_1^{\gamma_1}$. Suppose that there exists $B^1 \in \mc{Q}_1^{\gamma_2}$ which is a first nonlinear part for $A^1$. Let $R := \max\{\|A^1\|_{\mc{Q}_1^{\gamma_1}}, \|B^1\|_{\mc{Q}_1^{\gamma_2}}\}$, and let $T := \tau_{\gamma_1 \gamma_2}(R)$. Then there exists $B \in \mc{Q}_{T, 3R}^{\gamma_2}$ such that $A := A^1 + B$ is in $C^\infty((0, T) \times \torus^d, \lalg^d)$, and moreover $A$ is a solution to \eqref{eq:ZDDS} on $(0, T)$. 

Additionally, we have continuity in the data, in the following sense. Suppose that we have a sequence $\{A^1_n\}_{n \geq 1} \sse \mc{Q}_1^{\gamma_1}$ such that for each $n \geq 1$, $A^1_n(t) = e^{(t-s)\Delta} A^1_n(s)$ for all $s, t \in (0, 1]$, $s < t$. Suppose we have a sequence $\{B^1_n\}_{n \geq 1} \sse \mc{Q}_1^{\gamma_2}$, such that for each $n \geq 1$, $B^1_n$ is a first nonlinear part for $A^1_n$. Suppose that $\|A_n^1 - A^1\|_{\mc{Q}_1^{\gamma_1}} \ra 0$ and $\|B_n^1 - B^1\|_{\mc{Q}_1^{\gamma_2}} \ra 0$. Let $R_n := \max\{\|A^1_n\|_{\mc{Q}_1^{\gamma_1}}, \|B^1_n\|_{\mc{Q}_1^{\gamma_2}}\}$, and $T_n := \tau_{\gamma_1 \gamma_2}(R_n)$. For each $n \geq 1$, let $B_n \in \mc{Q}_{T_n, 3R_n}^{\gamma_2}$ be as constructed by the first part of the theorem, so that $A_n := A^1_n + B_n$ is a solution to \eqref{eq:ZDDS} on $(0, T_n)$. Then for all $T_0 \in (0, T)$, we have that $\|B_n - B\|_{\mc{Q}_{T_0}^{\gamma_2}} \ra 0$, which implies that $\|A_n - A\|_{\mc{Q}_{T_0}^{\max\{\gamma_1, \gamma_2\}}} \ra 0$. (Note that since $R_n \ra R$ and since $\tau_{\gamma_1 \gamma_2}$ is continuous, we have that $T_n \ra T$, so that for large enough $n$, $A_n, B_n$ are defined on $(0, T_0]$.)
\end{theorem}

We next state several auxiliary results that arise from the proof of Theorem \ref{thm:rough-distributional-zdds-local-existence}. 

\begin{lemma}\label{lemma:tau-gamma-bound}
Let $\gamma_1, \gamma_2, \tau_{\gamma_1 \gamma_2}$ be as in Theorem \ref{thm:rough-distributional-zdds-local-existence}. Then 
\[ \tau_{\gamma_1 \gamma_2}(R)^{-1} \leq \const_{\gamma_1, \gamma_2, d} (1 + R^{4/(1 - 2\max\{\gamma_1, \gamma_2\})}), ~~ R \geq 0. \]
\end{lemma}

The following lemma shows that for smooth initial data $A^0$, the solution to \eqref{eq:ZDDS} given by Theorem \ref{thm:rough-distributional-zdds-local-existence} coincides with the solution to \eqref{eq:ZDDS} with initial data $A^0$ (which is given by Theorem \ref{thm:zdds-existence}).

\begin{lemma}\label{lemma:zdds-smooth-initial-data-distributional-local-existence}
Let $\gamma_1, \gamma_2, \tau_{\gamma_1 \gamma_2}$ be as in Theorem \ref{thm:rough-distributional-zdds-local-existence}. Let $A^0$ be a smooth {\oneform}. Let $A^1(t) = e^{t \Delta} A^0$, $t \geq 0$. Let $B^1 = \rho(A^1)$. (Recall that by Remark \ref{remark:first-nonlinear-part-c1-example}, $B^1$ is a first nonlinear part for $A^1$.) Let $R :=  \max\{\|A^1\|_{\mc{Q}_1^{\gamma_1}}, \|B^1\|_{\mc{Q}_1^{\gamma_2}}\}$, $T := \tau_{\gamma_1 \gamma_2}(R) > 0$. Then, there exists $A \in C^\infty([0, T) \times \torus^d, \lalg^d)$ such that $A$ is the solution to \eqref{eq:ZDDS} on $[0, T)$ with initial data $A(0) = A^0$. Moreover, on $(0, T) \times \torus^d$, $A$ is equal to the solution to \eqref{eq:ZDDS} given by Theorem \ref{thm:rough-distributional-zdds-local-existence}.
\end{lemma}

\begin{cor}\label{cor:B-limit-B-n}
Let $\gamma_1, \gamma_2, \tau_{\gamma_1 \gamma_2}$,  $A^1, B^1, B$, $R, T$ be as in Theorem \ref{thm:rough-distributional-zdds-local-existence}. For $n \geq 2$, inductively define $B^n \in \mc{Q}_1^{\gamma_2}$ by
\[ B^n(t) := B^1(t) + \rho(B^{n-1})(t) + \eta(A^1, B^{n-1})(t), ~~ t \in (0, 1]. \]
Then $\lim_{n \toinf} \|B^n - B\|_{\mc{Q}_T^{\gamma_2}} = 0$.
\end{cor}

\subsection{Useful lemmas}\label{section:useful-lemmas}

In this section, we introduce some deterministic lemmas which will be needed later. For the first lemma, recall the definition of $\mc{P}_T^r$ from Definition \ref{def:P-T-r}, as well as the definitions of $\rho^{(2)}$ and $\rho^{(3)}$  from Definition \ref{def:rho}.  As usual, unless otherwise given, all proofs are in Appendix \ref{appendix:deterministic-pde-proofs}.

\begin{lemma}\label{lemma:rho-A-n-convergence}
Let $\{A^0_n\}_{n \leq \infty} \sse C^1(\torus^d, \lalg^d)$ be a sequence of {\oneforms}. For $n \leq \infty$, let $A^1_n(t) = e^{t\Delta} A^0_n$, $t \geq 0$. Let $T \in (0, 1]$, and suppose that $\|A_n^1 - A_\infty^1\|_{\mc{P}_T^1} \ra 0$.
Then for $j \in \{2, 3\}$, we have that $\|\rho^{(j)}(A_n^1) - \rho^{(j)}(A_\infty^1)\|_{\mc{P}_T^1} \ra 0$, and consequently, we also have that $\|\rho(A_n^1) - \rho(A_\infty^1)\|_{\mc{P}_T^1} \ra 0$.
\end{lemma}


\begin{lemma}\label{lemma:rho-3-bounds}
Let $\gamma \in [0, 1/3)$, $T \in (0, 1]$, $R \geq 0$. Let $A \in \mc{Q}_{T, R}^\gamma$. Then $\rho^{(3)}(A)$ is well-defined for $A$, and moreover $\rho^{(3)}(A) \in \mc{Q}_T^0$, and
\[ \|\rho^{(3)}(A)\|_{\mc{Q}_T^0} \leq \const T^{1-3\gamma} R^3. \]
Additionally, for $A_1, A_2 \in \mc{Q}_{T, R}^\gamma$, we have that 
\[ \|\rho^{(3)}(A_1) - \rho^{(3)}(A_2)\|_{\mc{Q}_T^0} \leq \const T^{1 - 3\gamma} R^2 \|A_1 - A_2\|_{\mc{Q}_T^\gamma}. \]
\end{lemma}

In the remainder of Section \ref{section:useful-lemmas}, we will give an explicit formula for $\rho(A^1)$, where $A^1$ is defined as  $A^1(t) = e^{t \Delta} A^0$, and $A^0$ is a smooth {\oneform}. This formula will be in terms of the Fourier coefficients of $A^0$. It will be used in Section \ref{section:technical-proofs}.

Recall from Section \ref{section:main-result} that we may view {\oneforms} $A^0 : \torus^d \ra \lalg^d$ equivalently as collections of $\R$-valued functions $(A^{0, a}_j, a \in [\lalgdim], j \in [d])$ which satisfy the relation \eqref{eq:one-form-R-valued-functions-relation}. Note that for any $1 \leq r \leq \infty$, $A^0 \in C^r(\torus^d, \lalg^d)$ if and only if $A^{0, a}_j \in C^r(\torus^d, \R)$ for all $a \in [\lalgdim]$, $j \in [d]$. In the following, recall also the structure constants $(f^{abc}, ~a, b,c \in [\lalgdim])$ defined at the end of Section \ref{section:notation}.

\begin{definition}\label{def:I-d}
For $m = (n^1, n^2) \in (\Z^d)^2$, $t \geq 0$, define
\[ I(m, t) := \int_0^t e^{-4\pi^2 |n^1 + n^2|^2(t-s)} e^{-4\pi^2 (|n^1|^2 + |n^2|^2) s} ds. \]
Additionally, for $a = (a_0, a_1, a_2) \in [\lalgdim]^3$, $j = (j_0, j_1, j_2) \in [d]^3$, define
\[ d(m, a, j) := \icomplex 2\pi f^{a_0 a_1 a_2} \big(\delta_{j_0 j_2} n^2_{j_1} - \delta_{j_0 j_1} n^1_{j_2} + (1/2) \delta_{j_1 j_2} (n^1_{j_0} - n^2_{j_0})\big). \]
Here, $\delta_{jk} = \ind(j = k)$ for $j, k \in [d]$.
\end{definition}

\begin{remark}\label{remark:I-and-d-properties} 
Note that $0 \leq I(m, t) \leq t$. Note that by \eqref{eq:structure-constants-permute-indices}, if $a_0, a_1, a_2$ are not distinct, then $f^{a_0 a_1 a_2} = 0$, and thus also $d(m, a, j) = 0$. Also, note that if we let $-m := (-n^1, -n^2)$, then (here we use that $f^{a_0 a_1 a_2}$ is real, which follows by definition -- recall \eqref{eq:structure-constants-def})
\beq\label{eq:d-complex-conjugate} d(-m, a, j) = -d(m, a, j) = \ovl{d(m, a, j)}. \eeq
Finally, note that $|d(m, a, j)| \leq \const (|n^1| + |n^2|)$. 
\end{remark}


\begin{lemma}\label{lemma:rho-A-1-smooth}
Let $A^0 \in C^\infty(\torus^d, \lalg^d)$ be a smooth {\oneform}. Let $A^1(t) = e^{t\Delta} A^0$, $t \geq 0$. Let $A^2 = \rho^{(2)}(A^1)$. For any $a_0 \in [\lalgdim]$, $j_0 \in [d]$, $t \geq 0$, we have that
\[ A^{2, a_0}_{j_0}(t) = \sum_{\substack{a_1, a_2 \in [\lalgdim] \\ j_1, j_2 \in [d]}}\sum_{n^1, n^2 \in \Z^d} I(m, t) d(m, a, j) \hat{A}^{0, a_1}_{j_1}(n^1) \hat{A}^{0, a_2}_{j_2}(n^2) e_{n^1 + n^2}, \]
where, for brevity, we have taken $m = (n^1, n^2)$, $a = (a_0,  a_1, a_2)$, $j = (j_0, j_1, j_2)$. Additionally, we have that 
\[\begin{split}
\ptl_t &A^{2, a_0}_{j_0}(t) = \Delta A^{2, a_0}_{j_0}(t) ~+ \\
&\sum_{\substack{a_1, a_2 \in [\lalgdim] \\ j_1, j_2 \in [d]}}\sum_{n^1, n^2 \in \Z^d} d(m, a, j) e^{-4\pi^2 |n^1|^2 t} \hat{A}^{0, a_1}_{j_1}(n^1) e^{-4\pi^2 |n^2|^2 t} \hat{A}^{0, a_2}_{j_2}(n^2) e_{n^1 + n^2}. \end{split} \]
(The series in the above two displays converge absolutely, by the decay of Fourier coefficients of smooth functions given in~\eqref{eq:fourier-coefficients-rapid-decay-general-dim}, combined with the facts that $0 \leq I(m, t) \leq t$ and $|d(m, a, j)| \leq \const(|n^1| + |n^2|)$.)
\end{lemma}


To prove Lemma \ref{lemma:rho-A-1-smooth}, we first show the following lemma.

\begin{lemma}\label{lemma:A-2-formula-symmetrized}
Let $A^0$ be a smooth {\oneform}. Let $A^1, A^2$ be defined using $A^0$ as in Lemma \ref{lemma:rho-A-1-smooth}. For $a_0 \in [\lalgdim]$, $j_0 \in [d]$, $t \geq 0$, we have that
\[\begin{split}
&A^{2, a_0}_{j_0}(t) \\
&= \sum_{\substack{a_1, a_2 \in [\lalgdim] \\ j \in [d]}} f^{a_0 a_1 a_2} \int_0^t e^{(t-s)\Delta} \big(A^{1, a_1}_j(s) \ptl_j A^{1, a_2}_{j_0}(s) - \ptl_j A^{1, a_1}_{j_0}(s) A^{1, a_2}_j(s) \\
&\qquad \qquad +(1/2)\big(\ptl_{j_0} A^{1, a_1}_j(s) A^{1, a_2}_j(s) - A^{1, a_1}_j(s) \ptl_{j_0} A^{1, a_2}_j(s)\big)\big) ds.
\end{split}\]
\end{lemma}
\begin{proof}
By definition, we have that (here we write $\rho^{(2)} = (\rho^{(2)}_j, j \in [d])$, so that $\rho^{(2)}_j : \torus^d \ra \lalg$ is a component function of $\rho^{(2)}$)
\[ A^2_{j_0}(t) = \rho^{(2)}_{j_0}(A^1)(t) = \int_0^t e^{(t-s)\Delta} X^{(2)}_{j_0}(A^1(s)) ds. \]
We have that
\[ X^{(2)}_{j_0}(A^1(s)) = \sum_{j \in [d]} [A^1_j(s), 2\ptl_j A_{j_0}^1(s) - \ptl_{j_0} A_j^1(s)]. \]
Now,
\[\begin{split}
[A^1_j(s), \ptl_j A_{j_0}^1(s)] &= \bigg[\sum_{a_1 \in [\lalgdim]} A^{1, a_1}_j(s) \onbasislalg^{a_1}, \sum_{a_2 \in [\lalgdim]} \ptl_j A^{1, a_2}_{j_0}(s) \onbasislalg^{a_2}\bigg] \\
&= \sum_{a_1, a_2 \in [\lalgdim]} A^{1, a_1}_j(s) \ptl_j A^{1, a_2}_{j_0}(s) [\onbasislalg^{a_1}, \onbasislalg^{a_2}] \\
&= \sum_{a_1, a_2 \in [\lalgdim]}  A^{1, a_1}_j(s) \ptl_j A^{1, a_2}_{j_0}(s)  \sum_{a \in [\lalgdim]} f^{a_1 a_2 a} \onbasislalg^a.
\end{split}\]
Using that $f^{a_1 a_2 a_0} = f^{a_0 a_1 a_2}$ (by \eqref{eq:structure-constants-permute-indices}), we obtain that
\[ [A^1_j(s), \ptl_j A_{j_0}^1(s)]^{a_0} = \sum_{a_1, a_2 \in [\lalgdim]}  f^{a_0 a_1 a_2}  A^{1, a_1}_j(s) \ptl_j A^{1, a_2}_{j_0}(s). \]
By expanding the other term similarly, we obtain
\[ \begin{split}
A^{2, a_0}_{j_0}(t)  = \sum_{\substack{a_1, a_2 \in [\lalgdim] \\ j \in [d]}} f^{a_0 a_1 a_2} \int_0^t e^{(t-s)\Delta} \big(2 &A^{1, a_1}_j(s) \ptl_j A^{1, a_2}_{j_0}(s) ~- \\
&A^{1, a_1}_j(s) \ptl_{j_0} A^{1, a_2}_j(s)\big) ds.
\end{split}\]
Now by swapping $a_1, a_2$, we also obtain
\[\begin{split}
A^{2, a_0}_{j_0}(t) = \sum_{\substack{a_1, a_2 \in [\lalgdim] \\ j \in [d]}} f^{a_0 a_2 a_1} \int_0^t e^{(t-s)\Delta} \big(2 &A^{1, a_2}_j(s) \ptl_j A^{1, a_1}_{j_0}(s) ~- \\
&A^{1, a_2}_j(s) \ptl_{j_0} A^{1, a_1}_j(s)\big) ds. 
\end{split}\]
Recall that $f^{a_0 a_2 a_1} = - f^{a_0 a_1 a_2}$ by \eqref{eq:structure-constants-permute-indices}. The desired result now follows by adding the two expressions for $A^{2, a_0}_{j_0}(t)$, and then dividing by $2$.
\end{proof}

\begin{proof}[Proof of Lemma \ref{lemma:rho-A-1-smooth}]
Throughout this proof, all infinite sums converge absolutely (by \eqref{eq:fourier-coefficients-rapid-decay-general-dim} and the facts that $0 \leq I(m, t) \leq t$, $|d(m, a, j)| \leq \const (|n^1| + |n^2|)$), and so we may freely swap the order of summation. Note that we may rewrite
\[\begin{split} \sum_{j \in [d]} \big(A^{1, a_1}_j(s) &\ptl_j A^{1, a_2}_{j_0}(s) - \ptl_j A^{1, a_1}_{j_0}(s) A^{1, a_2}_j(s) ~+ \\
&(1/2)\big(\ptl_{j_0} A^{1, a_1}_j(s) A^{1, a_2}_j(s) - A^{1, a_1}_j(s) \ptl_{j_0} A^{1, a_2}_j(s)\big)\big) \end{split}\]
as
\[\begin{split}
\sum_{j_1 j_2 \in [d]} \big( &\delta_{j_2 j_0} A^{1, a_1}_{j_1}(s) \ptl_{j_1} A^{1, a_2}_{j_2}(s) - \delta_{j_1 j_0} \ptl_{j_2} A^{1, a_1}_{j_1}(s) A^{1, a_2}_{j_2}(s) ~+ \\
&\delta_{j_1 j_2} (1/2)\big(\ptl_{j_0} A^{1, a_1}_{j_1}(s) A^{1, a_2}_{j_2}(s) - A^{1, a_1}_{j_1}(s) \ptl_{j_0} A^{1, a_2}_{j_2}(s)\big)\big).
\end{split}\]
Next, for $b \in [\lalgdim]$, $j, k \in [d]$, $s \geq 0$, we have that (recalling \eqref{eq:heat-semigroup-def} and using that $\ptl_j e_n = \icomplex 2\pi n_j e_n$)
\[ \ptl_j A^{1, b}_k(s) = \icomplex 2\pi \sum_{n \in \Z^d} n_j e^{-4\pi^2 |n|^2 s} \hat{A}^{0, b}_k(n) e_n.\]
Using this, it follows that the expression
\[\begin{split}
f^{a_0 a_1 a_2}\big(\delta_{j_2 j_0} &A^{1, a_1}_{j_1}(s) \ptl_{j_1} A^{1, a_2}_{j_2}(s) - \delta_{j_1 j_0} \ptl_{j_2} A^{1, a_1}_{j_1}(s) A^{1, a_2}_{j_2}(s) ~+ \\
&\delta_{j_1 j_2} (1/2)\big(\ptl_{j_0} A^{1, a_1}_{j_1}(s) A^{1, a_2}_{j_2}(s) - A^{1, a_1}_{j_1}(s) \ptl_{j_0} A^{1, a_2}_{j_2}(s)\big)\big)
\end{split}\]
is equal to 
\[ \sum_{n^1, n^2 \in \Z^d} d(m, a, j) e^{-4\pi^2 (|n^1|^2 + |n^2|^2)s} \hat{A}^{0, a_1}_{j_1}(n^1) \hat{A}^{0, a_2}_{j_2}(n^2)  e_{n^1 + n^2},\]
where  $m = (n^1, n^2)$, $a = (a_0, a_1, a_2)$, $j = (j_0, j_1, j_2)$.
Next, note that 
\[
e^{(t-s)\Delta} e_{n^1 + n^2} = e^{-4\pi^2 |n^1 + n^2|^2 (t-s)} e_{n^1 + n^2}
\]
for $n^1, n^2 \in \Z^d$, and
\[ \int_0^t e^{-4\pi^2 |n^1 + n^2|^2(t-s)} e^{-4\pi^2 (|n^1|^2 + |n^2|^2) s} ds = I(m, t). \]
Combining the above few observations with Lemma \ref{lemma:A-2-formula-symmetrized}, we obtain
\[ A^{2, a_0}_{j_0}(t) =  \sum_{\substack{a_1, a_2 \in [\lalgdim] \\ j_1, j_2 \in [d]}} \sum_{n^1, n^2 \in \Z^d} d(m, a, j) I(m, t) \hat{A}^{0, a_1}_{j_1}(n^1) \hat{A}^{0, a_2}_{j_2}(n^2) e_{n^1 + n^2},\]
which is the first desired result.

For the second desired result, note that for any $m = (n^1, n^2) \in (\Z^d)^2$, $t > 0$, we have that
\[ \ptl_t I(m, t) = - 4\pi^2 |n^1 + n^2|^2 I(m, t) + e^{-4\pi^2 |n^1|^2 t} e^{-4\pi^2 |n^2|^2 t}. \]
The desired result now follows (again, the interchange of differentiation and summation follows due to the rapid decay of the Fourier coefficients, i.e., \eqref{eq:fourier-coefficients-rapid-decay-general-dim}).
\end{proof}

From Lemma \ref{lemma:rho-A-1-smooth}, one can show the following. 
Recall the notation from Definition \ref{def:quadratic-forms}.

\begin{lemma}\label{lemma:A-2-is-quadratic-form}
Let $A^0, A^2$ be as in Lemma \ref{lemma:rho-A-1-smooth}. Suppose that for some $N \geq 0$, we have that $\hat{A}^0(n) = 0$ for all $|n|_\infty > N$. Then for any $a_0 \in [\lalgdim]$, $j_0 \in [d]$, $t_0 \in (0, 1]$, $x_0 \in \torus^d$, there exists a smooth function $K \in C^\infty(\mbb{I}^2, \R)$ such that 
\[ A^{2, a_0}_{j_0}(t_0, x_0) = (A^0, K A^0). \]
Additionally, for $l \in [d]$, there exists a smooth function $L \in C^\infty(\mbb{I}^2, \R)$ such that 
\[ \ptl_l A^{2, a_0}_{j_0}(t_0, x_0) = (A^0, L A^0). \]
Finally, for any $a_1 \in [\lalgdim]$, $j_1, j_2 \in [d]$, $x, y \in \torus^d$, we have that 
\[
K((a_1, j_1, x), (a_1, j_2, y)) = L((a_1, j_1, x), (a_1, j_2, y)) = 0.
\]
\end{lemma}
\begin{proof}
For $n \in \Z^d$, we may write
\[ \hat{A}^{0, a_1}_{j_1}(n) = \int_{\torus^d} A^{0, a_1}_{j_1}(x) e_{-n}(x) dx. \]
Combining this with Lemma \ref{lemma:rho-A-1-smooth} and the assumption that $\hat{A}^0(n) = 0$ for $|n|_\infty > N$, we may thus write
\[\begin{split}
&A^{2, a_0}_{j_0}(t_0, x_0)  = \sum_{\substack{a_1, a_2 \in [\lalgdim] \\ j_1, j_2 \in [d]}} \int_{\torus^d} \int_{\torus^d} A^{0, a_1}_{j_1}(x) A^{0, a_2}_{j_2}(y) \times \\
&\Bigg(\sum_{\substack{n^1, n^2 \in \Z^d \\ |n^1|_\infty, |n^2|_\infty \leq N}} I(m, t_0) d(m, a, j) e_{n^1 + n^2}(x_0) e_{-n^1}(x) e_{-n^2}(y)\Bigg) dx dy.
\end{split}\]
(In the above, $m = (n^1, n^2)$, $a = (a_0, a_1, a_2)$, $j = (j_0, j_1, j_2)$.) We may thus define, for $i_1 = (a_1, j_1, x)$, $i_2 = (a_2, j_2, y)$,
\[ K(i_1, i_2) := \sum_{\substack{n^1, n^2 \in \Z^d \\ |n^1|_\infty, |n^2|_\infty \leq N}} I(m, t_0) d(m, a, j) e_{n^1 + n^2}(x_0) e_{-n^1}(x) e_{-n^2}(y).\]
By using that $I(-m, t_0) = I(m, t_0)$ and $d(-m, a, j) = \ovl{d(m, a, j)}$, one can show that $\ovl{K(i_1, i_2)} = K(i_1, i_2)$, and so $K(i_1, i_2) \in \R$.
The fact that $K(i_1, i_2) \in \R$ follows because 
\[\begin{split}
\ovl{K(i_1, i_2)} &= \sum_{\substack{n^1, n^2 \in \Z^d \\ |n^1|_\infty, |n^2|_\infty \leq N}} I(m, t_0) \ovl{d(m, a, j)} e_{-(n^1 + n^2)}(x_0) e_{n^1}(x) e_{n^2}(y) \\
&= \sum_{\substack{n^1, n^2 \in \Z^d \\ |n^1|_\infty, |n^2|_\infty \leq N}} I(-m, t_0) \ovl{d(-m, a, j)} e_{n^1 + n^2}(x_0) e_{-n^1}(x) e_{-n^2}(y) \\
&=  \sum_{\substack{n^1, n^2 \in \Z^d \\ |n^1|_\infty, |n^2|_\infty \leq N}} I(m, t_0) d(m, a, j) e_{n^1 + n^2}(x_0) e_{-n^1}(x) e_{-n^2}(y) \\
&= K(i_1, i_2),
\end{split}\]
where we used that $I(-m, t_0) = I(m, t_0)$, and $d(-m, a, j) = \ovl{d(m, a, j)}$ (recall Remark \ref{remark:I-and-d-properties}). 
This shows the first claim. The second claim follows by a similar argument, where we define
\[ \begin{split}
&L(i_1, i_2) := \\
&\sum_{\substack{n^1, n^2 \in \Z^d \\ |n^1|_\infty, |n^2|_\infty \leq N}} \icomplex 2\pi(n^1 + n^2)_l I(m, t_0) d(m, a, j) e_{n^1 + n^2}(x_0) e_{-n^1}(x) e_{-n^2}(y).
\end{split}\]
The final claim follows because $d(m, a, j) = 0$ for any $a$ of the form $(a_0, a_1, a_1)$ (recall Remark \ref{remark:I-and-d-properties}).
\end{proof}

\section{Outline of intermediate results and proof of Theorem \ref{thm:main}}\label{section:distributional-initial-data}

In this section, we outline the intermediate results that are needed in the proof of Theorem \ref{thm:main}. Then, in Section \ref{section:proofs-main}, we show how to use these intermediate results to deduce Theorem \ref{thm:main}. The proofs of these intermediate results are deferred until Section \ref{section:technical-proofs}.

\subsection{Linear part}\label{section:linear-part}

Throughout this section, let $\rconn^0$ be a random $\lalg^d$-valued distribution satisfying Assumptions \ref{assumption:l2-regularity}, \ref{assumption:tail-bounds}, \ref{assumption:translation-invariance}, and \ref{assumption:bounded-by-fractional-greens-function}.  For this section, we just assume that Assumption \ref{assumption:bounded-by-fractional-greens-function} holds for some $\alpha \in (0, d)$, i.e., we do not need the restriction $\alpha > \max\{d - 4/3, d/2\}$ that appears in Theorem \ref{thm:main}. These assumptions hold, even if they are not explicitly stated in the various lemmas or propositions. The proofs of all results stated in this section are in Section \ref{section:technical-proofs-linear-part}.

Recall the definition of the heat kernel $\Phi$ (equation \eqref{eq:heat-kernel-def}). We proceed to define $\rconn^1$, which may be interpreted as $\rconn^1(t) = e^{t \Delta} \rconn^0$.

\begin{definition}\label{def:A-1}
Define the $\lalg^d$-valued stochastic process $\rconn^1 = (\rconn^1(t, x), t \in (0, 1], x \in \torus^d)$ by $\rconn^1(t, x) := (e^{t \Delta} \rconn^0)(x) = (\rconn^0, \Phi(t, x - \cdot))$.
\end{definition}

We will first show the following result about regularity of $\rconn^1$. 

\begin{lemma}\label{lemma:A-1-modification}
There exists a modification of $\rconn^1$
which has smooth sample paths, and which is a solution to the heat equation on $(0, 1] \times \torus^d$.
\end{lemma}

Thus from here on out, we will assume that 
$\rconn^1$ has been modified to have smooth sample paths which are solutions to the heat equation (so that $\rconn^1(t) = e^{(t-s)\Delta} \rconn^1(s)$ for all $s, t \in (0, 1]$, $s < t$). Next, we define the natural notion of Fourier truncations of $\rconn^1$. 


\begin{definition}\label{def:A-1-N}
Let $N \geq 0$.
Define the $\lalg^d$-valued stochastic process $\rconn^1_N = (\rconn^1_N(t, x), t \in (0, 1], x \in \torus^d)$ by $\rconn^1_N := \FT_N \rconn^1$.
\end{definition}

\begin{remark}\label{remark:A-1-N-convergence}
Since $\FT_N$ is linear and $\rconn^1 = e^{t \Delta} \rconn^0$, we have that $\rconn^1_N = e^{t\Delta} \FT_N \rconn^0 = e^{t \Delta} \rconn^0_N$ (recall we defined $\rconn^0_N = \FT_N \rconn^0$ in Definition \ref{def:A-0-N}). Also, we have that $\rconn^1_N(t, x) = (\FT_N \rconn^0, \Phi(t, x - \cdot)) = (\rconn^0, \FT_N \Phi(t, x - \cdot))$, and so by Assumption \ref{assumption:l2-regularity}, we have that for any $t \in (0, 1]$, $x \in \torus^d$, $\rconn^1_N(t, x) \stackrel{L^2}{\ra} \rconn^1(t, x)$.
\end{remark}

We now state the main result of Section \ref{section:linear-part}.

\begin{prop}\label{prop:linear-part-in-Q-space}
For $\varep > 0$, let $\gamma_\varep := (1/4)(d-\alpha) + \varep$. For any $\varep > 0$, $p \geq 1$, we have that
\[ \sup_{N \geq 0} \E\big[\|\rconn^1_N\|^p_{\mc{Q}_1^{\gamma_\varep}}\big], \E\big[\|\rconn^1\|^p_{\mc{Q}_1^{\gamma_\varep}}\big] \leq \const_{\varep, p} < \infty.  \]
Here, $\const_{\varep, p}$ depends only on $\varep, p, d$, and the various constants in Assumptions \ref{assumption:l2-regularity}--\ref{assumption:bounded-by-fractional-greens-function}; i.e., $\alpha$, $\exptb$, $\consttb$, etc.
Additionally, we have that
\[ \lim_{N \toinf} \E\big[\|\rconn^1_N - \rconn^1\|^p_{\mc{Q}_1^{\gamma_\varep}}\big] = 0.\]
\end{prop}

\begin{remark}\label{remark:A-1-surely-in-Q-space}
By Proposition \ref{prop:linear-part-in-Q-space}, upon replacing $\rconn^1$ by a suitable modification, we may assume that $\|\rconn^1\|_{\mc{Q}_1^{(1/4)(d-\alpha) + \varep}} < \infty$ for all $\varep > 0$. Hereafter, we assume that this holds for $\rconn^1$.
\end{remark}

\subsection{Nonlinear part}\label{section:nonlinear-part}

As in Section \ref{section:linear-part}, throughout this section, we assume that $\rconn^0$ is a random $\lalg^d$-valued distribution satisfying Assumptions \ref{assumption:l2-regularity}, \ref{assumption:tail-bounds}, \ref{assumption:translation-invariance}, and \ref{assumption:bounded-by-fractional-greens-function}. For this section, we assume (as in Theorem \ref{thm:main}) that Assumption \ref{assumption:bounded-by-fractional-greens-function} holds for some $\alpha \in (\max\{d - 4/3, d/2\}, d)$. Additionally, we assume that Assumption \ref{assumption:four-product-assumption} is satisfied. These assumptions hold, even if they are not explicitly stated in the various lemmas, corollaries, or propositions. 

Recall the process $\rconn^1$ constructed in Section \ref{section:linear-part}. This process is such that $\rconn^1(t) = e^{(t-s)\Delta} \rconn^1(s)$ for all $0 < s < t \leq 1$. In this section, we construct a first nonlinear part $\mbf{B}^1$ for $\rconn^1$, in the sense of Definition~\ref{def:nonlinear-part-of-linear-part}. We will do this by constructing $\rconn^2 = \rho^{(2)}(\rconn^1)$ and $\rconn^3 = \rho^{(3)}(\rconn^1)$ (recall Definition~\ref{def:rho} for the definitions of $\rho^{(2)}, \rho^{(3)}$), and then letting $\mbf{B}^1 = \rconn^2 + \rconn^3$. The construction of $\rconn^3$ is easier, so we handle it first.


\begin{definition}
Recall Remark \ref{remark:A-1-surely-in-Q-space} that $\rconn^1 \in \mc{Q}_1^{(1/4)(d-\alpha) + \varep}$ for all $\varep > 0$. Thus by Lemma \ref{lemma:rho-3-bounds}, and the assumption that $\alpha > d - 4/3$, we may define a $\lalg^d$-valued stochastic process $(\rconn^3(t, x), t \in (0, 1], x \in \torus^d)$ by $\rconn^3 := \rho^{(3)}(\rconn^1)$, and moreover this process is such that $\rconn^3 \in \mc{Q}_1^0$. 
Also, by the definition of $\rho^{(3)}$, the following holds. Take any $T_0 \in (0, 1)$, and let $\tilde{\rconn}^3 : [0,  1 - T_0] \ra C^1(\torus^d, \lalg^d)$ be defined by $\tilde{\rconn}^3(t) := \rconn(T_0 + t)$, $t \in [0, 1 - T_0]$. Then
\beq\label{eq:A-3-integral-identity} \tilde{\rconn}^3(t) = e^{t\Delta} \tilde{\rconn}^3(0) + \int_0^t e^{(t - s)\Delta} X^{(3)}(\rconn^1(T_0 + s)) ds, ~~ t \in [0, 1 - T_0].\eeq
For $N \geq 0$, define also the stochastic process $\rconn^3_N = (\rconn^3_N(t, x), t \in (0, 1], x \in \torus^d)$ by $\rconn^3_N := \rho^{(3)}(\rconn^1_N)$. 
\end{definition}

The next result shows that $\rconn^3_N$ converges to $\rconn^3$ as $N \toinf$, as expected.

\begin{lemma}\label{lemma:A-3-N-convergence}
For any $p \geq 1$, we have that 
\[ \sup_{N \geq 0} \E\big[\|\rconn^3_N\|^p_{\mc{Q}_1^0}\big], \E\big[\|\rconn^3\|^p_{\mc{Q}_1^0}\big] \leq \const_p < \infty. \]
Here $\const_p$ depends only on $p, d$, and the various constants in Assumptions \ref{assumption:l2-regularity}--\ref{assumption:bounded-by-fractional-greens-function}; i.e., $\alpha$, $\exptb$, $\consttb$, etc. Additionally,
\[ \lim_{N \toinf} \E\big[\|\rconn^3_N - \rconn^3\|_{\mc{Q}_1^0}^p\big] = 0. \]
\end{lemma}
\begin{proof}
Both claims follow by combining Lemma \ref{lemma:rho-3-bounds}, H\"{o}lder's inequality, and Proposition \ref{prop:linear-part-in-Q-space} with large enough $p$.
\end{proof}

We next proceed to construct $\rconn^2 = \rho^{(2)}(\rconn^1)$. We cannot just construct this deterministically as we did for $\rconn^3 = \rho^{(3)}(\rconn^1)$, because $\rconn^1$ is too rough, so that $\rho^{(2)}(\rconn^1)$ will not be well-defined. Instead, $\rconn^2$ will be constructed probabilistically. 

\begin{definition}\label{def:A-2-N}
For $N \geq 0$, define the process $\rconn^2_N = (\rconn^2_N(t, x), t \in (0, 1], x \in \torus^d)$ by $\rconn^2_N := \rho^{(2)}(\rconn^1_N)$. 
\end{definition} 

We proceed to construct $\rconn^2$ as an appropriate limit of $\rconn^2_N$. First,
we show the following result. The proof is in Section \ref{section:technical-proofs-nonlinear-part}.

\begin{lemma}\label{lemma:B-1-cauchy}
For any $t \in (0, 1]$, $x \in \torus^d$, we have that $\{\rconn^{2}_{N}(t, x)\}_{N \geq 0}$ is a Cauchy sequence in $L^2$.
\end{lemma}

This leads directly to the following definition.

\begin{definition}\label{def:nonlinear-part-l2-limit}
Define the $\lalg^d$-valued stochastic process $\rconn^2 = (\rconn^{2}(t, x), t \in (0, 1], x \in \torus^d)$ as follows. By Lemma \ref{lemma:B-1-cauchy}, we have that $\{\rconn^{2}_{N}(t, x)\}_{N \geq 0}$ is Cauchy in $L^2$, and thus the sequence converges in $L^2$. Define $\rconn^{2}(t, x)$ to be the limit.
\end{definition}

Having defined $\rconn^2$, the next step is the following. The proof is in Section \ref{section:technical-proofs-nonlinear-part}.

\begin{lemma}\label{lemma:A-2-continuous-modification}
The process $\rconn^2$ has a modification such that the function $t \mapsto \rconn^{2}(t)$ is a continuous function from $(0, 1]$ into $C^1(\torus^d, \lalg^d)$. 
\end{lemma}

Thus hereafter, we assume that (after a suitable modification) $\rconn^2$ is such that the function $t \mapsto \rconn^{2}(t)$ is a continuous function from $(0, 1]$ into $C^1(\torus^d, \lalg^d)$.

\begin{definition}\label{def:A-2-from-A-2-a-j}
Define the $\lalg^d$-valued stochastic process $\mbf{B}^1 = (\mbf{B}^1(t, x), t \in (0, 1], x \in \torus^d)$ by $\mbf{B}^1(t, x) := \rconn^2(t, x) + \rconn^3(t, x)$.
For $N \geq 0$, let $\mbf{B}^1_N := \rho(\rconn^1_N) = \rho^{(2)}(\rconn^1_N) + \rho^{(3)}(\rconn^1_N) = \rconn^2_N + \rconn^3_N$.
\end{definition}

The next result shows that $\mbf{B}^1$ is indeed a first nonlinear part for $\rconn^1$, in the sense of Definition \ref{def:nonlinear-part-of-linear-part}. The proof is in Section~\ref{section:technical-proofs-nonlinear-part}.

\begin{lemma}\label{lemma:B-1-is-nonlinear-part}
On an event of probability $1$, we have that for all $t_0, t_1 \in (0, 1]$, $t_0 < t_1$, $x \in \torus^d$, 
\[ \mbf{B}^1(t_1, x) = (e^{(t_1 - t_0) \Delta} \mbf{B}^1(t_0))(x) + \int_0^{t_1 - t_0} \big(e^{(t_1 - t_0 - s)\Delta} X(\rconn^1(t_0 + s))\big)(x) ds. \]
\end{lemma}

In light of Lemma \ref{lemma:B-1-is-nonlinear-part}, hereafter, we assume that $\rconn^1, \rconn^2, \rconn^3$ have been modified so that $\mbf{B}^1$ is a first nonlinear part of $\rconn^1$ (in the sense of Definition \ref{def:nonlinear-part-of-linear-part}). We can now finally state the main result of Section \ref{section:nonlinear-part}. The proof is in Section \ref{section:technical-proofs-nonlinear-part}.

\begin{prop}\label{prop:B-1}
For $\varep > 0$, let $\gamma_\varep := (1/2)(d-1-\alpha) + \varep$. For any $\varep > 0$, $p \geq 1$, we have that
\[ \sup_{N \geq 0} \E\big[\|\rconn^2_N\|^p_{\mc{Q}_1^{\gamma_\varep}}\big], \E\big[\|\rconn^2\|^p_{\mc{Q}_1^{\gamma_\varep}}\big] \leq \const_{\varep, p} < \infty, ~~\lim_{N \toinf} \E\big[\|\rconn^2_N - \rconn^2\|_{\mc{Q}_1^{\gamma_\varep}}^p\big] = 0.\]
Consequently, we have that 
\[ \sup_{N \geq 0} \E\big[\|\mbf{B}^1_N\|_{\mc{Q}_1^{\gamma_\varep}}^p\big], \E\big[\|\mbf{B}^1\|_{\mc{Q}_1^{\gamma_\varep}}^p\big] \leq \const_{\varep, p} < \infty, \]
\[ \lim_{N \toinf} \E\big[\|\mbf{B}^1_N - \mbf{B}^1\|_{\mc{Q}_1^{\gamma_\varep}}^p\big] = 0. \]
Here, $\const_{\varep, p}$ depends only on $\varep, p, d$, and the various constants in Assumptions \ref{assumption:l2-regularity}--\ref{assumption:four-product-assumption}; i.e., $\alpha$, $\exptb$, $\consttb$, etc.
\end{prop}
\begin{remark}\label{remark:B-1-surely-in-Q-space}
By Proposition \ref{prop:B-1}, upon replacing $\rconn^1, \mbf{B}^1$ by suitable modifications, we may assume that $\|\mbf{B}^1\|_{\mc{Q}_1^{(1/2)(d-1-\alpha) + \varep}} < \infty$ for all $\varep > 0$, while still ensuring that $\mbf{B}^1$ is a first nonlinear part of $\rconn^1$. Hereafter, we assume that this holds for $\rconn^1, \mbf{B}^1$. 
\end{remark}

\subsection{Proofs of Theorem \ref{thm:main} and Corollary \ref{cor:gff}}\label{section:proofs-main}

We can now prove Theorem \ref{thm:main} by combining Propositions \ref{prop:linear-part-in-Q-space} and \ref{prop:B-1} with Theorem \ref{thm:rough-distributional-zdds-local-existence}.

\begin{proof}[Proof of Theorem \ref{thm:main}]
Let $\rconn^1$ be as constructed in Section \ref{section:linear-part}, and let $\mbf{B}^1$ be as constructed in Section \ref{section:nonlinear-part}. We want to apply Theorem \ref{thm:rough-distributional-zdds-local-existence} to $\rconn^1, \mbf{B}^1$. First, note that by the assumption that $\alpha > d - 4/3$ in the statement of Theorem \ref{thm:main}, we have that $(1/4)(d-\alpha) < 1/2$, $(1/2)(d-1-\alpha) < 1/4$, $(1/4)(d-\alpha) + (1/2)(d-1-\alpha) < 1/2$. Therefore, we may take $\varep_0 > 0$ small enough such that defining $\gamma_1 := (1/4)(d-\alpha) + \varep_0$, $\gamma_2 := (1/2)(d-1-\alpha) +\varep_0$, we have that $\gamma_1 \in [0, 1/2)$, $\gamma_2 \in [0, 1/4)$, and $\gamma_1 + \gamma_2 < 1/2$. By Proposition \ref{prop:linear-part-in-Q-space} and Remark \ref{remark:A-1-surely-in-Q-space}, we have that $\|\rconn^1\|_{\mc{Q}_1^{\gamma_1}} < \infty$, and similarly by Proposition \ref{prop:B-1} and Remark \ref{remark:B-1-surely-in-Q-space}, we have that $\|\mbf{B}^1\|_{\mc{Q}_1^{\gamma_2}} < \infty$.

Let $R = \max\{\|\rconn^1\|_{\mc{Q}_1^{\gamma_1}}, \|\mbf{B}^1\|_{\mc{Q}_1^{\gamma_2}}\}$, and let $T = \tau_{\gamma_1 \gamma_2}(R)$, where $\tau_{\gamma_1 \gamma_2}$ is as in Theorem \ref{thm:rough-distributional-zdds-local-existence}. Note that since $R$ is a random variable and $\tau_{\gamma_1 \gamma_2}$ is continuous, we have that $T$ is also a random variable. 
The fact that $\E[T^{-p}] < \infty$ for all $p \geq 1$ follows by Lemma \ref{lemma:tau-gamma-bound} and Propositions \ref{prop:linear-part-in-Q-space} and \ref{prop:B-1}. By Theorem \ref{thm:rough-distributional-zdds-local-existence}, there exists $\mbf{B} \in \mc{Q}_{T, 3R}^{\gamma_2}$ such that $\rconn^1 + \mbf{B}$ is in $C^\infty((0, T) \times \torus^d, \lalg^d)$, and moreover $\rconn^1 + \mbf{B}$ is a solution to \eqref{eq:ZDDS} on $(0, T)$.
We can then define the process $\rconn = (\rconn(t, x), t \in (0, 1],  x \in \torus^d)$ by:
\[ \rconn(t, x) := \ind(t < T) (\rconn^1(t, x) + \mbf{B}(t, x)). \]

We next show that $\rconn$ is indeed a stochastic process, i.e., that $\rconn(t, x)$ is indeed a random variable (i.e., measurable) for all $t \in (0, 1], x \in \torus^d$. Since we know that $\rconn^1$ is a stochastic process, and $T$ is a random variable, we have that $\ind(t < T) \rconn^1(t, x)$ is a random variable. Next, following Corollary \ref{cor:B-limit-B-n}, for $n \geq 2$, we may inductively define $\mbf{B}^n \in \mc{Q}_1^{\gamma_2}$ by
\[ \mbf{B}^n := \mbf{B}^1 + \rho(\mbf{B}^{n-1}) + \eta(\rconn^1, \mbf{B}^{n-1}). \]
By Corollary \ref{cor:B-limit-B-n}, we have that $\|\mbf{B}^n - \mbf{B}\|_{\mc{Q}_T^{\gamma_2}} \ra 0$. Thus to show that $\ind(t < T) \mbf{B}(t, x)$ is a random variable, it suffices to show for all $n \geq 1$ that $\mbf{B}^{n} = (\mbf{B}^{n}(t, x), t \in (0, 1], x \in \torus^d)$ is a stochastic process. The base case $n = 1$ follows by construction. For general $n$, note that recalling the definitions of $\rho$ and $\eta$ (Definitions \ref{def:rho} and \ref{def:eta}), we have that $\mbf{B}^n$ is defined in terms of integrals involving $\rconn^1$ and $\mbf{B}^{n-1}$. We may thus proceed inductively, using the facts that $\rconn^1 \in \mc{Q}_1^{\gamma_1}$ and $\mbf{B}^{n-1} \in \mc{Q}_1^{\gamma_2}$ (and Lemmas \ref{lemma:distributional-contraction-bound} and \ref{lemma:distributional-cross-term-bound}) to express the integrals in terms of limits of finite sums involving $\rconn^1, \mbf{B}^{n-1}$.

We now move on to the second part of the theorem. Let $\{\rconn^1_N\}_{N \geq 0}$, $\{\mbf{B}^1_N\}_{N \geq 0}$ be as constructed in Sections \ref{section:linear-part} and \ref{section:nonlinear-part}, respectively. For $N \geq 0$, let $R_N = \max(\|\rconn^1_N\|_{\mc{Q}_1^{\gamma_1}}, \|\mbf{B}^1_N\|_{\mc{Q}_1^{\gamma_2}})$, and let $T_N = \tau_{\gamma_1 \gamma_2}(R_N)$. For the same reasons as before, we have that $\sup_{N \geq 0} \E[T_N^{-p}] < \infty$ for all $p \geq 1$. Also, by Propositions~\ref{prop:linear-part-in-Q-space} and~\ref{prop:B-1}, we have that $R_N \stackrel{P}{\ra} R$ (here $\stackrel{P}{\ra}$ denotes convergence in probability) and thus, since $\tau_{\gamma_1 \gamma_2}$ is continuous, we obtain that $T_N \stackrel{P}{\ra} T$. This implies that $T_N^{-1} \stackrel{P}{\ra} T^{-1}$. The fact that $\E[|T_N^{-1} - T^{-1}|^p] \ra 0$ for all $p \geq 1$ now follows by Vitali's (\cite[(21.2) Theorem]{RW1994}) theorem (combined with the $L^p$-boundedness for any $p \geq 1$, which gives uniform integrability).

For each $N \geq 0$, we apply Theorem \ref{thm:rough-distributional-zdds-local-existence} and Lemma \ref{lemma:zdds-smooth-initial-data-distributional-local-existence} to obtain $\mbf{B}_N \in \mc{Q}_{T_N, 3R_N}^{\gamma_2}$ such that $\rconn^1_N + \mbf{B}_N$ is in $C^\infty([0, T_N) \times \torus^d, \lalg^d)$, and moreover it is the solution to \eqref{eq:ZDDS} on $[0, T_N)$ with initial data $\rconn_N^0$ (recall that the solution is unique, by Lemma \ref{lemma:zdds-uniqueness}). We may thus define the process $\rconn_N = (\rconn_N(t, x), t \in [0, 1], x \in \torus^d)$ by
\[ \rconn_N(t, x) := \ind(t < T_N) (\rconn^1_N(t, x) + \mbf{B}_N(t, x)). \]
(The fact that $\rconn_N$ is a stochastic process follows by the same proof as before.) 
It remains to show the last claims about convergence of $\rconn_N$ to $\rconn$. Let $p \geq 1$, $\delta \in (0, 1)$, $\varep > 0$.
Note that by Proposition \ref{prop:linear-part-in-Q-space}, 
\[
\|\rconn^1_N - \rconn^1\|_{\mc{Q}_1^{\gamma_1}} \stackrel{P}{\ra} 0,
\]
and that by Proposition~\ref{prop:B-1},
\[
\|\mbf{B}^1_N - \mbf{B}^1\|_{\mc{Q}_1^{\gamma_2}} \stackrel{P}{\ra} 0.
\]
Combining this with Theorem \ref{thm:rough-distributional-zdds-local-existence}, we obtain that 
\[
\|\mbf{B}_N - \mbf{B}\|_{\mc{Q}_{(1-\delta)T}^{\gamma_2}} \stackrel{P}{\ra} 0.
\]
(This can be shown by using the fact that convergence to $0$ in probability is equivalent to the property that for any subsequence, there is a further subsequence which converges to $0$ a.s. The $(1-\delta)T$ comes from the fact that $\mbf{B}_N$ is a solution to \eqref{eq:ZDDS} on $[0, T_N)$, and $T_N$ may be less than $T$. However, we know that $T_N \stackrel{P}{\ra} T$.) We thus also obtain
\[
\|\rconn_N - \rconn\|_{\mc{Q}_{(1 - \delta)T}^{\gamma_1}} \stackrel{P}{\ra} 0,
\]
because the assumption that $\alpha > d - 4/3$ implies that $\gamma_1 > \gamma_2$. Now to show the last claims about convergence of $\rconn_N$ to $\rconn$, it suffices (by Vitali's theorem --- \cite[(21.2) Theorem]{RW1994}) to show that for any $p \geq 1$, the sequence 
\[
\{\|\rconn_N - \rconn\|_{\mc{Q}_{(1 - \delta)T}^{\gamma_2}}^p\}_{N \geq 0}
\]
is $L^2$-bounded (and thus uniformly integrable). Fix $p \geq 1$. Since $\rconn_N(t) = 0$ if $t \geq T_N$, we have that
\[  \|\rconn_N - \rconn\|_{\mc{Q}_{(1-\delta)T}^{\gamma_1}}^{2p} \leq \const_p \big(\|\rconn_N\|_{\mc{Q}_{T_N}^{\gamma_1}}^{2p} + \|\rconn\|_{\mc{Q}_T^{\gamma_1}}^{2p}\big). \]
We have that $\rconn_N = \rconn^1_N + \mbf{B}_N$, where $\rconn^1_N, \mbf{B}_N \in \mc{Q}_{T_N, 3R_N}^{\gamma_1}$, and similarly for $\rconn$. From this, we obtain \[ \|\rconn_N\|_{\mc{Q}_{T_N}^{\gamma_1}}^{2p} \leq \const_p R_N^{2p}, ~~ \|\rconn\|_{\mc{Q}_T^{\gamma_1}}^{2p} \leq \const_p R^{2p} .\]
Now by Propositions \ref{prop:linear-part-in-Q-space} and \ref{prop:B-1}, we have that $\sup_{N \geq 0} \E[R_N^{2p}] < \infty$, $\E[R^{2p}] <  \infty$. The desired $L^2$-boundedness now follows. Thus we have shown the last claims for $\varep = \varep_0$, where $\varep_0$ is a small enough quantity that we fixed at the beginning. The last claims for general $\varep > 0$ then follow, because due to monotonicity in $\varep$, it just suffices to show the claims for small enough $\varep > 0$.
\end{proof}

We next turn to proving Corollary \ref{cor:gff}.

\begin{lemma}\label{lemma:GFF-satisfies-assumptions}
Let $d = 3$, and let $\rconn^0$ be a 3D $\lalg^3$-valued GFF. Then Assumptions \ref{assumption:l2-regularity}--\ref{assumption:four-product-assumption} are satisfied, and moreover Assumption \ref{assumption:bounded-by-fractional-greens-function} is satisfied with $\alpha = 2$. 
\end{lemma}

Before we prove Lemma \ref{lemma:GFF-satisfies-assumptions}, we note that Corollary \ref{cor:gff} follows directly. 

\begin{proof}[Proof of Corollary \ref{cor:gff}]
This follows from Theorem \ref{thm:main} and Lemma \ref{lemma:GFF-satisfies-assumptions}. 
\end{proof}

The rest of this section is devoted to the proof of Lemma \ref{lemma:GFF-satisfies-assumptions}. Thus, we assume in the rest of this section that $d = 3$, and $\rconn^0$ is a 3D $\lalg^3$-valued GFF. As in Section~\ref{section:introduction}, this may be constructed explicitly, as follows. Fix a subset $I_\infty \sse \Z^3$ such that $0 \notin I_\infty$, and such that for each $n \in \Z^3 \setminus \{0\}$, exactly one of $n, -n$ is in $I_\infty$. Let $(Z^a_j(n), a \in [\lalgdim], j \in [d], n \in I_\infty)$ be an i.i.d.~collection of standard complex Gaussian random variables. For $n \in \Z^3 \setminus \{0\}$, $n \notin I_\infty$, define $Z^a_j(n) := \ovl{Z^a_j(-n)}$. Then define
\[ \rconn^{0, a}_j := \sum_{\substack{n \in \Z^3 \\ n \neq 0}} \frac{Z^a_j(n)}{|n|} e_n,  ~~ a \in [\lalgdim], j \in [3]. \]
Recall from Section \ref{section:introduction} (in particular, \eqref{eq:GFF-phi-def}) that this is interpreted as for any $\phi \in C^\infty(\threetorus, \R)$, we have that
\[ (\rconn^{0, a}_j, \phi) \stackrel{a.s.}{=} \sum_{\substack{n \in \Z^3 \\ n \neq 0}} \frac{Z^a_j(n)}{|n|} \ovl{\hat{\phi}(n)}. \]
(Recall the discussion after \eqref{eq:GFF-phi-def} for why the series converges.)
The proof of Lemma \ref{lemma:GFF-satisfies-assumptions} is split into several parts.

\begin{proof}[Proof of Assumption \ref{assumption:l2-regularity}]
It suffices to verify the assumption ``coordinate-wise", i.e. for the processes $\rconn^{0, a}_j$. Thus, fix $a \in [\lalgdim]$, $j \in [d]$. For $\phi \in C^\infty(\torus^3, \R)$, we have $(\rconn^{0, a}_j, \phi) \sim N(0, (\sigma^a_{\phi, j})^2)$, where
\[ (\sigma^a_{\phi, j})^2 = \sum_{\substack{n \in \Z^3 \\ n \neq 0}} \frac{1}{|n|^2} |\hat{\phi}(n)|^2 < \infty. \]
(The sum is finite by \eqref{eq:fourier-coefficients-rapid-decay-general-dim}.)
We then have that
\[ (\sigma^a_{\phi - \FT_N \phi, j})^2 = \sum_{\substack{n \in \Z^d \\ |n|_\infty > N}} \frac{1}{|n|^2} |\hat{\phi}(n)|^2 \ra 0 \text{ as $N \toinf$}. \]
Assumption \ref{assumption:l2-regularity} thus follows.
\end{proof}

We skip over Assumption \ref{assumption:tail-bounds} and move on to the next assumptions for now. 

\begin{proof}[Proof of Assumptions \ref{assumption:translation-invariance} and \ref{assumption:bounded-by-fractional-greens-function}]
From the explicit construction of $\rconn^{0, a}_j$ as a random Fourier series, we have that for $a_1, a_2 \in [\lalgdim]$, $j_1, j_2 \in [d]$, $\phi_1, \phi_2 \in C^\infty(\torus^d, \R)$,
\[\begin{split}
\E \big[(\rconn^{0, a_1}_{j_1}, \phi_1) (\rconn^{0, a_2}_{j_2}, \phi_2)\big] &= \ind_{(a_1, j_1) = (a_2, j_2)} \sum_{\substack{n \in \Z^3 \\ n \neq 0}} \frac{1}{|n|^2} \hat{\phi}_1(n) \ovl{\hat{\phi}_2(n)} \\
&= \ind_{(a_1, j_1) = (a_2, j_2)} \int_{\threetorus} \int_{\threetorus} \phi_1(x) G^2(x, y) \phi_2(y) dx dy.
\end{split}\]
Define the function $\rho : (\torus^d)^2 \ra L(\lalg^d, \lalg^d)$ by specifying that the $((a_1, j_1), (a_2, j_2))$ matrix entry (with respect to the basis on $\lalg^d$ induced by the orthonormal basis $(S^a, a \in [\lalgdim])$ of $\lalg$ that we fixed) is the function $\ind_{(a_1, j_1) = (a_2, j_2)} G^2$. Then for any linear map $K : \lalg^d \ra \lalg^d$, we have that
\[ \E[\langle (\rconn^0, \phi_1), K(\rconn^0, \phi_2)\rangle] = \int_{\torus^3} \int_{\torus^3} \phi_1(x) \phi_2(y) \Tr(K \rho(x, y)^t)  dx dy. \]
Since $G^2$ is translation invariant, so is $\rho$. Thus Assumption \ref{assumption:translation-invariance} is verified.

For Assumption \ref{assumption:bounded-by-fractional-greens-function}, note that from the definition of $\rho$,
\[ \Tr(\rho(x, y)) = (\lalgdim \cdot d) G^2(x, y).\]
Since $G^2$ is bounded from below by Lemma \ref{lemma:fractional-greens-function-properties}, it follows that by taking large enough $\constgf$, we have that $|\Tr(\rho(x, y))| \leq \constgf (G^2(x, y) + \constgf)$ for all $x, y \in \threetorus$, $x \neq y$, and thus Assumption \ref{assumption:bounded-by-fractional-greens-function} holds.
\end{proof}

\begin{proof}[Proof of Assumption \ref{assumption:four-product-assumption}]
Let all notation be as in Assumption \ref{assumption:four-product-assumption}. Note that since $a_1 \neq a_2$, we have that $(Z_1, Z_3)$ is independent of $(Z_2, Z_4)$. It follows that
\[ \E[Z_1 Z_2 \ovl{Z_3 Z_4}] = \E[Z_1 \ovl{Z_3}] \E [Z_2 \ovl{Z_4}].\]
Thus Assumption \ref{assumption:four-product-assumption} is satisfied with $\constfourp = 1$.
\end{proof}

\begin{proof}[Proof of Assumption \ref{assumption:tail-bounds}]
We will show the assumption with $\exptb = 1$. We will verify the assumption ``coordinate-wise", i.e. for the processes $\rconn^{0, a}_j$. First, for any $\phi \in C^\infty(\threetorus,  \R)$,  $a \in [\lalgdim]$, $j \in [3]$, we have that $(\rconn^{0, a}_j, \phi) \sim N(0, (\sigma^a_{\phi, j})^2)$ (here we define $(\sigma^a_{\phi, j})^2$ to be the variance of $(\rconn^{0, a}_j, \phi)$), and thus by the standard Gaussian tail bound, we have that 
\[ \p(|(\rconn^{0, a}_j, \phi)| > u) \leq 2 \exp(-u^2 / (2 (\sigma^a_{\phi, j})^2)), ~~u \geq 0.\]
By splitting into cases $u \leq \sigma^a_{\phi, j}$ and $u \geq \sigma^a_{\phi, j}$, it then follows that 
(for the latter case, note that $1 \leq 2 e^{-1/2}$)
\[ \p(|(\rconn^{0, a}_j, \phi)| > u) \leq 2 \exp(-u / (2 \sigma^a_{\phi, j})),  ~~u \geq 0.\]
We next turn to concentration for quadratic forms. Fix $N \geq 0$. Observe that $(\rconn^{0, a}_{N, j}(x), a \in [\lalgdim], j \in [3], x \in \threetorus)$ is a mean $0$ Gaussian process with smooth sample paths. Let $K$ be as in Assumption \ref{assumption:tail-bounds}. For notational simplicity, let $Q = (\rconn^0_N, K \rconn^0_N)$. Let $k \geq 1$. By approximating $\threetorus$ by a lattice with spacing $1/k$, we may obtain a random variable $Q_k$, which is a Riemann sum approximation of $Q$ (recall the definition of $(\rconn^0_N, K \rconn^0_N)$ in equation \eqref{eq:quadratic-form-def}). Moreover, $Q_k$ is a quadratic form of a centered Gaussian vector. Also, we have that $\E (Q_k)  = 0$, because $\rconn^0_N$ is a mean $0$ process, $\rconn^{0, a_1}_{N, j_1}, \rconn^{0, a_2}_{N, j_2}$ are independent for $a_1 \neq a_2$, and the assumption that $K((a, j_1, x), (a, j_2, y)) = 0$. Thus by Lemma \ref{lemma:quadratic-form-gaussian-concentration}, we have that for all $k \geq 1$,
\[ \p(|Q_k| > u) \leq 2 e^{3/2} \exp(-u / (2 (\E(Q_k^2))^{1/2})), ~~ u \geq 0.\]
Now since $\rconn^0_N$ has smooth sample paths, we have that $Q_k \ra Q$. It then follows that (by, e.g., Fatou's lemma)
\[ \p(|Q| > u) \leq \liminf_k \p(|Q_k| > u). \]
Thus to finish, it suffices to show that $\limsup_k \E (Q_k^2)\leq \E(Q^2)$ (note that $\E(Q) = 0$ as well, for the same reasons why $\E(Q_k) = 0$). Actually, we have that $\E[(Q_k - Q)^2] \ra 0$, because $Q_k \ra Q$, and the sequence $\{(Q_k - Q)^2\}_{k \geq 1}$ is uniformly integrable. The uniform integrability follows because $\sup_{k \geq 1} \E(Q_k^4), \E(Q^4) < \infty$, which itself can be seen from the fact that $\rconn^0_N$ is a Gaussian process such that $\sup_{x \in \threetorus} \E[|\rconn^{0}_{N}(x)|^2] < \infty$.
\end{proof}

\subsection{A related result}\label{section:useful-variant}

As mentioned in Section \ref{section:introduction}, the results of this paper will be applied in \cite{CaoCh2021}. In particular, we will use Corollary \ref{cor:gff} in \cite{CaoCh2021}. However, we will not directly use Theorem \ref{thm:main}. Instead, we now give a related result that will be more suited for the purposes of \cite{CaoCh2021}. First, suppose that $\rconn^0 = (\rconn^0(x), x \in \torus^d)$ is now a $\lalg^d$-valued stochastic process with smooth sample paths. Note that this naturally induces a random $\lalg^d$-valued distribution by defining $(\rconn^0, \phi) := \int_{\torus^d} \rconn^0(x) \phi(x) dx$ for all $\phi \in C^\infty(\torus^d, \R)$. We say that $\rconn^0$ satisfies some given assumption if the corresponding random $\lalg^d$-valued distribution satisfies the assumption.

\begin{prop}\label{prop:main-useful-variant}
Let $\rconn^0 = (\rconn^0(x), x \in \torus^d)$ be a $\lalg^d$-valued stochastic process with smooth sample paths. Suppose that it satisfies Assumptions \ref{assumption:l2-regularity}--\ref{assumption:four-product-assumption}. Moreover, suppose that Assumption \ref{assumption:bounded-by-fractional-greens-function} is satisfied with $\alpha \in (\max\{d - 4/3, d/2\}, d)$. Then there exists a $\lalg^d$-valued stochastic process $\rconn = (\rconn(t, x), t \in [0, 1), x \in \torus^d)$, and a random variable $T \in (0, 1]$, such that the following hold. The function $(t, x) \mapsto \rconn(t, x)$ is in $C^\infty([0, T) \times \torus^d, \lalg^d)$, and moreover it is the solution to \eqref{eq:ZDDS} on $[0, T)$ with initial data $\rconn(0) = \rconn^0$. Also, $\E[T^{-p}] \leq \const_p < \infty$ for all $p \geq 1$, where $\const_p$ depends only on $p, d$, and the various constants in Assumptions \ref{assumption:l2-regularity}--\ref{assumption:four-product-assumption}; i.e., $\alpha$, $\exptb$, $\consttb$, etc. Finally, for any $k \in \{0, 1\}$, $p \geq 1$, $\varep > 0$, we have that
\[\begin{split}
\E\bigg[\sup_{\substack{t \in (0, T)}} t^{p((k/2) + (1/4)(d-\alpha) + \varep)}\|\rconn(t)\|_{C^k}^p \bigg] &\leq \const_{p, \varep}. 
\end{split}\]
Here $\const_{p, \varep}$ depends only on $p, \varep, d$, and the various constants in Assumptions \ref{assumption:l2-regularity}--\ref{assumption:four-product-assumption}; i.e., $\alpha$, $\exptb$, $\consttb$, etc.
\end{prop}

This proposition gives bounds on $\rconn$, as opposed to Theorem \ref{thm:main}, which only gives that $\rconn$ exists, and that $\rconn_N$ converges to $\rconn$ in a suitable sense. Also, note that in contrast to Theorem \ref{thm:main}, we take sup over $t \in (0, T)$ as opposed to sup over $t \in (0, (1-\delta)T)$ in the final two inequalities. This is because we are only bounding $\rconn$, which we know exists on $(0, T)$, as opposed to $\rconn_N - \rconn$. Recall that $\rconn_N$ is a solution to \eqref{eq:ZDDS} only on $[0, T_N)$, and it may be the case that $T_N < T$ (but on the other hand, we do have that $T_N > (1-\delta) T$ with probability tending to $1$ as $N \toinf$). Before we prove Proposition \ref{prop:main-useful-variant}, we need the following natural lemma.

\begin{lemma}\label{lemma:main-useful-variant}
Let $\rconn^0 = (\rconn^0(x), x \in \torus^d)$ be as in Proposition \ref{prop:main-useful-variant}. Let $\rconn^1, \mbf{B}^1$ be constructed using $\rconn^0$ as in Sections \ref{section:linear-part} and \ref{section:nonlinear-part}. Then a.s., for all $t \in (0, 1], x \in \torus^d$, we have that $\rconn^1(t, x) = (e^{t \Delta} \rconn^0)(x)$, $\mbf{B}^1(t, x) = (\rho(\rconn^1)(t))(x)$. 
\end{lemma}
\begin{proof}
Let $\tilde{\rconn}^1 : [0, 1] \ra C^1(\torus^d, \lalg^d)$ be given by $\tilde{\rconn}^1(t) = e^{t\Delta} \rconn^0$. Let $\tilde{\mbf{B}}^1 : [0, 1] \ra C^1(\torus^d, \lalg^d)$ be given by $\tilde{\mbf{B}}^1 = \rho(\tilde{\rconn}^1)$. Let $\{\rconn^0_N\}_{N \geq 0}$ be the Fourier truncations of $\rconn^0$ as in Definition \ref{def:fourier-truncations}. Recall from Definitions \ref{def:A-1-N} and \ref{def:A-2-from-A-2-a-j} the processes $\rconn^1_N, \mbf{B}^1_N$, given by $\rconn^1_N(t) = e^{t\Delta} \rconn^0_N$, $\mbf{B}^1_N = \rho(\rconn^1_N)$. Now since $\rconn^0$ is smooth, we have that $\|\rconn^0_N - \rconn^0\|_{C^1} \ra 0$ (which folows by e.g. \eqref{eq:fourier-coefficients-rapid-decay-general-dim}). This implies that (using \eqref{eq:heat-semigroup-Cr-contraction})
$\|\rconn^1_N - \tilde{\rconn}^1\|_{\mc{P}_1^1} \ra 0$ (recall Definition \ref{def:P-T-r}). Then by Lemma \ref{lemma:rho-A-n-convergence}, we also obtain $\|\mbf{B}^1_N - \tilde{\mbf{B}}^1\|_{\mc{P}_1^1} \ra 0$. On the other hand, by Propositions \ref{prop:linear-part-in-Q-space} and \ref{prop:B-1}, we have that for some $\gamma_1, \gamma_2$,
\[ \lim_{N \toinf} \E[\|\rconn^1_N - \rconn^1\|_{\mc{Q}_1^{\gamma_1}}]  = 0, ~~ \lim_{N \toinf} \E[\|\mbf{B}^1_N - \mbf{B}^1\|_{\mc{Q}_1^{\gamma_2}} = 0. \] 
By combining the previous few observations, we obtain that a.s., $\|\rconn^1 - \tilde{\rconn}^1\|_{\mc{Q}_1^{\gamma_1}} = 0$ and $\|\mbf{B}^1 - \tilde{\mbf{B}}^1\|_{\mc{Q}_1^{\gamma_2}} = 0$. The desired result now follows.
\end{proof}

\begin{proof}[Proof of Proposition \ref{prop:main-useful-variant}]
We slightly modify the proof of Theorem \ref{thm:main}. As in that proof, take $\varep_0 > 0$ small enough such that defining $\gamma_1 := (1/4)(d-\alpha) + \varep_0$, $\gamma_2 := (1/2)(d-1-\alpha) +\varep_0$, we have that $\gamma_1 \in [0, 1/2)$, $\gamma_2 \in [0, 1/4)$, and $\gamma_1 + \gamma_2 < 1/2$. Let $\rconn^1, \mbf{B}^1$ be constructed using $\rconn^0$ as in Sections \ref{section:linear-part} and \ref{section:nonlinear-part}. By Lemma \ref{lemma:main-useful-variant}, after a suitable modification of $\rconn^1, \mbf{B}^1$, we have that $\rconn^1(t) = e^{t\Delta} \rconn^0$, $\mbf{B}^1 = \rho(\rconn^1)$. 
Then by arguing as in the proof of Theorem \ref{thm:main}, we obtain a stochastic process $\rconn = (\rconn(t, x), t \in (0, 1), x \in \torus^d)$ by letting
\[ \rconn(t, x) := \ind(t < T) (\rconn^1(t, x) + \mbf{B}(t, x)),\]
where $T = \tau_{\gamma_1 \gamma_2}(R)$, $R = \max\{\|\rconn^1\|_{\mc{Q}_1^{\gamma_1}}, \|\mbf{B}^1\|_{\mc{Q}_1^{\gamma_2}}\}$, and $\mbf{B} \in \mc{Q}_{T, 3R}^{\gamma_2}$. Moreover, $\rconn$ is the solution to \eqref{eq:ZDDS} on $(0, T)$ constructed by Theorem \ref{thm:rough-distributional-zdds-local-existence}, using $\rconn^1, \mbf{B}^1$. We can then extend this to a stochastic process on $[0, 1) \times \torus^d$, by setting $\rconn(0, x) := \rconn^0(x)$ for $x \in \torus^d$. Then by Lemma \ref{lemma:zdds-smooth-initial-data-distributional-local-existence}, we obtain that $\rconn$ is the solution to \eqref{eq:ZDDS} on $[0, T)$ with initial data $\rconn(0) = \rconn^0$.

Next, by Propositions \ref{prop:linear-part-in-Q-space} and \ref{prop:B-1}, we have that $\E[R^p] \leq \const_{p, \varep_0}$. By Lemma \ref{lemma:tau-gamma-bound}, we have that $\E[T^{-p}] \leq \const_p + \const_p \E[R^{4p / (1 - 2\max(\gamma_1, \gamma_2))}]$. From this, we obtain $\E[T^{-p}] \leq \const_p$ for all $p \geq 1$, where $\const_p$ depends only on $p, d$, and the various constants in Assumptions \ref{assumption:l2-regularity}--\ref{assumption:four-product-assumption}.

For the final two inequalities, note that since $\mbf{B} \in \mc{Q}_{T, 3R}^{\gamma_2}$ and $\rconn^1 \in \mc{Q}_{T, R}^{\gamma_1}$, we have that $\|\rconn\|_{\mc{Q}_{T}^{\gamma_1}} \leq 4R$ (here we use that $\gamma_1 > \gamma_2$, which follows by the assumption that $\alpha > d - 4/3$). Thus recalling that $\E[R^p] \leq \const_{p, \varep_0}$, we have shown the last two inequalities for $\varep = \varep_0$, where $\varep_0$ is a small enough quantity that we fixed at the beginning. The last two inequalities for general $\varep > 0$ then follow, because due to the monotonicity in $\varep$, it just suffices to show the inequalities for small enough $\varep > 0$. 
\end{proof}

\section{Technical proofs}\label{section:technical-proofs}

We first show some general results which will be needed for both the linear and nonlinear parts. Recall the covariance function $\rho : (\torus^d)^2 \ra L(\lalg^d, \lalg^d)$ from Assumption \ref{assumption:translation-invariance}.
For notational simplicity, let $\tau : \torus^d \ra \R$ be defined by $\tau(x) := \Tr(\rho(x, 0))$. Since $\rho$ is integrable and translation invariant by Assumption \ref{assumption:translation-invariance}, we have that $\tau$ is also integrable. We will denote the Fourier coefficients of $\tau$ by $\hat{\tau}(n)$ for $n \in \Z^d$.


The following lemma shows that the translation invariance assumption leads to the Fourier coefficients being uncorrelated. 

\begin{lemma}\label{lemma:fourier-coeff-uncorrelated}
Suppose that Assumption \ref{assumption:translation-invariance} holds.  For any $n^1, n^2 \in \Z^d$, we have that
\[ \E[\langle \hat{\rconn}^0(n^1),  \hat{\rconn}^0(n^2)\rangle] = \ind(n^1 = n^2) \hat{\tau}(n^1).\]
Consequently, $\E[|\hat{\rconn}^0(n)|^2] = \hat{\tau}(n) \geq 0$. Additionally, for any $a_1, a_2 \in [\lalgdim]$, $j_1, j_2 \in [d]$, $n^1, n^2 \in \Z^d$, we have that
\[ |\E\big[\hat{\rconn}^{0, a_1}_{j_1}(n^1) \ovl{\hat{\rconn}^{0, a_2}_{j_2}(n^2)}\big]| \leq  \ind(n^1 = n^2) \hat{\tau}(n^1).\]
\end{lemma}
\begin{proof}
Let $K : \lalg^d \ra \lalg^d$ be a linear map. Note that we may extend $K : (\lalg^d)^\C \ra (\lalg^d)^\C$ by setting (for $a, b \in \lalg^d$) $K(a + \icomplex b) := Ka + \icomplex Kb$. Now by Assumption \ref{assumption:translation-invariance} and linearity, we have that for $\C$-valued test functions $\phi_1, \phi_2$,
\[ \E[\langle (\rconn^0, \phi_1), K (\rconn^0, \phi_2) \rangle] = \int_{\torus^d} \int_{\torus^d} \phi_1(x) \ovl{\phi_2(y)} \Tr(K \rho(x,y)^t) dx dy.\]
Here, the inner product on $(\lalg^d)^\C$ is given by (for $a_1, a_2, b_1, b_2 \in \lalg^d$)
\[\langle a_1 + \icomplex b_1, a_2 + \icomplex b_2 \rangle = \langle a_1, a_2 \rangle + \langle b_1, b_2 \rangle + \icomplex (\langle b_1, a_2 \rangle - \langle a_1, b_2 \rangle). \]
Applying this with $\phi_1 = e_{-n^1}$, $\phi_2 = e_{-n^2}$, we obtain
\[\begin{split}
\E[\langle\hat{\rconn}^0(n^1), K \hat{\rconn}^0(n^2) \rangle] &= \int_{\torus^d} \int_{\torus^d} e_{-n^1}(x) e_{n^2}(y) \Tr(K \rho(x, y)^t) dx dy \\ 
&= \int_{\torus^d} dy e_{n^2 - n^1}(y) \int_{\torus^d} dx e_{-n^1}(x-y) \Tr(K \rho(x-y, 0)^t) \\
&= \int_{\torus^d} dy e_{n^2 - n^1}(y) \int_{\torus^d} dx e_{-n^1}(x) \Tr(K\rho(x, 0)^t) \\
&= \ind(n^1 = n^2) \int_{\torus^d} dx \Tr(K \rho(x, 0)^t) e_{-n^1}(x).
\end{split}\]
By taking $K$ to be the identity map, we obtain the first claim.

For the second claim, given $a_1, a_2 \in [\lalgdim]$, $j_1, j_2 \in [d]$, first take $K : \lalg^d \ra \lalg^d$ to be the linear map with the property that for $\xi^1 = (\xi^1_j, j \in [d]), \xi^2 = (\xi^2_j, j \in [d]) \in \lalg^d$, we have that $\langle \xi^1, K \xi^2 \rangle = \xi^{1, a_1}_{j_1} \xi^{2, a_2}_{j_2}$ (here $\xi^i_j = \sum_{a \in [\lalgdim]} \xi^{i, a}_j S^a$ for $i \in \{1, 2\}$, $j \in [d]$). Observe then that the extension of $K$ to $(\lalg^d)^\C$ now has the property that $\langle \xi^1, K \xi^2 \rangle = \xi^{1, a_1}_{j_1} \ovl{\xi^{2, a_2}_{j_2}}$, where now $\xi^i \in (\lalg^d)^\C$, $\xi^i = (\xi^i_j, j \in [d])$, $\xi^i_j \in \lalg^\C$, $\xi^i_j = \sum_{a \in [\lalgdim]} \xi^{i, a}_j S^a$, $\xi^{i, a}_j \in \C$. Applying the above result with this $K$, we obtain that 
\[ \E\big[\hat{\rconn}^{0, a_1}_{j_1}(n^1) \ovl{\hat{\rconn}^{0, a_2}_{j_2}(n^2)}\big] \text{ is nonzero only if $n^1 = n^2$.}\]
On the other hand, using that $|z_1 z_2| \leq (1/2)|z_1|^2 + (1/2) |z_2|^2$,
we may obtain
\[ |\E\big[\hat{\rconn}^{0, a_1}_{j_1}(n^1) \ovl{\hat{\rconn}^{0, a_2}_{j_2}(n^2)}\big]| \leq (1/2) \E[|\hat{\rconn}^{0, a_1}_{j_1}(n^1)|^2] + (1/2) \E[\hat{\rconn}^{0, a_2}_{j_2}(n^2)|^2]. \]
Using that $|\hat{\rconn}^{0, a}_j(n)|^2 \leq |\hat{\rconn}^0(n)|^2$, combined with the first claim and the previous few observations, we may obtain
\[ |\E\big[\hat{\rconn}^{0, a_1}_{j_1}(n^1) \ovl{\hat{\rconn}^{0, a_2}_{j_2}(n^2)}\big]| \leq \ind(n^1 = n^2) \hat{\tau}(n^1), \]
as desired.
\end{proof}

The following few lemmas will be needed in Sections \ref{section:technical-proofs-linear-part} and \ref{section:technical-proofs-nonlinear-part}. 

\begin{lemma}\label{lemma:fractional-greens-function-heat-kernel}
Let $\alpha \in (0, d)$. For $t \in (0, 1]$, we have that
\[ \|e^{t \Delta} G_0^\alpha\|_{C^0} \leq \const_{d, \alpha} t^{-(1/2)(d-\alpha)}.\]
Here $\const_{d, \alpha}$ depends only on $d, \alpha$.
\end{lemma}
\begin{proof}
Recall from Lemma \ref{lemma:fractional-greens-function-properties} that the Fourier coefficients of $G_0^\alpha$ are $\hat{G}_0^\alpha(0) = 0$, $\hat{G}_0^\alpha(n) = |n|^{-\alpha}$, $n \in \Z^d \setminus \{0\}$. Thus we have that
\[\begin{split}
\|e^{t \Delta} G_0^\alpha\|_{C^0} &\leq \const \sum_{\substack{n \in \Z^d \\ n \neq 0}} e^{-4\pi^2 |n|^2 t} \frac{1}{|n|^\alpha} \leq \const + \const \sum_{r=1}^\infty r^{d - 1 -\alpha} e^{-4\pi^2 r^2 t}.
\end{split}\]
Now note that
\[\begin{split}
\sum_{r=1}^\infty r^{d - 1 -\alpha} e^{-4\pi^2 r^2 t} &\leq \const \int_1^\infty r^{d-1-\alpha} e^{-4\pi^2 r^2 t} dr  \\
&=C t^{-(1/2)(d-\alpha)} \int_{\sqrt{t}}^\infty u^{d-1-\alpha} e^{-4\pi^2 u^2} du. \end{split}\]
Using that $\alpha < d$, we have that $d - 1 - \alpha > -1$, and thus
\[ \int_{\sqrt{t}}^\infty u^{d-1-\alpha} e^{-4\pi^2 u^2} du \leq \int_0^1 u^{d-1-\alpha} du + \int_1^\infty e^{-4\pi^2 u^2} du \leq \const_{d, \alpha}.\]
The desired result now follows by combining the previous few estimates (and using that $t \in (0, 1]$ to bound $\const \leq \const t^{-(1/2)(d-\alpha)}$).
\end{proof}

\begin{lemma}\label{lemma:heat-semigroup-covariance}
Suppose that Assumption \ref{assumption:bounded-by-fractional-greens-function} holds for some $\alpha \in (0, d)$. For any $k \geq 0, t \in (0, 1]$, we have that
\[ \|e^{t \Delta} \tau\|_{C^k} \leq C_k t^{-(d+k-\alpha)/2}.\]
Here, $\const_k$ depends only on $k$, $d$, $\alpha$, and the constant $\constgf$ from Assumption \ref{assumption:bounded-by-fractional-greens-function}.
\end{lemma}
\begin{proof}
By Assumption \ref{assumption:bounded-by-fractional-greens-function} and the monotonicity property \eqref{eq:heat-semigroup-monotonicity}, we have that
\[ \|e^{t \Delta} \tau\|_{C^0} \leq \constgf \big(\|e^{t \Delta} G_0^\alpha\|_{C^0} + \constgf\big). \]
By Lemma \ref{lemma:fractional-greens-function-heat-kernel}, we obtain (here we use that $t \leq 1$)
\[ \|e^{t \Delta} \tau\|_{C^0} \leq \const t^{-(1/2)(d-\alpha)}, \]
which is the desired result for $k = 0$. Next, observe that for $k \geq 1$, $t \in (0, 1]$, we have that (applying \eqref{eq:heat-semigroup-Cr-C-r-plus-half} with $u = 1$ in the first inequality)
\[\begin{split}
\|e^{t \Delta} \tau\|_{C^k} &= \|e^{(t/2) \Delta} e^{(t/2) \Delta} \tau\|_{C^k} \\
&\leq \const (t/2)^{-1/2} \|e^{(t/2) \Delta} \tau^{a_1 a_2}_{j_1 j_2}\|_{C^{k-1}} \\
&\leq \const t^{-1/2} \|e^{(t/2) \Delta} \tau^{a_1 a_2}_{j_1 j_2}\|_{C^{k-1}}.
\end{split}\]
The desired result now follows by induction on $k$.
\end{proof}

\begin{definition}\label{def:dchain}
For $t \in (0, 1]$, define the metric $\dchain_t$ on $(t/2, t] \times \torus^d$ by
\[ \dchain_t((r, x), (s, y)) := \frac{|r-s|}{t} + \min\bigg\{\frac{\metricspace(x, y)}{\sqrt{t}}, 1\bigg\}.\]
For $\varep > 0$, let $N_{t, \varep}$ be the minimum number of $\varep$-balls needed to cover the metric space $((t/2, t] \times \torus^d, \dchain_t)$.
\end{definition}

\begin{lemma}\label{lemma:covering-number-bound}
For any $t \in (0, 1]$, $\varep > 0$, we have that
\[  N_{t, \varep} \leq \const t^{-d/2} \varep^{-(d+1)}.  \]
If $\varep \geq 3/2$, then we have that $N_{t, \varep} = 1$. 
\end{lemma}
\begin{proof}
First, note that the diameter of $((t/2, t] \times \torus^d, \dchain_t)$ is at most $3/2$, and thus the second claim follows. For the first claim, note that the metric space $(t/2, t]$ equipped with Euclidean distance may be covered by $O(\varep^{-1})$ balls of radius $(\varep / 2) t$. Let $\{x_i\}_{i \in [n]} \sse (t/2, t]$ be such  a cover. The metric space $(\torus^d, \metricspace)$ may be covered by $O((t^{1/2} \varep)^{-d})$ balls of radius $\sqrt{t} (\varep /2)$. Let $\{y_j\}_{j \in [m]} \sse \torus^d$ be such a cover. We then have that $\{(x_i, y_j)\}_{i \in [n], j \in [m]}$ is an $\varep$-cover of $((t/2, t] \times \torus^d, \dchain_t)$, and additionally $mn = O(t^{-d/2} \varep^{-(d+1)})$. The desired result now follows.
\end{proof}

\begin{lemma}\label{lemma:entropy-integral-bound}
For any $t \in (0, 1]$, $\beta > 0$, we have that
\[ \int_0^\infty (\log N_{t, \varep})^{\beta} d\varep \leq  \const_{\beta, d} (1 + |\log t|^{\beta}). \]
Here $\const_{\beta, d}$ depends only on $\beta, d$.
\end{lemma}
\begin{proof}
By Lemma \ref{lemma:covering-number-bound}, we may bound (using that $\beta \neq 0$)
\[\int_0^\infty (\log  N_{t, \varep} )^{\beta} d\varep  \leq \int_0^{3/2} \big(\log{(C t^{-d/2} \varep^{-(d+1)})} \big)^{\beta} d\varep.\]
Now, note that
\[\big(\log{(C t^{-d/2} \varep^{-(d+1)})}\big)^{\beta} \leq \const_{\beta, d} + \const_{\beta, d} |\log t|^{\beta} + \const_{\beta, d} |\log \varep|^{\beta}. \]
The desired result now follows by noting that (here using that $\beta \geq 0$)
\[ \int_0^{3/2} |\log \varep|^{\beta} d\varep \leq \const_\beta < \infty. \qedhere \]
\end{proof}

In what follows, given a (possibly vector-valued) random variable $X$, we will write $\|X\|_{L^2}$ as a shorthand for the $L^2$ norm of $X$, i.e., $\|X\|_{L^2} := (\E(|X|^2))^{1/2}$.

\subsection{Linear part}\label{section:technical-proofs-linear-part}

As in Section \ref{section:linear-part}, throughout this section, let $\rconn^0$ be a random $\lalg^d$-valued distribution satisfying Assumptions \ref{assumption:l2-regularity}, \ref{assumption:tail-bounds}, \ref{assumption:translation-invariance}, and \ref{assumption:bounded-by-fractional-greens-function}. We just assume that Assumption \ref{assumption:bounded-by-fractional-greens-function} holds for some $\alpha \in (0, d)$ --- that is, for this section, we do not need the restriction that $\alpha > \max\{d - 4/3, d/2\}$ which appears in Theorem \ref{thm:main}. These assumptions hold, even if they are not explicitly stated in the various lemmas or propositions. We first show the following result. Recall the definition of $\rconn^{1}_{N}$ from Definition \ref{def:A-1-N}.

\begin{lemma}\label{lemma:linear-part-regularity}
We have that a.s., for all $t > 0$, $k \geq 0$, 
\[ \sum_{n \in \Z^d} (1 + |n|)^k e^{-4\pi^2 |n|^2 t} |\hat{\rconn}^0(n)| < \infty.  \]
Consequently, we have that a.s., for all $t \in (0, 1]$, $x \in \torus^d$,
\[ \lim_{N \toinf} \rconn^{1}_{N}(t, x) = \sum_{n \in \Z^d} e^{-4\pi^2 |n|^2 t} \hat{\rconn}^{0}(n) e_n(x) \]
We also have that for all $t \in (0, 1]$, $x \in \torus^d$,
\[ \rconn^{1}(t, x) \stackrel{a.s.}{=} \sum_{n \in \Z^d} e^{-4\pi^2 |n|^2 t} \hat{\rconn}^{0}(n) e_n(x). \]
\end{lemma}
\begin{proof}
It suffices to prove the a.s.~result for fixed $t > 0$, since we can then obtain the a.s. result simultaneously for all $\{t_n\}_{n \geq 1}$, which is a sequence chosen so that $t_n \downarrow 0$. The a.s.~result simultaneously for all $t > 0$ then follows because of monotonicity in $t$. 
Towards this end, note that by the Cauchy--Schwarz inequality,
\[\begin{split}
\sum_{n \in \Z^d} &(1 + |n|)^{k} e^{-4\pi^2 |n|^2 t} |\hat{\rconn}^{0}(n)|  \leq \const \bigg(\sum_{n \in \Z^d} e^{-4\pi^2 |n|^2 t} |\hat{\rconn}^{0}(n)|^2 \bigg)^{1/2},
\end{split}\]
for some finite constant $\const$.
To finish, observe that (using Lemmas \ref{lemma:fourier-coeff-uncorrelated} and \ref{lemma:heat-semigroup-covariance})
\[ \E \bigg[\sum_{n \in \Z^d} e^{-4\pi^2 |n|^2 t} | \hat{\rconn}^{0}(n)|^2\bigg] = \sum_{n \in \Z^d} e^{-4\pi^2 |n|^2 t} \hat{\tau}(n) = (e^{t\Delta} \tau)(0) < \infty. 
\]
The a.s. convergence follows immediately from the first claim. The a.s. equality follows from the a.s. convergence and the fact that $\rconn^{1}_{N}(t, x) \stackrel{L^2}{\ra} \rconn^{1}(t, x)$ (recall Remark \ref{remark:A-1-N-convergence}).
\end{proof}

\begin{proof}[Proof of Lemma \ref{lemma:A-1-modification}]
Let $E$ be the event that for all $t > 0$, $k \geq 0$, we have that 
\[ \sum_{n \in \Z^d} (1 + |n|)^k e^{-4\pi^2 |n|^2 t} |\hat{\rconn}^{0}(n)| < \infty.  \]
Define the modification $\tilde{\rconn}^1 = (\tilde{\rconn}^{1}(t, x), t \in (0, 1], x \in \torus^d)$ as follows:
\[ \tilde{\rconn}^{1}(t, x) := \ind_E \sum_{n \in \Z^d} e^{-4\pi^2 |n|^2 t} \hat{\rconn}^{0}(n) e_n(x). \]
By Lemma \ref{lemma:linear-part-regularity}, we have that $\p(E) = 1$. It follows from this and the last claim of Lemma \ref{lemma:linear-part-regularity} that $\tilde{\rconn}^1$ is indeed a modification of $\rconn^1$. The fact that $\tilde{\rconn}^1$ is smooth and solves the heat equation follows by construction.
\end{proof}

\begin{remark}
From Lemma \ref{lemma:linear-part-regularity}, we can also ensure that (after a suitable modification) for all $k \geq 0$,
$t_0 \in (0, 1]$, 
\beq\label{eq:linear-part-positive-time} \sup_{t \in [t_0, 1]} \|\rconn^{1}(t)\|_{C^k} < \infty, \eeq
\beq\label{eq:linear-part-convergence-positive-time}\lim_{N \toinf} \sup_{t \in [t_0, 1]} \|\rconn^{1}_{N}(t) - \rconn^{1}(t)\|_{C^k} = 0. \eeq
\end{remark}




We next begin to work towards the proof of Proposition \ref{prop:linear-part-in-Q-space}.

\begin{lemma}\label{lemma:linear-part-moment-estimates}
For any $k \in \{0, 1, 2\}$, $l_1, \ldots, l_k \in [d]$, $t \in (0, 1], x \in \torus^d$, we have that
\beq\label{eq:A1-bound} \E[|\ptl_{l_1} \cdots \ptl_{l_k} \rconn^1(t, x)|^2] \leq \const t^{-(1/2)(d-\alpha) - k}. \eeq
The above inequalities are also true with $\rconn^{1}$ replaced by $\rconn^{1}_{N}$ for any $N \geq 0$. Here, $\const$ depends only on $d$ and the constants $\alpha, \constgf$ from Assumption \ref{assumption:bounded-by-fractional-greens-function}.
\end{lemma}
\begin{proof}
Let $N \geq 0$. We have that for some constant $C$,
\[ \ptl_{l_1} \cdots \ptl_{l_k} \rconn^1_N(t, x) = C \sum_{\substack{n \in \Z^d \\ |n|_\infty \leq N}} n_{l_1} \cdots n_{l_k} e^{-4\pi^2 |n|^2 t} \hat{\rconn}^0(n) e_n(x).\]
Thus by Lemma \ref{lemma:fourier-coeff-uncorrelated}, we obtain
\[ \E[|\ptl_{l_1} \cdots \ptl_{l_k} \rconn^1_N(t, x)|^2] \leq \const \sum_{\substack{n \in \Z^d \\ |n|_\infty \leq N}} |n|^{2k} e^{-8\pi^2 |n|^2 t} \hat{\tau}(n) \leq \const \|e^{2t\Delta} \tau\|_{C^{2k}}. \]
By Lemma \ref{lemma:heat-semigroup-covariance}, the right hand side is bounded by $\const t^{-(1/2)(d - \alpha) - k}$, where the constant $\const$ is uniform in $N$. To finish, we use Fatou's lemma, combined with equation \eqref{eq:linear-part-convergence-positive-time}.
\end{proof}

Let $t \in (0, 1]$, $x \in \torus^d$. For $N, M \geq 0$, let 
\[ \mbf{D}^{1}_{N}(t, x) := \rconn^{1}_{N}(t, x) - \rconn^{1}(t, x), ~~~\mbf{D}^{1}_{N, M}(t, x) := \rconn^{1}_{N}(t, x) - \rconn^{1}_{M}(t, x). \]


\begin{lemma}\label{lemma:linear-part-difference-moment-estimates}
There exists a sequence $\{\delta_N^{\ref{lemma:linear-part-difference-moment-estimates}}\}_{N \geq 0}$ of non-increasing functions $\delta_N^{\ref{lemma:linear-part-difference-moment-estimates}} : (0, 1] \ra [0, 1]$ such that for any $t \in (0, 1]$, $\lim_{N \toinf} \delta_N^{\ref{lemma:linear-part-difference-moment-estimates}}(t) = 0$, and 
for any $k \in \{0, 1, 2\}$, $l_1, \ldots, l_k \in [d]$, $N\ge 0$, $t \in (0, 1]$, $x \in \torus^d$, 
\[ \E[|\ptl_{l_1} \cdots \ptl_{l_k} \mbf{D}^1_N(t, x)|^2] \leq \const t^{-(1/2)(d-\alpha) - k} \delta_N^{\ref{lemma:linear-part-difference-moment-estimates}}(t). \]
Here, $\const$ depends only on $d$ and the constants $\alpha, \constgf$ from Assumption \ref{assumption:bounded-by-fractional-greens-function}.
\end{lemma}
\begin{proof}
It suffices to show the inequalities with $\mbf{D}^{1}_{N}$ replaced by $\mbf{D}^{1}_{N, M}$ for any $M \geq N$, since $\mbf{D}^{1}_{N, M} \ra \mbf{D}^{1}_{N}$ as $M \toinf$ (recall \eqref{eq:linear-part-convergence-positive-time}). 
For $N \geq 0$, $j \geq 0$, $t \in (0, 1]$, define
\[ S(N, j, t) := \sum_{\substack{n \in \Z^d \\ |n|_\infty > N}} |n|^j e^{-8\pi^2 |n|^2 t} \hat{\tau}(n).\]
Arguing as in the proof of Lemma \ref{lemma:linear-part-moment-estimates}, we can obtain that
\[ \sup_{M \geq N} \E\big[|\ptl_{l_1} \cdots \ptl_{l_k} \mbf{D}^{1}_{N, M}(t, x)|^2\big] \leq \const_k S(N, 2k, t). \]
Observe that for $m \in \{0, 2, 4\}$, we have that $S(N, m, t) \leq S(0, m, t)$, and moreover, from the proof of Lemma \ref{lemma:linear-part-moment-estimates}, we have that $S(0, m, t) \leq \const t^{-(1/2)(d+m - \alpha)}$. We may thus define 
\[ \delta_N^{\ref{lemma:linear-part-difference-moment-estimates}}(t_0) := \sup_{t \in [t_0, 1]} \max_{m \in \{0, 2, 4\}} \frac{S(N, m, t_0)}{\max\{1, S(0, m, t_0)\}}, ~~ t_0 \in (0, 1].\]
This ensures that $\delta_N^{\ref{lemma:linear-part-difference-moment-estimates}}$ is non-increasing, and that it maps into $[0, 1]$. To finish, we need to show that $\delta_N^{\ref{lemma:linear-part-difference-moment-estimates}}(t_0) \ra 0$. To show this, it suffices to show that for $m \in \{0, 2, 4\}$, we have that
\[ \lim_{N \toinf} \sup_{t \in [t_0, 1]} S(N, m, t_0) = 0.\]
Fix $m \in \{0, 2, 4\}$. Note that for any $N \geq 0$, $S(N, m, t)$ is non-increasing in $t$ (this follows since $\hat{\tau}(n) \geq 0$ for all $n \in \Z^d$ --- recall Lemma \ref{lemma:fourier-coeff-uncorrelated}), and thus it just suffices to show that $S(N, m, t_0) \ra 0$. This follows by the fact that $S(0, m, t_0) < \infty$, combined with the definition of $S(N, m, t_0)$ and dominated convergence.
\end{proof}

In what follows, recall the definition of $\dchain_t$ in Definition \ref{def:dchain}. 

\begin{lemma}\label{lemma:linear-part-entropy-estimate}
For any $t \in (0, 1]$, and any $r, s \in (t/2, t]$, $x, y \in \torus^d$, we have that
\[ \|\rconn^{1}(r, x) - \rconn^{1}(s, y)\|_{L^2} \leq \const t^{-(1/4)(d-\alpha)} \dchain_t((r, x), (s, y)).\]
The above inequality is also true with $\rconn^{1}$ replaced by $\rconn^{1}_{N}$ for any $N \geq 0$. Here, $\const$ depends only on $d$ and the constants $\alpha, \constgf$ from Assumption \ref{assumption:bounded-by-fractional-greens-function}.
\end{lemma}
\begin{proof}
We will show the result for $\rconn^1$. The proof for $\rconn^{1}_{N}$ for $N \geq 0$ will be essentially the same. We will show that
\[\begin{split}
\|\rconn^{1}(r, x) - \rconn^{1}(r, y)\|_{L^2} &\leq \const t^{-(1/4)(d-\alpha)} \min\bigg\{ \frac{\metricspace(x, y)}{\sqrt{t}}, 1\bigg\}, \\
\|\rconn^{1}(r, y) - \rconn^{1}(s, y)\|_{L^2} &\leq \const t^{-(1/4)(d-\alpha)} \frac{|r-s|}{t}. \end{split}\] 
Let $\wloop : [0, 1] \ra \torus^d$ be a geodesic from $y$ to $x$, so that $\wloop$ has constant speed $|\wloop'| = \metricspace(x, y)$. Then
\[ \rconn^{1}(r, x) - \rconn^{1}(r, y) = \int_0^1 \nabla \rconn^{1}(r, \wloop(u)) \cdot \wloop'(u) du. \]
We thus have
\[\begin{split}
|\rconn^{1}(r, x) - \rconn^{1}(r, y)|^2 &\leq \int_0^1 | \nabla \rconn^{1}(r, \wloop(u)) \cdot \wloop'(u)|^2  du  \\
&\leq \metricspace(x, y)^2 \int_0^1 | \nabla \rconn^{1}(r, \wloop(u))|^2 du. 
\end{split}\]
Taking expectations and applying \eqref{eq:A1-bound} gives 
\[ \|\rconn^{1}(r, x) - \rconn^{1}(r, y)\|_{L^2} \leq \const t^{-(1/4)(d-\alpha)} \frac{\metricspace(x, y)}{t^{1/2}}. \]
Combining this with \eqref{eq:A1-bound}, we obtain the first desired inequality.

For the second inequality, assume without loss of generality that $s < r$, and note that
\[ \rconn^{1}(r, y) - \rconn^{1}(s, y) = \int_s^r \ptl_u \rconn^{1}(u, y) du = \int_s^r \Delta \rconn^{1}(u, y) du. \]
We thus obtain
\[ |\rconn^{1}(r, y) - \rconn^{1}(s, y)|^2 \leq |r-s| \int_s^r |\Delta \rconn^{1}(u, y)|^2 du. \] 
Applying \eqref{eq:A1-bound}, we obtain the second desired inequality.
\end{proof}

The following result will allow us to show the convergence of $\rconn^1_N$ to $\rconn^1$ (recall the statement of Proposition \ref{prop:linear-part-in-Q-space}).

\begin{lemma}\label{lemma:linear-part-difference-entropy-bound}
There is a sequence $\{\delta_N^{\ref{lemma:linear-part-difference-entropy-bound}}\}_{N \geq 0}$ of functions $\delta_N^{\ref{lemma:linear-part-difference-entropy-bound}} : (0, 1] \ra [0, 1]$, such that the following hold. For any $t \in (0, 1]$, we have that $\lim_{N \toinf} \delta_N^{\ref{lemma:linear-part-difference-entropy-bound}}(t) = 0$. Also, for any $N \geq 0$, $t \in (0, 1]$, $r, s \in (t/2, t]$, $x, y \in \torus^d$, we have that
\[\begin{split} 
\|\mbf{D}^{1}_{N}(r, x)\|_{L^2} &\leq \const t^{-(1/4)(d-\alpha)} \delta_N^{\ref{lemma:linear-part-difference-entropy-bound}}(t), \\
\|\mbf{D}^{1}_{N}(r, x) - \mbf{D}^{1}_{N}(s, y)\|_{L^2} &\leq \const t^{-(1/4)(d-\alpha)} \dchain_t((r, x), (s, y)) \delta_N^{\ref{lemma:linear-part-difference-entropy-bound}}(t).
\end{split}\]
Here, $\const$ depends only on $d$ and the constants $\alpha, \constgf$ from Assumption \ref{assumption:bounded-by-fractional-greens-function}.
\end{lemma}
\begin{proof}
Let $\{\delta_N^{\ref{lemma:linear-part-difference-moment-estimates}}\}_{N \geq 0}$ be the sequence of functions from Lemma \ref{lemma:linear-part-difference-moment-estimates}. For $t \in (0, 1]$, let $\delta_N^{\ref{lemma:linear-part-difference-entropy-bound}}(t) := (\delta_N^{\ref{lemma:linear-part-difference-moment-estimates}}(t/2))^{1/2}$. The first inequality follows by Lemma~\ref{lemma:linear-part-difference-moment-estimates} and the fact that $\delta_N^{\ref{lemma:linear-part-difference-moment-estimates}}$ is non-increasing. The second inequality follows by the same argument as in the proof of Lemma \ref{lemma:linear-part-entropy-estimate}, where we use Lemma \ref{lemma:linear-part-difference-moment-estimates} in place of Lemma \ref{lemma:linear-part-moment-estimates}. In the course of the argument, we also use that $\delta_N^{\ref{lemma:linear-part-difference-moment-estimates}}$ is non-increasing, so that the value of this function at $u \in (t/2, t]$ is bounded by its value at $t/2$.
\end{proof}

By combining Lemmas \ref{lemma:linear-part-entropy-estimate} and \ref{lemma:linear-part-difference-entropy-bound} with Assumption \ref{assumption:tail-bounds}, we obtain the following result.

\begin{lemma}\label{lemma:A-1-tail-bounds}
For $t \in (0, 1]$, $r, s \in (t/2, t]$, $x, y \in \torus^d$, we have that for $u \geq 0$,
\[ \p(|\rconn^{1}(r, x)| > u) \leq \const \exp(- (u /t^{-(1/4)(d-\alpha)})^{\exptb}),\]
\[\begin{split}
\p(|\rconn^{1}(r, x) -  &\rconn^{1}(s, y)| > u) \leq  \\
&\const \exp\big(- \big(u / {(t^{-(1/4)(d-\alpha)} \dchain_t((r, x), (s, y)))}\big)^{\exptb} / C\big).
\end{split}\]
The above inequalities are also true with $\rconn^{1}$ replaced by $\rconn^{1}_{N}$ for any $N \geq 0$. Additionally, let $\{\delta_N^{\ref{lemma:linear-part-difference-entropy-bound}}\}_{N \geq 0}$ be the sequence of functions from Lemma \ref{lemma:linear-part-difference-entropy-bound}. Then for any $N \geq 0$, we have that for $u \geq 0$,
\[ \p(|\mbf{D}^{1}_{N}(r, x)| > u) \leq \const \exp\big(- (u / ({t^{-(1/4)(d-\alpha)} \delta_N^{\ref{lemma:linear-part-difference-entropy-bound}}(t)}))^{\exptb} / C\big),\]
\[ \begin{split}
\p(|\mbf{D}^{1}_{N}(r, x) &- \mbf{D}^{1}_{N}(s, y)| > u) \leq \\
&\const \exp\big(- (u / ({t^{-(1/4)(d-\alpha)} \dchain_t((r, x), (s, y)) \delta_N^{\ref{lemma:linear-part-difference-entropy-bound}}(t)}))^{\exptb} / C\big).
\end{split}\]
Here, $\const$ depends only on $d$ and the constants $\exptb, \consttb, \alpha, \constgf$ from Assumptions \ref{assumption:tail-bounds} and \ref{assumption:bounded-by-fractional-greens-function}.
\end{lemma}
\begin{proof}
We will prove the first two inequalities for $\rconn^{1}$. The proof for $\rconn^{1}_{N}$ for $N \geq 0$ will be essentially the same. Note that (recall \eqref{eq:heat-kernel-def} for the definition of $\Phi$)
\[ \rconn^{1}(r, x) \stackrel{a.s.}{=} (\rconn^{0}, \Phi(r, x - \cdot)). \]
The first inequality now follows by Assumption \ref{assumption:tail-bounds} and the fact that 
\[
\|\rconn^{1}(r, x)\|_{L^2} \leq \const t^{-(1/4)(d-\alpha)}
\]
(which holds by Lemma \ref{lemma:linear-part-moment-estimates}). Similarly, note that
\[ \begin{split}
\rconn^{1}(r, x) - \rconn^{1}(s, y) &\stackrel{a.s.}{=} (\rconn^{0}, \Phi(r, x - \cdot)) - (\rconn^{0}, \Phi(s, y - \cdot)) \\
&\stackrel{a.s.}{=} (\rconn^{0}, \Phi(r, x - \cdot) - \Phi(s, y - \cdot)). 
\end{split}\]
The second inequality now follows by Assumption \ref{assumption:tail-bounds} and Lemma \ref{lemma:linear-part-entropy-estimate}. For the last two inequalities, note that (recall Definition \ref{def:A-1-N})
\[ \rconn^{1}_{N}(r, x) \stackrel{a.s.}{=} (\FT_N \rconn^{0}, \Phi(r, x - \cdot)) = (\rconn^0, (\FT_N \Phi)(r, x - \cdot)) . \]
We then proceed as before, using Lemma \ref{lemma:linear-part-difference-entropy-bound} instead of Lemmas \ref{lemma:linear-part-moment-estimates} and \ref{lemma:linear-part-entropy-estimate}.
\end{proof}

The tail bounds from Lemma \ref{lemma:A-1-tail-bounds} allow us to obtain following result.

\begin{lemma}\label{lemma:A-1-exp-sup-local}
For any $p \geq 1$, $t_0 \in (0, 1]$, we have that
\[ \E\bigg[\sup_{\substack{t \in (t_0/2, t_0] \\ x \in \torus^d}} |\rconn^{1}(t, x)|^p\bigg] \leq \const t^{-p(1/4)(d-\alpha)} (1 + |\log t_0|^{1/\exptb})^p.  \]
The above also holds with $\rconn^{1}$ replaced by $\rconn^{1}_{N}$ for any $N \geq 0$. Additionally, let $\{\delta_N^{\ref{lemma:linear-part-difference-entropy-bound}}\}_{N \geq 0}$ be the sequence of functions from Lemma \ref{lemma:linear-part-difference-entropy-bound}. Then for any $N \geq 0$, we have that
\[\begin{split}
\E\bigg[\sup_{\substack{t \in (t_0/2, t_0] \\ x \in \torus^d}} |\rconn^{1}_{N}(t, x)& - \rconn^{1}(t, x)|^p\bigg] \leq \\
&\const t^{-p(1/4)(d-\alpha)} (1 + |\log t_0|^{1/\exptb})^p (\delta_N^{\ref{lemma:linear-part-difference-entropy-bound}}(t_0))^p. 
\end{split}\]
Here, $\const$ depends only on $d, p$ and the constants $\exptb, \consttb, \alpha, \constgf$ from Assumptions \ref{assumption:tail-bounds} and \ref{assumption:bounded-by-fractional-greens-function}.
\end{lemma}
\begin{proof}
Fix $p \geq 1$. Define the stochastic process $(X_{t, x}, (t, x) \in (t_0/2, t_0] \times \torus^d)$ by
\[ X_{t, x} := (\const^{-1/\exptb}) t^{(1/4)(d-\alpha)} \rconn^{1}(t, x), \]
where $\const$ is the constant from Lemma \ref{lemma:A-1-tail-bounds}. Then by Lemma \ref{lemma:A-1-tail-bounds}, we have that for $(r, x), (s, y) \in (t_0/2, t_0] \times \torus^d$, $u \geq 0$,
\[
\p(|X_{r, x}| > u) \leq \const \exp(-u^\exptb), \]
\[\p\big(|X_{r, x} - X_{s, y}| > u \dchain_t((r, x), (s, y))\big) \leq \const \exp(-u^\exptb). \]
By the first inequality and \cite[Lemma A.2]{Dirksen2015}, we have that 
\[ \sup_{(r, x) \in (t_0/2, t_0] \times \torus^d} \E\big[|X_{r, x}|^p] \leq \const. \]
Using this bound, together with the second inequality in the previous display and the tail bound from Lemma~\ref{lemma:entropy-integral-bound}, we can now apply Theorem \ref{thm:dirksen-exp-sup} to obtain the first desired result. The second desired result follows similarly.
\end{proof}

We can now finally prove Proposition \ref{prop:linear-part-in-Q-space}. 

\begin{proof}[Proof of Proposition \ref{prop:linear-part-in-Q-space}]
First, observe that for $t \in (0, 1]$, we have that (applying \eqref{eq:heat-semigroup-Cr-C-r-plus-half} with $u = 1$)
\[ \|\rconn^1(t)\|_{C^1} = \|e^{(t/2)\Delta} \rconn^1(t/2)\|_{C^1} \leq \const t^{-1/2} \|\rconn^1(t/2)\|_{C^0}.\]
It follows that
\[ \sup_{t \in (0, 1]} t^{(1/2) + \gamma_\varep} \|\rconn^1(t)\|_{C^1} \leq \const \sup_{t \in (0, 1]} t^{\gamma_\varep} \|\rconn^1(t)\|_{C^0}.\]
The same thing holds with $\rconn^1$ replaced by $\rconn^1_N$ for any $N \geq 0$. Thus for the first desired result, it suffices to show that
\[ \sup_{N \geq 0} \E\bigg[\sup_{t \in (0, 1]} t^{p\gamma_\varep} \|\rconn^1_N(t)\|_{C^0}^p \bigg], ~~ \E\bigg[\sup_{t \in (0, 1]} t^{p\gamma_\varep} \|\rconn^1(t)\|_{C^0}^p \bigg] < \infty. \]
Similarly, for the second desired result, it suffices to show that
\[ \lim_{N \toinf} \E \bigg[\sup_{t \in (0, 1]} t^{p\gamma_\varep} \|\rconn^1_N(t) - \rconn^1(t)\|_{C^0}^p \bigg] = 0. \]
The first result follows by combining Lemma \ref{lemma:A-1-exp-sup-local} with Lemma \ref{lemma:first-chaining-bound}. The second result follows by combining Lemma \ref{lemma:A-1-exp-sup-local} with Lemma \ref{lemma:chaining-bound-convergence-to-0}.
\end{proof}

\subsection{Nonlinear part}\label{section:technical-proofs-nonlinear-part} 

As in Section \ref{section:nonlinear-part}, we assume throughout this section that $\rconn^0$ is a random $\lalg^d$-valued distribution satisfying Assumptions \ref{assumption:l2-regularity}-\ref{assumption:four-product-assumption}.
For this section, we just assume that Assumption \ref{assumption:bounded-by-fractional-greens-function} holds with $\alpha \in (d/2, d)$.  These assumptions hold, even if they are not explicitly stated in the various lemmas, corollaries, or propositions. 

For many of the arguments in this section, we will work with the scalar quantities $\rconn^{2, a_0}_{N, j_0}$ instead of the vector quantity $\rconn^2_N$. Accordingly, recall the definitions of $I$ and $d$ in Definition \ref{def:I-d}, and recall that by Lemma~\ref{lemma:rho-A-1-smooth}, $\rconn^{2, a_0}_{N, j_0}$ has the following explicit form (in the following, $m = (n^1, n^2)$, $a = (a_0, a_1, a_2)$, $j = (j_0, j_1, j_2)$):
\beq\label{eq:A-2-N-explicit}\begin{split}
&\rconn^{2, a_0}_{N, j_0}(t_0, x_0) = \\
&\sum_{\substack{a_1, a_2 \in [\lalgdim] \\ j_1, j_2 \in [d]}} \sum_{\substack{n^1, n^2 \in \Z^d \\ |n^1|_\infty, |n^2|_\infty \leq N}} I(m, t_0) d(m, a, j) \hat{\rconn}^{0, a_1}_{N, j_1}(n^1) \hat{\rconn}^{0, a_2}_{N, j_2}(n^2) e_{n^1 + n^2}(x_0). \end{split}\eeq

\begin{definition}
For $N \geq 0$, $k \geq 0$, $t \geq 0$, define
\[\begin{split}
S(N&, k, t) := \\
&\sum_{\substack{n^1, n^2 \in \Z^d \\ \max\{|n^1|_\infty, |n^2|_\infty\} \geq N}} |n^1 + n^2|^k I((n^1, n^2), t)^2 (|n^1|^2 + |n^2|^2) \hat{\tau}(n^1) \hat{\tau}(n^2).
\end{split}\]
\end{definition}

We now state the following technical lemma, which is one of the key intermediate steps for proving the results of Section \ref{section:nonlinear-part}. The proof is in Appendix \ref{appendix:tau1-tau2-bound}.

\begin{lemma}\label{lemma:tau-1-tau-2-bound}
For all $k \geq 0$, there exists a sequence $\{\delta_{N, k}^{\ref{lemma:tau-1-tau-2-bound}}\}_{N \geq 0}$ of non-increasing functions $\delta_{N, k}^{\ref{lemma:tau-1-tau-2-bound}} : (0, 1] \ra [0, 1]$ such that the following hold. For any $t \in (0, 1]$, we have that $\lim_{N \toinf} \delta_{N, k}^{\ref{lemma:tau-1-tau-2-bound}}(t) = 0$. Also, for any $N \geq 0$, 
$t \in (0, 1]$, we have that
\[ S(N, k, t) \leq \const_k t^{-(d-1 + (k/2)-\alpha)} \delta_{N, k}^{\ref{lemma:tau-1-tau-2-bound}}(t).\]
Here, $\const_k$ depends only on $k, d$, and the constants $\alpha, \constgf$ from Assumption \ref{assumption:bounded-by-fractional-greens-function}.
\end{lemma}

We proceed to use Lemma \ref{lemma:tau-1-tau-2-bound} to obtain moment bounds on $\rconn^2_N$.

\begin{lemma}\label{lemma:A-2-bounds}
For any $k \in \{0, 1, 2, 3\}$, $l_1, \ldots, l_k \in [d]$, $N \geq 0$, 
$t \in (0, 1]$, $x \in \torus^d$, we have that
\[ \E\big[|\ptl_{l_1 \cdots l_k} \rconn^{2}_{N}(t, x)|^2\big] \leq \const t^{-(d + k - 1-\alpha)}. \]
Here, $\const$ depends only on $d$ and the constants $\alpha, \constgf, \constfourp$ from Assumptions \ref{assumption:bounded-by-fractional-greens-function}, \ref{assumption:four-product-assumption}.
\end{lemma}
\begin{proof}
We will work with the scalar quantities $\mbf{A}^{2, a_0}_{N, j_0}$. Fix $a = (a_0, a_1, a_2)$, $j = (j_0, j_1, j_2)$. Define
\[ \mbf{B}(t, x) := \sum_{\substack{n^1, n^2 \in \Z^d \\ |n_1|_\infty, |n_2|_\infty \leq N}} e_{n^1 + n^2}(x) I(m, t) d(m, a, j) \hat{\rconn}^{0, a_1}_{j_1}(n^1) \hat{\rconn}^{0, a_2}_{j_2}(n^2). \]
We first look at the $k = 0$ case. For this, it suffices to show that $\E[|\mbf{B}(t, x)|^2] \leq \const t^{-(d - 1 -\alpha)}$. Towards this end, note that
\[\begin{split}
\E[|\mbf{B}(t, x)|^2] &= \sum_{\substack{n^1, n^2 \in \Z^d \\ k^1, k^2 \in \Z^d \\ |n^i|_\infty, |k^i|_\infty \leq N \\ i = 1, 2}} I((n^1, n^2), t) I((k^1, k^2), t) d((n^1, n^2), a, j) ~\times \\
&\qquad \qquad \qquad \qquad \ovl{d((k^1, k^2), a, j)} e_{n^1 + n^1 - (k^1 + k^2)}(x) ~\times \\
&\qquad \qquad \qquad \qquad \E\big[\hat{\rconn}^{0, a_1}_{j_1}(n^1) \hat{\rconn}^{0, a_2}_{j_2}(n^2) \ovl{\hat{\rconn}^{0, a_1}_{j_1}(k^1)} \ovl{\hat{\rconn}^{0, a_2}_{j_2}(k^2)}\big].
\end{split}\]
We have that (recall Remark \ref{remark:I-and-d-properties}) 
\[ |d((n^1, n^2), a, j)| \leq \const (|n^1| + |n^2|). \]
Also by Remark \ref{remark:I-and-d-properties}, we can assume $a_1 \neq a_2$, since otherwise $d((n^1, n^2), a, j) = 0$.
Thus by Assumption \ref{assumption:four-product-assumption}, we have that
\[\begin{split}
\big| \E\big[\hat{\rconn}^{0, a_1}_{j_1}(n^1)& \hat{\rconn}^{0, a_2}_{j_2}(n^2) \ovl{\hat{\rconn}^{0, a_1}_{j_1}(k^1)} \ovl{\hat{\rconn}^{0, a_2}_{j_2}(k^2)}\big]\big| \leq \constfourp \big(|E_1 E_2| + |E_3 E_4| \big) \end{split}\]
where
\[ E_1 := \E\big[\hat{\rconn}^{0, a_1}_{j_1}(n^1) \ovl{\hat{\rconn}^{0, a_1}_{j_1}(k^1)}\big] ,~~ E_2 := \E \big[\hat{\rconn}^{0, a_2}_{j_2}(n^2) \ovl{\hat{\rconn}^{0, a_2}_{j_2}(k^2)}\big], \]
\[ E_3 := \E\big[\hat{\rconn}^{0, a_1}_{j_1}(n^1) \ovl{\hat{\rconn}^{0, a_2}_{j_2}(k^2)}\big], ~~ E_4 := \E\big[\hat{\rconn}^{0, a_2}_{j_2}(n^2) \ovl{\hat{\rconn}^{0, a_1}_{j_1}(k^1)}\big].\]
By Lemma \ref{lemma:fourier-coeff-uncorrelated}, we have that 
\[
|E_1| \leq  \ind(n^1 = k^1) \hat{\tau}(n^1),  ~~ |E_2| \leq \ind(n^2 = k^2) \hat{\tau}(n^2), \]
\[ |E_3| \leq \ind(n^1 = k^2) \hat{\tau}(n^1), ~~~ |E_4| \leq \ind(n^2 = k^1) \hat{\tau}(n^2). \]
Combining these bounds, we see that it suffices to bound (note that $I((n^1, n^2), t) = I((n^2, n^1), t)$)
\[ \sum_{n^1, n^2 \in \Z^d} I((n^1, n^2), t)^2  (|n^1|^2 + |n^2|^2) \hat{\tau}(n^1) \hat{\tau}(n^2). \]
The $k = 0$ case now follows by Lemma \ref{lemma:tau-1-tau-2-bound} with $N = 0, k = 0$.

For the case $k \in \{1, 2, 3\}$, note that following the same steps as before, we may reduce to bounding
\[ \sum_{n^1, n^2 \in \Z^d} |n^1 + n^2|^{2k} I((n^1, n^2), t)^2  (|n^1|^2 + |n^2|^2) \hat{\tau}(n^1) \hat{\tau}(n^2). \]
By Lemma \ref{lemma:tau-1-tau-2-bound}, this is bounded by $\const t^{-(d -1 + k -\alpha)}$, as desired.
\end{proof}

\begin{definition}\label{def:D-N-M}
For $N, M \geq 0$, 
$t \in (0, 1]$, $x \in \torus^d$, let 
\[ \mbf{D}^{2}_{N, M}(t, x) := \rconn^{2}_{N}(t, x) - \rconn^{2}_{M}(t, x). \]
\end{definition}

\begin{lemma}\label{lemma:nonlinear-part-derivative-difference-moment-bounds}
There exists a sequence $\{\delta_N^{\ref{lemma:nonlinear-part-derivative-difference-moment-bounds}}\}_{N \geq 0}$ of non-increasing functions $\delta_N^{\ref{lemma:nonlinear-part-derivative-difference-moment-bounds}} : (0, 1] \ra [0, 1]$ such that the following hold. For any $t \in (0, 1]$, we have that $\lim_{N \toinf} \delta_N^{\ref{lemma:nonlinear-part-derivative-difference-moment-bounds}}(t) = 0$. Also, for any $k \in \{0, 1, 2, 3\}$, $l_1, \ldots, l_k \in [d]$, $M \geq N \geq 0$, 
$t \in (0, 1]$, $x \in \torus^d$, we have that
\[ \E\big[|\ptl_{l_1 \cdots l_k} \mbf{D}^{2}_{N, M}(t, x)|^2\big] \leq \const t^{-(d + k - 1-\alpha)} \delta_N^{\ref{lemma:nonlinear-part-derivative-difference-moment-bounds}}(t). \]
Here, $\const$ depends only on $d$ and the constants $\alpha, \constgf, \constfourp$ from Assumptions \ref{assumption:bounded-by-fractional-greens-function}, \ref{assumption:four-product-assumption}.
\end{lemma}
\begin{proof}
For $m \geq 0$, let $\{\delta_{N, m}^{\ref{lemma:tau-1-tau-2-bound}}\}_{N \geq 0}$ be the sequence of functions from Lemma \ref{lemma:tau-1-tau-2-bound}. Define
\[ \delta_N^{\ref{lemma:nonlinear-part-derivative-difference-moment-bounds}}(t) := \max_{m \in \{0, 2, 4, 6\}} \delta_{N, m}^{\ref{lemma:tau-1-tau-2-bound}}(t), ~~ t \in (0, 1].\]
We first look at the $k = 0$ case. We will work with the scalar quantities $\mbf{D}^{2, a_0}_{N, M, j_0}$. By arguing as in the proof of Lemma \ref{lemma:A-2-bounds}, we may bound
\begin{align*}
&\E\big[|\mbf{D}^{2, a_0}_{N, M j_0}(t, x)|^2\big]  \\
&\leq \const \sum_{\substack{a_1, a_2 \in [\lalgdim] \\ j_1, j_2 \in [d]}} \sum_{\tau_1, \tau_2 \in \{\tau^{a_1}_{j_1}, \tau^{a_2}_{j_2}\}} \sum_{\substack{n^1, n^2 \in \Z^d \\ N < \max\{|n^1|_\infty, |n^2|_\infty\} \leq M}} I((n^1, n^2), t)^2 ~\times\\
&\qquad \qquad \qquad \qquad \qquad \qquad \qquad \qquad \qquad  (|n^1|^2 + |n^2|^2) \hat{\tau}(n^1) \hat{\tau}(n^2).
\end{align*}
Note that we may obtain a further upper bound by replacing the sum over $n^1, n^2 \in \Z^d$ such that $N < \max\{|n^1|_\infty, |n^2|_\infty\} \leq M$ by a sum over $n^1, n^2 \in \Z^d$ such that $\max\{|n^1|_\infty, |n^2|_\infty\} > N$. The $k = 0$ case now follows by Lemma \ref{lemma:tau-1-tau-2-bound}. The case $k \in \{1, 2, 3\}$ may be argued similarly.
\end{proof}

\begin{proof}[Proof of Lemma \ref{lemma:B-1-cauchy}]
This is now a direct consequence of Lemma \ref{lemma:nonlinear-part-derivative-difference-moment-bounds}.
\end{proof}



We next prove various technical lemmas which will help in obtaining moment bounds on quantities such as $\rconn^2_N(t, x) - \rconn^2_N(s, y)$. 

\begin{definition}
Let $N \geq 0$, $a_0 \in [\lalgdim]$, $j_0 \in [d]$, $t \in (0, 1]$, $x \in \torus^d$. Define
\[ \begin{split}
\mbf{F}^{a_0}_{N, j_0}(t, x) := \sum_{\substack{a_1, a_2 \in [\lalgdim] \\ j_1, j_2 \in [d]}}&\sum_{\substack{n^1, n^2 \in \Z^d \\ |n_1|_\infty, |n_2|_\infty \leq N}} d(m, a, j) ~\times \\
&e^{-4\pi^2 |n^1|^2 t} \hat{\rconn}^{0, a_1}_{j_1}(n^1) e^{-4\pi^2 |n^2|^2 t}\hat{\rconn}^{0, a_2}_{j_2}(n^2) e_{n^1 + n^2}(x) ,
\end{split}\]
where $m = (n_1,n_2)$. Using the collection of $\R$-valued process $(\mbf{F}^a_{N, j}, a \in [\lalgdim], j \in [d])$ (as well as the relation \eqref{eq:one-form-R-valued-functions-relation}, we may define the $\lalg^d$-valued process $\mbf{F}_N = (\mbf{F}_N(t, x), t \in (0, 1], x \in \torus^d)$.
\end{definition}

By Lemma \ref{lemma:rho-A-1-smooth}, we have that for any $N \geq 0$, 
\beq\label{eq:A2-time-derivative} \ptl_t \rconn^{2}_{N}(t, x) = \Delta \rconn^{2}_{N}(t, x) + \mbf{F}_{N}(t, x), ~~ t \in (0, 1], x \in \torus^d. \eeq

\begin{lemma}\label{lemma:F-bound}
For any $N \geq 0$, 
$t \in (0, 1]$, $x \in \torus^d$, $l \in [d]$, we have that
\[ \E\big[|\mbf{F}_{N}(t, x)|^2\big] \leq \const t^{-(d+1-\alpha)}, ~~~ \E[|\ptl_l \mbf{F}_{N}(t, x)|^2] \leq \const t^{-(d+2 - \alpha)}.\]
Here, $\const$ depends only on $d$ and the constants $\alpha, \constgf, \constfourp$ from Assumptions \ref{assumption:bounded-by-fractional-greens-function}, \ref{assumption:four-product-assumption}.
\end{lemma}
\begin{proof}
We will work with the scalar quantities $\mbf{F}^{a_0}_{N, j_0}$. Fix $a = (a_0, a_1, a_2)$, $j = (j_0, j_1, j_2)$. Define
\[\begin{split}
\mbf{B}(t, x) := \sum_{\substack{n^1, n^2 \in \Z^d \\ |n_1|_\infty, |n_2|_\infty \leq N}} &e_{n^1 + n^2}(x) d(m, a, j) ~\times \\
&e^{-4\pi^2 |n^1|^2 t} \hat{\rconn}^{0, a_1}_{j_1}(n^1) e^{-4\pi^2 |n^2|^2 t}\hat{\rconn}^{0, a_2}_{j_2}(n^2),
\end{split}\]
where $m = (n^1, n^2)$, as usual. 
For the first inequality, it suffices to show that $\E\big[|\mbf{B}(t, x)|^2\big] \leq \const t^{-(d+1-\alpha)}$. Towards this end, we have that
\begin{align*}
&\E\big[|\mbf{B}(t, x)|^2\big] = \sum_{\substack{n^1, n^2 \in \Z^d \\ k^1, k^2 \in \Z^d \\ |n^i|_\infty, |k^i|_\infty \leq N \\ i = 1, 2}} e_{n^1 + n^2 - (k^1 + k^2)}(x) 
d((n^1, n^2), a, j) \ovl{d((k^1, k^2), a, j)} ~\times \\
& e^{-4\pi^2(|n^1|^2 + |n^2|^2 + |k^1|^2 + |k^2|^2) t} 
\E\big[\hat{\rconn}^{0, a_1}_{j_1}(n^1) \hat{\rconn}^{0, a_2}_{j_2}(n^2) \ovl{\hat{\rconn}^{0, a_1}_{j_1}(k^1)} \ovl{\hat{\rconn}^{0, a_2}_{j_2}(k^2)}\big].
\end{align*}
Arguing as in the proof of Lemma \ref{lemma:A-2-bounds}, we may reduce to bounding
\[ \sum_{n^1, n^2 \in \Z^d} (|n^1|^2 + |n^2|^2) e^{-8\pi^2 |n^1|^2 t} \hat{\tau}(n^1) e^{-8\pi^2 |n^2|^2 t} \hat{\tau}(n^2),  \]
Using that
\[ |n|^2 e^{-4\pi^2 |n|^2 t} \leq \sup_{x \geq 0} x e^{-4\pi^2 x t} \leq \const t^{-1}, \] 
we obtain the further upper bound
\[ \const t^{-1} \sum_{n^1 \in \Z^d} e^{-4\pi^1 |n^1|^2 t} \hat{\tau}(n^1) \sum_{n^2 \in \Z^d} e^{-4\pi^2 |n^2|^2 t} \hat{\tau}(n^2).  \]
Observe that the above is equal to
\[ \const t^{-1} (e^{t \Delta} \tau)(0) (e^{t \Delta} \tau)(0). \]
Using that $\|e^{t \Delta} \tau\|_{C^0}, \|e^{t \Delta} \tau\|_{C^0} \leq \const t^{-(1/2)(d-\alpha)}$ (by Lemma \ref{lemma:heat-semigroup-covariance}), the first desired result now follows.

For the second inequality, by arguing as before, we may reduce to bounding
\[ \sum_{n^1, n^2 \in \Z^d} |n^1 + n^2|^2 (|n^1|^2 + |n^2|^2) e^{-8\pi^2 |n^1|^2 t} \hat{\tau}(n^1) e^{-8\pi^2 |n^2|^2 t} \hat{\tau}(n^2). \]
We may bound
\[ |n^1 + n^2|^2 (|n^1|^2 + |n^2|^2) \leq \const (|n^1|^4 + |n^2|^4), \]
and so arguing as before, we obtain the further upper bound
\[ \const t^{-2} (e^{t \Delta} \tau)(0) (e^{t\Delta} \tau)(0), \]
which is bounded by $\const t^{-2} t^{-(d-\alpha)}$, as desired.
\end{proof}

\begin{lemma}\label{lemma:F-bound-difference}
There exists a sequence $\{\delta_N^{\ref{lemma:F-bound-difference}}\}_{N \geq 0}$ of non-increasing functions $\delta_N^{\ref{lemma:F-bound-difference}} : (0, 1] \ra [0, 1]$, such that the following hold. For any $t \in (0, 1]$, we have that $\lim_{N \toinf} \delta_N^{\ref{lemma:F-bound-difference}}(t) = 0$. Also, for any $M \geq N \geq 0$, 
$t \in (0, 1]$, $x \in \torus^d$, we have that
\[ \E\big[|\mbf{F}_{N}(t, x) - \mbf{F}_{M}(t, x)|^2\big] \leq \const t^{-(d+1-\alpha)} \delta_N^{\ref{lemma:F-bound-difference}}(t).\]
Additionally, for any $l \in [d]$, we have that
\[ \E\big[|\ptl_l \mbf{F}_{N}(t, x) - \ptl_l \mbf{F}_{M}(t, x)|^2\big] \leq \const t^{-(d+2 - \alpha)}\delta_N^{\ref{lemma:F-bound-difference}}(t).\] 
Here, $\const$ depends only on $d$ and the constants $\alpha, \constgf, \constfourp$ from Assumptions \ref{assumption:bounded-by-fractional-greens-function}, \ref{assumption:four-product-assumption}. 
\end{lemma}
\begin{proof}
We will work with the scalar quantities $\mbf{F}^{a_0}_{N, j_0}$. For $N \geq 0$, $t \in (0, 1]$, define
\[ G_N(t) := \sum_{\substack{a_1, a_2 \in [\lalgdim] \\ j_1, j_2 \in [d]}}\sum_{\substack{n^1, n^2 \in \Z^d \\ \max\{|n^1|_\infty, |n^2|_\infty\} \geq N}} e^{-4\pi^2 |n^1|^2 t} \hat{\tau}(n^1) e^{-4\pi^2 |n^2|^2 t} \hat{\tau}(n^2).\]
By arguing as in the proof of Lemma \ref{lemma:F-bound}, we may obtain
\[\begin{split} \E\big[|\mbf{F}^{a_0}_{N, j_0}(t, x) - \mbf{F}^{a_0}_{M, j_0}(t, x)|^2\big] &\leq \const t^{-1} G_N(t), \\
\E\big[|\ptl_l \mbf{F}^{a_0}_{N, j_0}(t, x) - \ptl_l \mbf{F}^{a_0}_{M, j_0}(t, x)|^2\big] &\leq \const t^{-2} G_N(t). \end{split}\]
Observe that $G_N(t) \leq G_0(t)$ for all $t \in (0, 1]$, and from the proof of Lemma \ref{lemma:F-bound}, we have that $G_0(t) \leq \const t^{-(d-\alpha)}$. Thus we may define 
\[ \delta_N^{\ref{lemma:F-bound-difference}}(t_0) := \sup_{t \in [t_0, 1]} \frac{G_N(t)}{\max\{1, G_0(t)\}}, ~~ t_0 \in (0, 1].\]
This ensures that $\delta_N^{\ref{lemma:F-bound-difference}}$ maps into $[0, 1]$, and that it is non-increasing. It remains to show that $\lim_{N \toinf} \delta_N^{\ref{lemma:F-bound-difference}}(t_0) = 0$ for all $t_0 \in (0, 1]$. Fix $t_0 \in (0, 1]$, $a_1, a_2 \in [\lalgdim]$, $j_1, j_2 \in [d]$. Since $\max\{|n^1|_\infty, |n^2|_\infty\} \geq N$ implies that at least one of $|n^1|_\infty, |n^2|_\infty$ is at least $N$, it suffices to show that
\[ \lim_{N \toinf} \sup_{t \in [t_0, 1]} \sum_{\substack{n^1, n^2 \in \Z^d \\ |n^1|_\infty > N}} e^{-4\pi^2 |n^1|^2 t} \hat{\tau}(n^1) e^{-4\pi^2 |n^2|^2 t} \hat{\tau}(n^2) = 0. \]
Note that without the limit, the left hand side above can be bounded by
\[ \sum_{\substack{n^1 \in \Z^d \\ |n^1|_\infty > N}}e^{-4\pi^2 |n^1|^2 t_0} \hat{\tau}(n^1) \sum_{n^2 \in \Z^d} e^{-4\pi^2 |n^2|^2 t_0} \hat{\tau}(n^2) . \]
Note that the second sum is $(e^{t_0 \Delta} \tau)(0)$, which is finite by Lemma \ref{lemma:heat-semigroup-covariance} (and the fact that $t_0 > 0$). For the same reason, we have that 
\[ \sum_{\substack{n^1 \in \Z^d}}e^{-4\pi^2 |n^1|^2 t_0} \hat{\tau}(n^1) < \infty, \] 
and thus by dominated convergence, we have that
\[ \lim_{N \toinf} \sum_{\substack{n^1 \in \Z^d \\ |n^1|_\infty > N}}e^{-4\pi^2 |n^1|^2 t_0} \hat{\tau}(n^1) = 0. \]
The desired result now follows.
\end{proof}

We next use the previous technical lemmas to control the $C^0$ norm of $\rconn^2_N(t)$, culminating in Proposition \ref{prop:nonlinear-part-c0-bound} below. After we control the $C^0$ norm, we will then move on to controlling the $C^1$ norm. In the following, recall the definition of $\dchain_t$ from Definition \ref{def:dchain}. 

\begin{lemma}\label{lemma:nonlinear-part-finite-N-bounds}
For any $N \geq 0$, the following holds. For 
$t \in (0, 1]$, $r, s \in (t/2, t]$, $x, y \in \torus^d$, we have that
\[\begin{split}
\|\rconn^{2}_{N}(t, x)\|_{L^2} &\leq \const t^{-(1/2)(d-1-\alpha)}, \\
\|\rconn^{2}_{N}(r, x) - \rconn^{2}_{N}(s, y)\|_{L^2} &\leq \const t^{-(1/2)(d-1-\alpha)} \dchain_t((r, x), (s, y)).
\end{split}\]
Here, $\const$ depends only on $d$ and the constants $\alpha, \constgf, \constfourp$ from Assumptions \ref{assumption:bounded-by-fractional-greens-function}, \ref{assumption:four-product-assumption}. 
\end{lemma}
\begin{proof}
The first inequality follows by Lemma \ref{lemma:A-2-bounds}. To show the second inequality, we will show that
\[\begin{split}
\|\rconn^{2}_{N}(r, x) - \rconn^{2}_{N}(r, y)\|_{L^2} &\leq \const t^{-(1/2)(d-1-\alpha)} \min \bigg\{ \frac{\metricspace(x, y)}{\sqrt{t}}, 1\bigg\}, \\
\|\rconn^{2}_{N}(r, y) - \rconn^{2}_{N}(s, y)\|_{L^2} &\leq \const t^{-(1/2)(d-1-\alpha)} \frac{|r-s|}{t}.
\end{split}\] 
Let $\wloop : [0, 1] \ra \torus^d$ be a geodesic from $y$ to $x$, so that $\wloop$ has constant speed $|\wloop'| \equiv \metricspace(x, y)$. Then
\[ \rconn^{2}_{N}(r, x) - \rconn^{2}_{N}(r, y) = \int_0^1 \nabla \rconn^{2}_{N}(r, \wloop(u)) \cdot \wloop'(u) du. \]
We thus have
\[\begin{split}
|\rconn^{2}_{N}(r, x) - \rconn^{2}_{N}(r, y)|^2 &\leq \int_0^1 |\nabla \rconn^{2}_{N}(r, \wloop(u)) \cdot \wloop'(u)|^2 du \\
&\leq \metricspace(x, y)^2 \int_0^1 |\nabla \rconn^{2}_{N}(r, \wloop(u))|^2 du. 
\end{split}\]
Applying  Lemma \ref{lemma:A-2-bounds}, we obtain
\[ \|\rconn^{2, a}_{N, j}(r, x) - \rconn^{2, a}_{N, j}(r, y)\|_{L^2} \leq \const t^{-(1/2)(d-1-\alpha)} t^{-1/2}\metricspace(x, y).\]
Combining this with the first inequality, the first claim follows.

For the second claim, assume without loss of generality that $s < r$. Recalling \eqref{eq:A2-time-derivative}, we have that
\[\begin{split}
\rconn^{2}_{N}(r, y) - \rconn^{2}_{N}(s, y) &= \int_s^r \ptl_u \rconn^{2}_{N}(u, y) du  \\
&= \int_s^r \Delta \rconn^{2}_{N}(u, y) du + \int_s^r \mbf{F}_{N}(u, y) du. \end{split}\]
We thus have that
\begin{align*}
&|\rconn^{2}_{N}(r, y) - \rconn^{2}_{N}(s, y)|^2 \leq \const (r-s) \int_s^r \big(|\Delta \rconn^{2}_{N}(u, y)|^2 + |\mbf{F}_{N}(u, y)|^2\big)du.
\end{align*}
By Lemmas \ref{lemma:A-2-bounds} and \ref{lemma:F-bound}, we obtain
\[ \|\rconn^{2}_{N}(r, y) - \rconn^{2}_{N}(s, y)\|_{L^2} \leq \const t^{-(1/2)(d-1-\alpha)}\frac{|r-s|}{t}, \]
as desired.
\end{proof}

In the following, recall the definition of $\mbf{D}^{2}_{N, M}$ from Definition \ref{def:D-N-M}. 

\begin{lemma}\label{lemma:nonlinear-part-finite-N-difference-bound}
There is a sequence $\{\delta_N^{\ref{lemma:nonlinear-part-finite-N-difference-bound}}\}_{N \geq 0}$ of functions $\delta_N^{\ref{lemma:nonlinear-part-finite-N-difference-bound}} : (0, 1] \ra [0, 1]$ such that the following hold. 
For any $t \in (0, 1]$, we have that $\lim_{N \toinf} \delta_N^{\ref{lemma:nonlinear-part-finite-N-difference-bound}}(t) = 0$. Also, for any $M \geq N \geq 0$,
$t \in (0, 1]$, $r, s \in (t/2, t]$, $x, y \in \torus^d$, we have that
\[ \|\mbf{D}^{2}_{N, M}(r, x)\|_{L^2} \leq \const t^{-(1/2)(d-1-\alpha)} \delta_N^{\ref{lemma:nonlinear-part-finite-N-difference-bound}}(t),\]
\[ \|\mbf{D}^{2}_{N, M}(r, x) - \mbf{D}^{2}_{N, M}(s, y)\|_{L^2} \leq \const t^{-(1/2)(d-1-\alpha)} \dchain_t((r, x), (s, y)) \delta_N^{\ref{lemma:nonlinear-part-finite-N-difference-bound}}(t). \]
Here, $\const$ depends only on $d$ and the constants $\alpha, \constgf, \constfourp$ from Assumptions \ref{assumption:bounded-by-fractional-greens-function}, \ref{assumption:four-product-assumption}. 
\end{lemma}
\begin{proof}
Let $\{\delta_N^{\ref{lemma:nonlinear-part-derivative-difference-moment-bounds}}\}_{N \geq 0}$, $\{\delta_N^{\ref{lemma:F-bound-difference}}\}_{N \geq 0}$ be as in Lemmas \ref{lemma:nonlinear-part-derivative-difference-moment-bounds} and \ref{lemma:F-bound-difference}. For $N \geq 0$, define 
\[ \delta_N^{\ref{lemma:nonlinear-part-finite-N-difference-bound}}(t) := \max\{\delta_N^{\ref{lemma:nonlinear-part-derivative-difference-moment-bounds}}(t/2), \delta_N^{\ref{lemma:F-bound-difference}}(t/2)\}^{1/2}, ~~ t \in (0, 1]. \]
The two inequalities then follow by the same argument as in the proof of Lemma \ref{lemma:nonlinear-part-finite-N-bounds}, using Lemmas \ref{lemma:nonlinear-part-derivative-difference-moment-bounds} and \ref{lemma:F-bound-difference} in place of Lemmas \ref{lemma:A-2-bounds} and \ref{lemma:F-bound} for the needed moment estimates. In the course of the argument, we also use that $\delta_N^{\ref{lemma:nonlinear-part-derivative-difference-moment-bounds}}$ and $\delta_N^{\ref{lemma:F-bound-difference}}$ are non-increasing, so that we may bound the values of these functions at $u \in (t/2, t]$ by their values at $t/2$.
\end{proof}

\begin{definition}
For $N \geq 0$, 
$t \in (0, 1]$, $x \in \torus^d$, let
\[ \mbf{D}^{2}_{N}(t, x) := \rconn^{2}_{N}(t, x)- \rconn^{2}(t, x). \]
\end{definition}

The following result is a direct consequence of Lemmas \ref{lemma:nonlinear-part-finite-N-bounds} and \ref{lemma:nonlinear-part-finite-N-difference-bound} and Definition \ref{def:nonlinear-part-l2-limit}. 

\begin{cor}\label{cor:D-N-l2-bound}
Let $\{\delta_N^{\ref{lemma:nonlinear-part-finite-N-difference-bound}}\}_{N \geq 0}$ be the sequence of functions from Lemma \ref{lemma:nonlinear-part-finite-N-difference-bound}. Then for any $N \geq 0$, 
$t \in (0, 1]$, $r \in (t/2, t]$, $x \in \torus^d$, we have that
\[ \|\mbf{D}^{2}_{N}(r, x)\|_{L^2} \leq \const t^{-(1/2)(d-1-\alpha)} \delta_N^{\ref{lemma:nonlinear-part-finite-N-difference-bound}}(t).\]
Here, $\const$ depends only on $d$ and the constants $\alpha, \constgf, \constfourp$ from Assumptions \ref{assumption:bounded-by-fractional-greens-function}, \ref{assumption:four-product-assumption}. 
\end{cor}


We can now use Assumption \ref{assumption:tail-bounds}, Lemma \ref{lemma:A-2-is-quadratic-form}, and the various moment estimates to obtain tail bounds for $\rconn^2$ and related quantities. 

\begin{lemma}\label{lemma:A-2-tail-bounds}
For any 
$t \in (0, 1]$, $r, s \in (t/2, t]$, $x, y \in \torus^d$, we have that for $u \geq 0$,
\[ \p(|\rconn^{2}(r, x)| > u) \leq \const \exp\big(-(u / {t^{-(1/2)(d-1-\alpha)}})^\exptb /C \big),  \]
\[ \begin{split}
\p(|\rconn^{2}(r, x) -& \rconn^{2}(s, y)| > u) \leq \\
&\const \exp\big(- \big(u / {(t^{-(1/2)(d-1-\alpha)} \dchain_t((r, x), (s, y))})\big)^\exptb /C \big).
\end{split}\]
The above inequalities also hold with $\rconn^{2}$ replaced by $\rconn^{2}_{N}$ for any $N \geq 0$. Additionally, let $\{\delta_N^{\ref{lemma:nonlinear-part-finite-N-difference-bound}}\}_{N \geq 0}$ be the sequence of functions from Lemma \ref{lemma:nonlinear-part-finite-N-difference-bound}. Then for any $N \geq 0$, we have that for $u \geq 0$,
\[ \p(|\mbf{D}^{2}_{N}(r, x)| > u) \leq \const \exp\big(- \big(u / {(t^{-(1/2)(d-1-\alpha)} \delta_N^{\ref{lemma:nonlinear-part-finite-N-difference-bound}}(t))}\big)^\exptb /C \big) ,\]
\[\begin{split}
\p(|\mbf{D}^{2}_{N}(r, x) &- \mbf{D}^{2}_{N}(s, y)| > u) \leq  \\
&\const \exp\big(- \big(u / {(t^{-(1/2)(d-1-\alpha)} \dchain_t((r, x), (s, y)) \delta_N^{\ref{lemma:nonlinear-part-finite-N-difference-bound}}(t))}\big)^\exptb /C \big) .
\end{split}\]
Here, $\const$ depends only on $d$ and the constants $\exptb, \consttb, \alpha, \constgf, \constfourp$ from Assumptions \ref{assumption:tail-bounds}, \ref{assumption:bounded-by-fractional-greens-function} and \ref{assumption:four-product-assumption}. 
\end{lemma}
\begin{proof}
We have that $\rconn^{2}_{N}(r, x) \stackrel{L^2}{\ra} \rconn^{2}(r, x)$, and thus the convergence also happens in distribution. Thus by the portmanteau lemma, we have that for $u \geq 0$,
\[ \p(|\rconn^{2}(r, x)| > u) \leq \liminf_{N \toinf} \p(
|\rconn^{2}_{N}(r, x)| > u). \]
The first desired result now follows by combining Assumption \ref{assumption:tail-bounds}, Lemma \ref{lemma:A-2-is-quadratic-form}, and Lemma \ref{lemma:nonlinear-part-finite-N-bounds}. The second desired result follows similarly. The third and fourth desired results also follow similarly, where we use Lemma \ref{lemma:nonlinear-part-finite-N-difference-bound} in place of Lemma \ref{lemma:nonlinear-part-finite-N-bounds}.
\end{proof}

Lemma \ref{lemma:A-2-tail-bounds} may be used to obtain the following result.

\begin{lemma}\label{lemma:A-2-a-j-cont-mod}
The process $\rconn^2 = (\rconn^{2}(t, x), t \in (0, 1], x \in \torus^d)$ has a continuous modification.
\end{lemma}
\begin{proof}
Let $t_0 \in (0, 1]$. Define the process $X = (X_{t, x}, (t, x) \in (t_0/2, t_0] \times \torus^d)$ by
\[ X_{t, x} := (C^{-1/\exptb}) t_0^{(1/2)(d-1-\alpha)} \rconn^{2}(t, x). \]
Then by Lemma \ref{lemma:A-2-tail-bounds}, we have that for all $(r, x), (s, y) \in (t_0/2, t_0] \times \torus^d$,
\[ \p\big(|X_{r, x} - X_{s, y}| > u \dchain_t((r, x), (s, y))\big) \leq \const \exp(-u^\exptb), ~~ u \geq 0.\]
Then by Theorem \ref{thm:continuous-modification}, the process $X$ has a continuous modification, and thus so does $(\rconn^{2}(t, x), (t, x) \in (t_0/2, t_0] \times \torus^d)$. Since this holds for any $t_0 \in (0, 1]$, by Lemma \ref{lemma:continuous-modification}, we obtain that $(\rconn^{2}(t, x), t \in (0, 1], x \in \torus^d)$ has a continuous modification, as desired.
\end{proof}

Hereafter, we assume that (after a suitable modification) the process $\rconn^2 = (\rconn^{2}, t \in (0, 1], x \in \torus^d)$ has continuous sample paths. The following lemma is the analogue of Lemma \ref{lemma:A-1-exp-sup-local}.


\begin{lemma}\label{lemma:A-2-exp-sup-local}
For any $p \geq 1$, 
$t_0 \in (0, 1]$, we have that
\[ \E\bigg[ \sup_{(t, x) \in (t_0/2, t_0] \times \torus^d} |\rconn^{2}(t, x)|^p \bigg] \leq \const t_0^{-p(1/2)(d-1-\alpha)} (1 + |\log t_0|^{1/\exptb})^p. \]
The above inequality also holds with $\rconn^{2}$ replaced by $\rconn^{2}_{N}$ for any $N \geq 0$. 
Additionally, let $\{\delta_N^{\ref{lemma:nonlinear-part-finite-N-difference-bound}}\}_{N \geq 0}$ be the sequence of functions from Lemma \ref{lemma:nonlinear-part-finite-N-difference-bound}. Then for any $N \geq 0$, we have that
\[\begin{split}
\E\bigg[\sup_{(t, x) \in (t_0/2, t_0] \times \torus^d} |\mbf{D}^{2}_{N}&(t, x)|\bigg] \leq \const t_0^{-p(1/2)(d-1-\alpha)} (1 + |\log t_0|^{1/\exptb})^p (\delta_N^{\ref{lemma:nonlinear-part-finite-N-difference-bound}}(t_0))^p. 
\end{split}\]
Here, $\const$ depends only on $d, p,$ and the constants $\exptb, \consttb, \alpha, \constgf, \constfourp$ from Assumptions \ref{assumption:tail-bounds}, \ref{assumption:bounded-by-fractional-greens-function} and \ref{assumption:four-product-assumption}. 
\end{lemma}
\begin{proof}
Let $\const$ be as in Lemma \ref{lemma:A-2-tail-bounds}. Define the stochastic process $(X_{t, x}, (t, x) \in (t_0/2, t_0] \times \torus^d)$ by 
\[ X_{t, x} :=  (C^{-1/\exptb}) t_0^{(1/2)(d-1-\alpha)} \rconn^{2}(t, x). \]
Then by Lemma \ref{lemma:A-2-tail-bounds}, we have that for all $(r, x), (s, y) \in (t_0/2, t_0] \times \torus^d$,
\[ \p(|X_{r, x}| > u) \leq \const \exp(-u^\exptb), ~~ u \geq 0, \]
\[ \p\big(|X_{r, x} - X_{s, y}| > u \dchain_t((r, x), (s, y))\big) \leq \const \exp(-u^\exptb), ~~ u \geq 0.\]
By the first inequality and \cite[Lemma A.2]{Dirksen2015}, we have that
\[ \sup_{(r, x) \in (t_0/2, t_0] \times \torus^d} \E\big[|X_{r, x}|^p] \leq \const. \]
Thus by combining the above with Theorem \ref{thm:dirksen-exp-sup} and Lemma \ref{lemma:entropy-integral-bound}, we obtain
\[ \E\bigg[\sup_{(t, x) \in (t_0/2, t_0] \times \torus^d} |X_{t, x}|^p\bigg] \leq \const + \const(1 +  |\log t_0|^{1/\exptb})^p.\]
The first desired result now follows. The proof for $\rconn^{2}_{N}$ in place of $\rconn^{2}$ is essentially the same. 
The second desired result is argued similarly.
\end{proof}




\begin{prop}\label{prop:nonlinear-part-c0-bound}
For $\varep > 0$, let $\gamma_\varep := (1/2)(d-1-\alpha) + \varep$. For any $\varep > 0$, $p \geq 1$, we have that
\[ \sup_{N \geq 0} \E\bigg[\sup_{t \in (0, 1]} t^{p\gamma_\varep} \|\rconn^2_N(t)\|_{C^0}^p\bigg], ~~ \E\bigg[\sup_{t \in (0, 1]} t^{p\gamma_\varep} \|\rconn^2(t)\|_{C^0}^p\bigg] \leq \const_{\varep, p} < \infty. \]
Moreover, we have that
\[ \lim_{N \toinf} \E\bigg[\sup_{t \in (0, 1]} t^{p\gamma_\varep} \|\rconn^2_N(t) - \rconn^2(t)\|_{C^0}^p\bigg] = 0.\]
Here the constant $\const_{\varep, p}$ depends only on $\varep, p, d$, and the constants $\exptb, \consttb, \alpha, \constgf, \constfourp$ from Assumptions \ref{assumption:tail-bounds}, \ref{assumption:bounded-by-fractional-greens-function} and \ref{assumption:four-product-assumption}. 
\end{prop}
\begin{proof}
The first result follows by combining Lemma \ref{lemma:A-2-exp-sup-local} with Lemma \ref{lemma:first-chaining-bound}. The second result follows by combining Lemma \ref{lemma:A-2-exp-sup-local} with Lemma \ref{lemma:chaining-bound-convergence-to-0}.
\end{proof}

As previously mentioned, having controlled the $C^0$ norm, we now move on to controlling the $C^1$ norm. We first show the following preliminary result, which will also allow us to prove Lemmas \ref{lemma:A-2-continuous-modification} and \ref{lemma:B-1-is-nonlinear-part}.

\begin{lemma}\label{lemma:nonlinear-part-duhamel-probabilistic}
For all $t_0, t_1 \in (0, 1]$, $t_0 < t_1$, $x \in \torus^d$, we have that a.s.,
\begin{align}
\rconn^2(t_1, x) &= (e^{(t_1 - t_0) \Delta} \rconn^2(t_0))(x) \notag \\
&\qquad + \int_0^{t_1 - t_0} \big(e^{(t_1 - t_0 - s)\Delta} X^{(2)}(\rconn^1(t_0 + s))\big)(x) ds. \label{eq:A-2-integral-identity} 
\end{align}
\end{lemma}
\begin{proof}
By Lemma \ref{lemma:rho-later-time}, the result is true if we replace $\rconn^2, \rconn^1$ by $\rconn^2_N, \rconn^1_N$, since $\rconn^2_N = \rho(\rconn^1_N)$ by definition (recall Definition \ref{def:A-2-N}). Taking $N \toinf$, we have (recall Definition \ref{def:nonlinear-part-l2-limit}) that $\rconn^2_N(t_1, x) \stackrel{L^2}{\ra} \rconn^2(t_1, x)$. To finish, it suffices to show that 
\[
(e^{(t_1 - t_0)\Delta} \rconn^2_N(t_0))(x) \stackrel{P}{\ra} (e^{(t_1 - t_0)\Delta} \rconn^2(t_0))(x),
\]
and that
\[\begin{split}
\int_0^{t_1 - t_0} \big(e^{(t_1 - t_0 - s)\Delta} &X^{(2)}(\rconn^1_N(t_0 + s))\big)(x) ds \stackrel{P}{\ra} \\
&\int_0^{t_1 - t_0} \big(e^{(t_1 - t_0 - s)\Delta} X^{(2)}(\rconn^1(t_0 + s))\big)(x) ds.
\end{split}\]
Note that by Proposition \ref{prop:nonlinear-part-c0-bound}, $\E[ \|\rconn^2_N(t_0) - \rconn^2(t_0)\|_{C^0}] \ra 0$. The first claim now follows since (applying \eqref{eq:heat-semigroup-Cr-contraction} with $r = 0$) 
\[ \|e^{(t_1 - t_0)\Delta} (\rconn^2_N(t_0) - \rconn^2(t_0))\|_{C^0} \leq \|\rconn^2_N(t_0) - \rconn^2(t_0)\|_{C^0}. \]
For the second claim, define $\tilde{\rconn}^1_N(t) := \rconn^1_N(t + t_0)$ for $t \in [0, 1 - t_0]$, and analogously for $\tilde{\rconn}^1$. Since $\rconn_N^1(0)$ is smooth, we have that $\rconn^1_N \in \mc{P}_1^1$ for all $N \geq 0$, and thus $\tilde{\rconn}_N^1 \in \mc{P}_{1 - t_0}^1$ for all $N \geq 0$. We also have that $\tilde{\rconn}^1 \in \mc{P}_{1 - t_0}^1$, which follows since $\rconn^1$ is a solution to the heat equation (by Lemma \ref{lemma:A-1-modification}). Now by \eqref{eq:linear-part-convergence-positive-time}, we have that $\|\tilde{\rconn}^1_N - \tilde{\rconn}^1\|_{\mc{P}_{1 - t_0}^1} \ra 0$. The second claim follows by Lemma \ref{lemma:rho-A-n-convergence}.
\end{proof}

\begin{cor}\label{cor:A-2-integral-identity}
On an event of probability $1$, we have that \eqref{eq:A-2-integral-identity} holds for all $t_0, t_1 \in (0, 1]$, $t_0 < t_1$, $x \in \torus^d$.
\end{cor}
\begin{proof}
Let $E$ be the event that \eqref{eq:A-2-integral-identity} holds for all $t_0, t_1$ in a countable dense subset of $(0, 1]$, and all $x$ in a countable dense subset of $\torus^d$. By Lemma \ref{lemma:nonlinear-part-duhamel-probabilistic}, $\p(E) = 1$. Note that $\rconn^2$ has continuous sample paths (recall just after Lemma \ref{lemma:A-2-a-j-cont-mod}), and that $\sup_{t \in [t_0, 1]} \|\rconn^1(t)\|_{C^1} < \infty$ for all $t_0 \in (0, 1]$ (recall \eqref{eq:linear-part-positive-time}). The latter implies that for $t_0 \in (0, 1]$, if we define $\tilde{\rconn}^1(t) := \rconn^1(t_0 + t)$, then $\tilde{\rconn}^1 \in \mc{P}_{1-t_0}^1$, and thus by Lemma \ref{lemma:rho-in-P-T-0}, we have that $\rho^{(2)}(\tilde{\rconn}^1) \in \mc{P}_{1 - t_0}^1$ as well. By combining the previous few observations, we have that on the event $E$, the identity \eqref{eq:A-2-integral-identity} extends by continuity to all $t_0, t_1 \in (0, 1]$, $t_0  < t_1$, and $x \in \torus^d$.
\end{proof}

\begin{proof}[Proof of Lemma \ref{lemma:A-2-continuous-modification}]
Let $E$ be the probability $1$ event given by Corollary \ref{cor:A-2-integral-identity}. As in the proof of Corollary \ref{cor:A-2-integral-identity}, on the event $E$, we have that for all $t_0 \in (0, 1]$, $t_1 \in [t_0, 1]$,
\beq\label{eq:A-2-integral-equation} \rconn^2(t_1) = e^{(t_1 - t_0)\Delta} \rconn^2(t_0) + \rho^{(2)}(\tilde{\rconn}^1)(t_1 - t_0), \eeq
where $\tilde{\rconn}^1(t) := \rconn^1(t_0 + t)$ for $t \in [0, 1 - t_0]$. Moreover, as noted in that proof, we have that $\rho^{(2)}(\tilde{\rconn}^1) \in \mc{P}_{1 - t_0}^1$. Combining this with the fact that $\rconn^2(t) \in C^0(\torus^d, \lalg^d)$ for all $t \in (0, 1]$ (so that $s \mapsto e^{s \Delta} \rconn^2(t)$ is a continuous function $(0, \infty) \ra C^1(\torus^d, \lalg^d)$ for all $t \in (0, 1]$), we obtain that on $E$, the map $t \mapsto \rconn^2(t)$ is a continuous function from $(0, 1] \ra C^1(\torus^d, \lalg^d)$. We can modify $\rconn^2$ to be identically $0$ off $E$. 
\end{proof}

\begin{proof}[Proof of Lemma \ref{lemma:B-1-is-nonlinear-part}]
This follows by combining Corollary \ref{cor:A-2-integral-identity} with \eqref{eq:A-3-integral-identity}.
\end{proof}


We now turn to getting bounds on $\nabla \rconn^2(t)$. 
As done after Lemma \ref{lemma:A-2-continuous-modification}, we will assume that (after a suitable modification) $t \mapsto \rconn^2(t)$ is a continuous function from $(0, 1] \ra C^1(\torus^d, \lalg^d)$. Even more, from the the proof of Lemma \ref{lemma:A-2-continuous-modification}, we may assume that \eqref{eq:A-2-integral-equation} holds for all $t_0 \in (0, 1]$, $t_1 \in [t_0, 1]$. The next lemma is the analogue of Lemma \ref{lemma:nonlinear-part-finite-N-bounds}.

\begin{lemma}\label{lemma:nonlinear-part-derivative-entropy-estimates}
For any $N \geq 0$, the following holds. For $l \in [d]$,
$t \in (0, 1]$, $r, s \in (t/2, t]$, $x, y \in \torus^d$, we have that
\[\begin{split}\|\ptl_l \rconn^{2}_{N}(t, x)\|_{L^2} &\leq \const t^{-(1/2)(d-\alpha)}, \\
\|\ptl_l \rconn^{2}_{N}(r, x) - \ptl_l \rconn^{2}_{N}(s, y)\|_{L^2} &\leq \const t^{-(1/2)(d-\alpha)} \dchain_t((r, x), (s, y)). 
\end{split}\]
Here, $\const$ depends only on $d$ and the constants $\alpha, \constgf, \constfourp$ from Assumptions \ref{assumption:bounded-by-fractional-greens-function}, \ref{assumption:four-product-assumption}. 
\end{lemma}
\begin{proof}
The first inequality follows by Lemma \ref{lemma:A-2-bounds}. To show the second inequality, we will show that
\[\begin{split}
\|\ptl_l \rconn^2_N(r, x) - \ptl_l \rconn^2_N(r, y)\|_{L^2} &\leq \const t^{-(1/2)(d-\alpha)} \min\bigg\{\frac{\metricspace(x, y)}{\sqrt{t}}, 1\bigg\}, \\
\|\ptl_l \rconn^2_N(r, y) - \ptl_l \rconn^2_N(s, y)\|_{L^2} &\leq \const t^{-(1/2)(d-\alpha)} \frac{|r-s|}{t}. \end{split}\]
For the first claim, let $\wloop : [0, 1] \ra \torus^d$ be a geodesic from $y$ to $x$, so that $\wloop$ has constant speed $|\wloop'| \equiv \metricspace(x, y)$. We have that
\[ \ptl_l \rconn^2_N(r, x) - \ptl_l \rconn^2_N(r, y) = \int_0^1 \nabla \ptl_l \rconn^2_N(r, \wloop(u)) \cdot \wloop'(u) du. \]
Thus
\[ |\ptl_l \rconn^2_N(r, x) - \ptl_l \rconn^2_N(r, y)|^2 \leq \metricspace(x, y)^2 \int_0^1 |\nabla \ptl_l \rconn^2_N(r, \wloop(u))|^2 du. \] 
The first claim now follows by Lemma \ref{lemma:A-2-bounds}, which gives
\[ \|\ptl_l \rconn^2_N(r, x) - \ptl_l \rconn^2_N(r, y)\|_{L^2} \leq \const \frac{\metricspace(x, y)}{\sqrt{t}} t^{-(d-\alpha)/2}, \]
as well as
\[\begin{split}
\|\ptl_l \rconn^2_N(r, x) - \ptl_l \rconn^2_N(r, y)\|_{L^2} &\leq \|\ptl_l \rconn^2_N(r, x) \|_{L^2} + \|\ptl_l \rconn^2_N(r, y)\|_{L^2} \\
&\leq \const t^{-(d-\alpha)/2}.
\end{split}\]
For the second bound, assume without loss of generality that $s < r$. Recalling \eqref{eq:A2-time-derivative}, we have that
\[\begin{split}
\ptl_l \rconn^2_N(r, y) - \ptl_l \rconn^2_N(s, y) &= \int_s^r \ptl_u \ptl_l \rconn^2_N(u, y) du  \\
&= \int_s^r \ptl_l \Delta \rconn^2_N(u, y) du + \int_s^r \ptl_l \mbf{F}^a_{N, j}(u, y) du. \end{split}\]
We thus have that
\[\begin{split}
|\ptl_l \rconn^2_N(r, y)& - \ptl_l \rconn^2_N(s, y)|^2 \leq \\
&\const (r-s) \int_s^r |\ptl_l \Delta \rconn^2_N(u, y)|^2 + |\ptl_l \mbf{F}^a_{N, j}(u, y)|^2 du.     
\end{split} \]
The second claim now follows by Lemmas \ref{lemma:A-2-bounds} and \ref{lemma:F-bound}.
\end{proof}

The next lemma is the analogue of Lemma \ref{lemma:nonlinear-part-finite-N-difference-bound}.

\begin{lemma}\label{lemma:nonlinear-part-derivative-finite-N-difference-bound}
There is a sequence $\{\delta_N^{\ref{lemma:nonlinear-part-derivative-finite-N-difference-bound}}\}_{N \geq 0}$ of maps $\delta_N^{\ref{lemma:nonlinear-part-derivative-finite-N-difference-bound}} : (0, 1] \ra [0, 1]$ such that the following hold. 
For any $t \in (0, 1]$, we have that $\lim_{N \toinf} \delta_N^{\ref{lemma:nonlinear-part-derivative-finite-N-difference-bound}}(t) = 0$. Also, for any $M \geq N \geq 0$, $l \in [d]$,
$t \in (0, 1]$, $r, s \in (t/2, t]$, $x, y \in \torus^d$, we have that
\[\begin{split}
\|\ptl_l \mbf{D}^{2}_{N, M}(r, x)\|_{L^2} &\leq \const t^{-(1/2)(d-\alpha)} \delta_N^{\ref{lemma:nonlinear-part-derivative-finite-N-difference-bound}}(t),\\
\|\ptl_l \mbf{D}^{2}_{N, M}(r, x) - \ptl_l \mbf{D}^{2}_{N, M}(s, y)\|_{L^2} &\leq \const t^{-(1/2)(d-\alpha)} \dchain_t((r, x), (s, y)) \delta_N^{\ref{lemma:nonlinear-part-derivative-finite-N-difference-bound}}(t). 
\end{split}\]
Here, $\const$ depends only on $d$ and the constants $\alpha, \constgf, \constfourp$ from Assumptions \ref{assumption:bounded-by-fractional-greens-function}, \ref{assumption:four-product-assumption}. 
\end{lemma}
\begin{proof}
Let $\{\delta_N^{\ref{lemma:nonlinear-part-derivative-difference-moment-bounds}}\}_{N \geq 0}$, $\{\delta_N^{\ref{lemma:F-bound-difference}}\}_{N \geq 0}$ be as in Lemmas \ref{lemma:nonlinear-part-derivative-difference-moment-bounds} and \ref{lemma:F-bound-difference}. For $N \geq 0$, define 
\[ \delta_N^{\ref{lemma:nonlinear-part-derivative-finite-N-difference-bound}}(t) := \max(\delta_N^{\ref{lemma:nonlinear-part-derivative-difference-moment-bounds}}(t/2), \delta_N^{\ref{lemma:F-bound-difference}}(t/2))^{1/2}, ~~ t \in (0, 1]. \]
The two inequalities then follow by the same argument as in the proof of Lemma \ref{lemma:nonlinear-part-derivative-entropy-estimates}, using Lemmas \ref{lemma:nonlinear-part-derivative-difference-moment-bounds} and \ref{lemma:F-bound-difference} in place of Lemmas \ref{lemma:A-2-bounds} and \ref{lemma:F-bound} for the needed moment estimates. In the course of the argument, we also use that $\delta_N^{\ref{lemma:nonlinear-part-derivative-difference-moment-bounds}}$ and $\delta_N^{\ref{lemma:F-bound-difference}}$ are non-increasing, so that we may bound the values of these functions at $u \in (t/2, t]$ by their values at $t/2$.
\end{proof}

\begin{lemma}\label{lemma:A2-derivative-l2-convergence}
For any $l \in [d]$, we have that $\ptl_l \rconn^{2}_{N}(t, x) \stackrel{L^2}{\ra} \ptl_l \rconn^{2}(t, x)$.
\end{lemma}
\begin{proof}
By Lemma \ref{lemma:nonlinear-part-derivative-finite-N-difference-bound}, the sequence $\{\ptl_l \rconn^{2}_{N}(t, x)\}_{N \geq 0}$ is Cauchy in $L^2$, and thus it converges in $L^2$ to some random variable, call it $Y$. To finish, it suffices to also show that $\ptl_l \rconn^{2}_{N}(t, x) \stackrel{P}{\ra} \ptl_l \rconn^{2}(t, x)$. Towards this end, let $t_0 \in (0, t)$, and let $\tilde{\rconn}^1(t) := \rconn^1(t_0 + t)$ for $t \in [0, 1 - t_0]$. We have that (by \eqref{eq:A-2-integral-equation})
\[ \begin{split}
\ptl_l \rconn^{2}(t, x) = (\ptl_l e^{(t-t_0)\Delta} &\rconn^{2}(t_0))(x) + \big(\ptl_l \rho^{(2)}(\tilde{\rconn}^1)(t - t_0)\big)(x).
\end{split}\]
By construction, the above identity is also true with $\rconn^2, \rconn^1$ replaced by $\rconn^2_N, \rconn^1_N$ for any $N \geq 0$. Thus it suffices to show that
\[ (\ptl_l e^{(t - t_0)\Delta} \rconn^{2}_{N}(t_0))(x) \stackrel{P}{\ra}  (\ptl_l e^{(t-t_0)\Delta} \rconn^{2}(t_0))(x),\]
\[ \|\rho^{(2)}(\tilde{\rconn}^1_N)(t - t_0) - \rho^{(2)} (\tilde{\rconn}^1)(t-t_0)\|_{C^1} \stackrel{P}{\ra} 0. \]
(Here $\tilde{\rconn}^1_N(t) := \rconn^1_N(t_0 + t)$ for $t \in [0, 1 - t_0]$.)
Note that by Proposition \ref{prop:nonlinear-part-c0-bound}, $\E\big[\|\rconn^{2}_{N}(t_0) - \rconn^{2}(t_0)\|_{C^0}\big] \ra 0$.
We have that 
\[\begin{split}
\|\ptl_l e^{(t - t_0)\Delta} \rconn^{2}_{N}(t_0) & - \ptl_l e^{(t - t_0)\Delta} \rconn^{2}(t_0)\|_{C^0} \leq \\
&\|e^{(t - t_0)\Delta} \rconn^{2}_{N}(t_0) - e^{(t - t_0)\Delta} \rconn^{2}(t_0)\|_{C^1}. \end{split}\]
Then applying \eqref{eq:heat-semigroup-Cr-C-r-plus-half} with $u = 1$, we obtain the further upper bound
\[\const (t- t_0)^{-1/2} \|\rconn^{2, a}_{N, j}(t_0) - \rconn^{2, a}_j(t_0)\|_{C^0}.\]
The first claim follows. For the second claim, note that by \eqref{eq:linear-part-convergence-positive-time},
\[
\|\tilde{\rconn}^1_N - \tilde{\rconn}^1\|_{\mc{P}_{1 - t_0}^1} \ra 0.
\]
The second claim then follows by Lemma \ref{lemma:rho-A-n-convergence}. 
\end{proof}



We now use Assumption \ref{assumption:tail-bounds} and the various moment estimates to obtain tail bounds for $\ptl_l \rconn^2$ and related quantities. The following lemma is the analogue of Lemma \ref{lemma:A-2-tail-bounds}. 

\begin{lemma}\label{lemma:A-2-derivative-tail-bounds}
For any $l \in [d]$, $t \in (0, 1]$, $r, s \in (t/2, t]$, $x, y \in \torus^d$, we have that for $u \geq 0$,
\[\p(|\ptl_l \rconn^{2}(r, x)| > u) \leq \const \exp\bigg(-\frac{1}{\const} \bigg(\frac{u}{t^{-(1/2)(d-\alpha)}}\bigg)^\exptb\bigg),\]
\[\begin{split}
\p(|\ptl_l \rconn^{2}(r, x)& - \ptl_l \rconn^{2}(s, y)| > u) \leq \\
&\const \exp\bigg(-\frac{1}{\const} \bigg(\frac{u}{t^{-(1/2)(d-\alpha)} \dchain_t((r, x), (s, y))}\bigg)^\exptb\bigg).
\end{split}\]
The above inequalities also hold with $\rconn^{2}$ replaced by $\rconn^{2}_{N}$ for any $N \geq 0$. Additionally, let $\{\delta_N^{\ref{lemma:nonlinear-part-derivative-finite-N-difference-bound}}\}_{N \geq 0}$ be the sequence of functions from Lemma \ref{lemma:nonlinear-part-derivative-finite-N-difference-bound}. Then for any $N \geq 0$, we have that for $u \geq 0$,
\[ \p(|\ptl_l \mbf{D}^{2}_{N}(r, x)| > u) \leq \const\exp\bigg(-\frac{1}{\const} \bigg(\frac{u}{t^{-(1/2)(d-\alpha)}\delta_N^{\ref{lemma:nonlinear-part-derivative-finite-N-difference-bound}}(t)}\bigg)^\exptb\bigg),  \]
\[\begin{split}
\p(|\ptl_l \mbf{D}^{2}_{N}(r, x) &- \ptl_l \mbf{D}^{2}_{N}(s, y)| > u) \leq \\
&\const \exp\bigg(-\frac{1}{\const} \bigg(\frac{u}{t^{-(1/2)(d-\alpha)} \dchain_t((r, x), (s, y)) \delta_N^{\ref{lemma:nonlinear-part-derivative-finite-N-difference-bound}}(t)}\bigg)^\exptb\bigg) .
\end{split}\]
Here, $\const$ depends only on $d$ and the constants $\exptb, \consttb, \alpha, \constgf, \constfourp$ from Assumptions \ref{assumption:tail-bounds}, \ref{assumption:bounded-by-fractional-greens-function} and \ref{assumption:four-product-assumption}.
\end{lemma}
\begin{proof}
By Lemma \ref{lemma:A2-derivative-l2-convergence}, we have that $\ptl_l \rconn^{2}_{N}(r, x) \stackrel{L^2}{\ra} \ptl_l \rconn^{2}(r, x)$, and thus the convergence also happens in distribution. Thus by the portmanteau lemma, we have that for $u \geq 0$,
\[\begin{split}
\p(|\ptl_l\rconn^{2}(r, x)| > u) \leq \liminf_{N \toinf} \p(|\ptl_l\rconn^{2}_N(r, x)| > u).
\end{split}\]
The first desired result now follows by combining Assumption \ref{assumption:tail-bounds}, Lemma \ref{lemma:A-2-is-quadratic-form}, and Lemma \ref{lemma:nonlinear-part-derivative-entropy-estimates}. The second desired result follows similarly. The third and fourth desired results also follow similarly, where we use Lemma \ref{lemma:nonlinear-part-derivative-finite-N-difference-bound} in place of Lemma \ref{lemma:nonlinear-part-derivative-entropy-estimates}.
\end{proof}

The following lemma is the analogue of Lemma \ref{lemma:A-2-exp-sup-local}. 

\begin{lemma}\label{lemma:A-2-derivative-exp-sup-local}
For any $p \geq 1$, $l \in [d]$, $t_0 \in (0, 1]$, we have that
\[ \E\bigg[ \sup_{(t, x) \in (t_0/2, t_0] \times \torus^d} |\ptl_l \rconn^{2}(t, x)|^p\bigg] \leq \const t_0^{-p(1/2)(d-\alpha)} (1 + |\log t_0|^{1/\exptb})^p. \]
The above also holds with $\ptl_l \rconn^{2}$ replaced by $\ptl_l \rconn^{2}_{N}$ for any $N \geq 0$. Additionally, let $\{\delta_N^{\ref{lemma:nonlinear-part-derivative-finite-N-difference-bound}}\}_{N \geq 0}$ be the sequence of functions from Lemma \ref{lemma:nonlinear-part-derivative-finite-N-difference-bound}. Then for any $N \geq 0$, we have that
\[\begin{split}
\E\bigg[ \sup_{(t, x) \in (t_0/2, t_0] \times \torus^d} |\ptl_l \mbf{D}^{2}_{N}&(t, x)|^p \bigg] \leq  \\
&\const t_0^{-p(1/2)(d-\alpha)} (1 + |\log t_0|^{1/\exptb})^p(\delta_N^{\ref{lemma:nonlinear-part-derivative-finite-N-difference-bound}}(t_0))^p .
\end{split}\]
Here, $\const$ depends only on $d, p$, and the constants $\exptb, \consttb, \alpha, \constgf, \constfourp$ from Assumptions \ref{assumption:tail-bounds}, \ref{assumption:bounded-by-fractional-greens-function} and \ref{assumption:four-product-assumption}.
\end{lemma}
\begin{proof}
The argument is exactly the same as the proof of Lemma \ref{lemma:A-2-exp-sup-local}, where instead of using Lemma \ref{lemma:A-2-tail-bounds}, we use Lemma \ref{lemma:A-2-derivative-tail-bounds}.
\end{proof}

The following proposition is the analogue of Proposition \ref{prop:nonlinear-part-c0-bound}. 

\begin{prop}\label{prop:A-2-derivative}
For $\varep > 0$, let $\gamma_\varep := (1/2)(d-1-\alpha) + \varep$. For any $\varep > 0$, $p \geq 1$, we have that
\[\sup_{N \leq \infty} \E\bigg[\sup_{t \in (0, 1]} t^{p((1/2) + \gamma_\varep)} \|\nabla \rconn^2_N(t)\|_{C^0}^p\bigg] \leq \const_{\varep, p} < \infty, \]
where we use the notation $\nabla \rconn^2_\infty := \nabla \rconn^2$.
Moreover, we have that
\[ \lim_{N \toinf} \E\bigg[\sup_{t \in (0, 1]} t^{p((1/2) + \gamma_\varep)} \|\nabla \rconn^2_N(t) - \nabla \rconn^2(t)\|_{C^0}^p\bigg] = 0.\]
Here, $\const_{\varep, p}$ depends only on $\varep, p, d$, and the constants $\exptb, \consttb, \alpha, \constgf, \constfourp$ from Assumptions \ref{assumption:tail-bounds}, \ref{assumption:bounded-by-fractional-greens-function} and \ref{assumption:four-product-assumption}.
\end{prop}
\begin{proof}
The first result follows by combining Lemma \ref{lemma:A-2-derivative-exp-sup-local} with Lemma \ref{lemma:first-chaining-bound}. The second result follows by combining Lemma \ref{lemma:A-2-derivative-exp-sup-local} with Lemma \ref{lemma:chaining-bound-convergence-to-0}.
\end{proof}

\begin{proof}[Proof of Proposition \ref{prop:B-1}]
The first two claims follow by combining Propositions \ref{prop:nonlinear-part-c0-bound} and \ref{prop:A-2-derivative}. The final two claims follow by combining the first two claims with Lemma \ref{lemma:A-3-N-convergence}.
\end{proof}

\appendix 

\numberwithin{theorem}{section}

\section{Suprema of stochastic processes}\label{appendix:sup-of-stochastic-processes}

Let $(\rconn(t, x), t \in (0, 1], x \in \torus^d)$ be a $\lalg^d$-valued stochastic process with continuous sample paths.

\begin{lemma}\label{lemma:first-chaining-bound}
Let $p \geq 1$. Suppose that there is some $\const_0$, $\gamma$, $\beta$ such that the following holds. For all $t_0 \in (0, 1]$, we have that
\[ \E\bigg[\sup_{t \in (t_0/2, t_0], x \in \torus^d} |\rconn(t, x)|^p\bigg] \leq \const_0 t_0^{-p\gamma} (1 + |\log t_0|^\beta)^p. \]
Then for any $\varep > 0$, there is some non-increasing function $\delta : \N \ra [0, \infty)$ depending only on $p, \const_0, \beta, \varep$ such that for any integer $k_0 \geq 0$, we have that
\[ \E\bigg[ \sup_{t \in (0, 2^{-k_0}], x \in \torus^d} t^{p(\gamma + \varep)} |\rconn(t, x)|^p \bigg] \leq \delta(k_0). \]
Moreover, $\delta(k_0) \ra 0$ as $k_0 \toinf$.
\end{lemma}
\begin{proof}
Let $k_0 \geq 0$. We may bound
\[\begin{split}
\E\bigg[ \sup_{t \in (0, 2^{-k_0}], x \in \torus^d} &t^{p(\gamma + \varep)} |\rconn(t, x)|^p \bigg] \leq \sum_{k=k_0}^\infty \E\bigg[\sup_{t \in (2^{-(k + 1)}, 2^{-k}], x \in \torus^d} t^{p(\gamma + \varep)} |\rconn(t, x)|^p\bigg], \end{split}\]
which may be further bounded by
\[\const_0 \sum_{k=k_0}^\infty (2^{-k})^{p(\gamma + \varep)} 2^{pk\gamma} (1 + (k \log 2)^\beta)^p.\]
Thus we may set $\delta(k_0)$ to be the above. The fact that $\lim_{k_0 \toinf} \delta(k_0) = 0$ follows because $\delta(0) < \infty$ combined with dominated convergence.
\end{proof}

Now suppose we have a sequence $\{(\rconn_N(t, x), t \in (0, 1], x \in \torus^d)\}_{N \geq 0}$ of $\lalg^d$-valued stochastic processes with continuous sample paths.

\begin{lemma}\label{lemma:chaining-bound-convergence-to-0}
Let $p \geq 1$. Suppose there is a sequence $\{\delta_N\}_{N \geq 0}$ of functions $\delta_N : (0, 1] \ra [0, 1]$, and $\const_0, \gamma, \beta$, such that the following hold. For all $t \in (0, 1]$, we have that $\lim_{N \toinf} \delta_N(t) = 0$. Also, for all $t_0 \in (0, 1]$, $N \geq 0$, we have that 
\[ \E\bigg[\sup_{t \in (t_0/2, t_0], x \in \torus^d} |\rconn_N(t, x)|^p\bigg] \leq \const_0 t_0^{-p\gamma} (1 + |\log t_0|^\beta)^p \delta_N(t_0). \]
Then for any $\varep > 0$, we have that \[ \lim_{N \toinf} \E\bigg[ \sup_{t \in (0, 1], x \in \torus^d} t^{p(\gamma + \varep)} |\rconn_N(t, x)|^p \bigg] = 0. \]
\end{lemma}
\begin{proof}
Fix $\varep > 0$. Let $k_0 \geq 0$. For $N \geq 0$, we may bound
\[
\E\bigg[ \sup_{t \in (0, 1], x \in \torus^d} t^{p(\gamma + \varep)} |\rconn_N(t, x)|^p \bigg] \leq I_{1, k_0, N} + I_{2, k_0, N}, \]
where
\[ I_{1, k_0, N} := \E\bigg[ \sup_{t \in (0, 2^{-k_0}], x \in \torus^d} t^{p(\gamma + \varep)} |\rconn_N(t, x)|^p \bigg], \]
\[ I_{2, k_0, N} := \E\bigg[ \sup_{t \in (2^{-k_0}, 1], x \in \torus^d} t^{p(\gamma + \varep)} |\rconn_N(t, x)|^p \bigg].\]
By Lemma \ref{lemma:first-chaining-bound}, we have some function $\delta : \N \ra [0, \infty)$ such that $\delta(k) \ra 0$ as $k \toinf$, and such that 
\[ \sup_{N \geq 0} I_{1, k_0, N} \leq \delta(k_0). \] 
Next, observe that
\[\begin{split}
I_{2, k_0, N} &\leq \sum_{k=0}^{k_0-1}\E\bigg[\sup_{t \in (2^{-(k + 1)}, 2^{-k}], x \in \torus^d} t^{p(\gamma + \varep)} |\rconn_N(t, x)|^p\bigg]  \\
&\leq \const_0 \sum_{k=0}^{k_0 - 1} 2^{-pk(\gamma + \varep)} (2^{-k})^{-p\gamma} (1 + |k \log 2|^\beta)^p \delta_N(2^{-k}).
\end{split}\]
Since $k_0$ is finite, and $\delta_N$ converges pointwise to $0$, we obtain that for any fixed $k_0$, $\lim_{N \toinf} I_{2, k_0, N} = 0$. We thus obtain for any fixed $k_0$,
\[ \limsup_{N \toinf} \E\bigg[ \sup_{t \in (0, 1], x \in \torus^d} t^{p(\gamma + \varep)} |\rconn_N(t, x)|^p \bigg] \leq \delta(k_0). \]
Using that $\delta(k_0) \ra 0$ as $k_0 \toinf$, the desired result now follows.
\end{proof}

In the following, recall the definition of $\dchain_t$ in Definition \ref{def:dchain}. 
For $t_0 \in (0, 1]$, let $T_{t_0} := (t_0/2, t_0] \times \torus^d$. 

\begin{theorem}\label{thm:continuous-modification}
Let $(V, |\cdot|)$ be a normed finite-dimensional vector space. Let $t_0 \in (0, 1]$. Let $(X_{t, x}, (t, x) \in T_{t_0})$ be a $V$-valued stochastic process. Suppose for some constants $\const_1 \geq 0$, $\beta > 0$, we have that for all $(r, x), (s, y) \in T_{t_0}$,
\[ \p\big(|X_{r, x} - X_{s, y}| > u \dchain_{t_0}((r, x), (s, y))\big) \leq \const_1 \exp(-u^{\beta}), ~~ u \geq 0.\]
Then $(X_{t, x}, (t, x) \in T_{t_0})$ has a continuous modification.
\end{theorem}
\begin{proof}
By \cite[Lemma A.2]{Dirksen2015}, for any $p \geq 1$, $(r, x), (s, y) \in T_{t_0}$, we have that
\[ \E[|X_{r, x} - X_{s, y}|^p] \leq \const \dchain_{t_0}((r, x), (s, y))^p, \]
where $\const$ depends only on $C_1$, $\beta$, $p$.
To obtain the continuous modification, we apply \cite[Theorem 2.3.1]{Khosh2002}. Fix $p \geq 2(d+1)$. We proceed to show that the conditions of the theorem are met. By Lemma \ref{lemma:covering-number-bound}, the metric space $(T_{t_0}, \dchain_{t_0})$ is totally bounded. Next, define the function $\Psi(x) := \const x^p$ for $x \geq 0$, where $\const$ is as in the above display. We then have by construction that
\[ \E[|X_{r, x} - X_{s, y}|^{p}] \leq \Psi(\dchain_t((r, x), (s, y))). \]
Finally, define $f(x) := x^{1/2}$ for $x \geq 0$. Observe that $\int_0^1 r^{-1} f(r) dr < \infty$. Also, by Lemma \ref{lemma:covering-number-bound}, we have that (using $p \geq 2(d+1)$ in the second inequality)
\[ \int_0^1 N(T_{t_0}, \dchain_{t_0}, r) \Psi(2r) f(r)^{-p} dr \leq \const t_0^{-(d/2)} \int_0^1 r^{-(d+1)} (2r)^{p} r^{-p/2} dr < \infty. \]
We have thus shown conditions (i) and (ii) of \cite[Theorem 2.3.1]{Khosh2002}, and thus the theorem gives the desired continuous modification.
\end{proof}

\begin{lemma}\label{lemma:continuous-modification}
Let $X = (X_{t, x}, (t, x) \in (0, 1] \times \torus^d)$ be an $\R$-valued stochastic process. Suppose that for all $t_0 \in (0, 1]$, the process $(X_{t, x}, (t, x) \in T_{t_0})$ has a continuous modification. Then $X$ has a continuous modification.
\end{lemma}
\begin{proof}
Take a sequence $\{t_n\}_{n \geq 0}$ such that $t_0 = 1$, $t_n \downarrow 0$, and $t_{n+1} \in (t_n/2, t_n)$ for all $n \geq 0$. For each $n$, let $Y^n = (Y^n_{t, x}, (t, x) \in T_{t_n})$ be a continuous modification of $(X_{t, x}, (t, x) \in T_{t_n})$. For each $n, m \geq 0$, let $S_{m, n}$ be a countable dense subset of $T_{t_m} \cap T_{t_n}$. Let $E$ be the event that for all $m, n \geq 0$, we have that $Y^m_{t, x} = Y^n_{t, x}$ for all $(t, x) \in S_{m, n}$. Then $\p(E) = 1$, and on the event $E$, we have (by continuity) that for all $m, n \geq 0$ and all $(t, x) \in T_{t_m} \cap T_{t_n}$, $Y^m_{t, x} = Y^n_{t, x}$. We may thus define a continuous modification $Y = (Y_{t, x}, (t, x) \in (0, 1] \times \torus^d)$ as follows. Fix $(t, x) \in (0, 1] \times \torus^d$. Take $n \geq 0$ such that $(t, x) \in T_{t_n}$. On $E^c$, define $Y_{t, x} := 0$, and on $E$, define $Y_{t, x} := Y^n_{t, x}$. The fact that $Y$ has continuous sample paths follows because on the event $E$, for all $n \geq 0$, we have that $Y_{t, x} = Y^n_{t, x}$ for all $(t, x) \in T_{t_n}$.
\end{proof}

Recall the notation from Definition \ref{def:dchain}, in particular $\dchain_t$, $N_{t, \varep}$.

\begin{theorem}\label{thm:dirksen-exp-sup}
Let $(V, |\cdot|)$ be a normed finite-dimensional vector space. Let $t_0 \in (0, 1]$. Let $(X_{t, x}, (t, x) \in T_{t_0})$ be a $V$-valued stochastic process with continuous sample paths. Suppose for some constants $\const_1 \geq 0$, $\beta > 0$, we have that for all $(r, x), (s, y) \in T_{t_0}$,
\[ \p\big(|X_{r, x} - X_{s, y}| > u \dchain_{t_0}((r, x), (s, y))\big) \leq \const_1 \exp(-u^\beta), ~~ u \geq 0.\]
Then for any $p \geq 1$, we have that
\[\begin{split}
\E \bigg[\sup_{(r, x) \in T_{t_0}} |X_{r, x}|^p \bigg] \leq \const_2 &\sup_{(r, x) \in T_{t_0}} \E [|X_{r, x}|^p] + \const_2 \bigg( \int_0^\infty (\log N_{t_0, \varep})^{1/\beta} d\varep\bigg)^p .
\end{split}\]
Here $\const_2$ depends only on $V$, $\const_1$,  $\beta$, and $p$.
\end{theorem}
\begin{proof}
If $C_1 \leq 2$, then this follows directly from \cite[Theorem 3.2]{Dirksen2015}. Thus, suppose that $C_1 > 2$. Take $C_0 \geq \frac{\log C_1}{\log 2}$. Note that $C_0 \geq 1$. Define a stochastic process $(Y_{t, x}, (t, x) \in T_{t_0})$ by
\[ Y_{t, x} := \frac{X_{t, x}}{C_0^{1/\beta}} , ~~ (t, x) \in T_{t_0}. \]
We claim that for any $(r, x), (s, y) \in T_{t_0}$, we have that
\[ \p(|Y_{r, x} - Y_{s, y}| > u) \leq 2 \exp(-u^\beta), ~~ u \geq 0.\]
Given this claim, the desired result then follows by applying \cite[Theorem 3.2]{Dirksen2015} to the process $(Y_{t, x}, (t, x) \in T_{t_0})$. 

To show the claim, take $u \geq 0$, and note that by assumption, we have that
\[ \p(|Y_{r, x} - Y_{s, y}| > u) \leq C_1 \exp(-C_0 u^\beta) = C_1 \exp(-(C_0 - 1) u^\beta) \exp(-u^\beta).\]
Let 
\[ a := \bigg(\frac{\log (C_1 / 2)}{C_0 - 1}\bigg)^{1/\beta}. \]
Observe that if $u \geq a$, then $C_1 \exp(-(C_0 - 1)u^\beta) \leq 2$. Now suppose that $u < a$. Using that $C_0 - 1 \geq (\log C_1 - \log 2) / \log 2$, we obtain $u^\beta \leq \log 2$, from which it follows that
\[ \p(|Y_{r, x} - Y_{s, y}| > u) \leq 1 \leq 2 \exp(-u^\beta). \]
The claim now follows.
\end{proof}

\section{Concentration of Gaussian quadratic forms}\label{appendix:quadratic-form-concentration}

\begin{lemma}\label{lemma:quadratic-form-gaussian-concentration}
Let $Q$ be a quadratic form in centered Gaussian random variables. That is, $Q$ is of the form 
\[ Q = X^T M X = \sum_{i, j = 1}^n X_i M_{ij} X_j, \]
where $n \geq 1$, $X = (X_1, \ldots, X_n)$ is a mean $0$ Gaussian random vector, and $M$ is an $n \times n$ matrix. Then for any $u \geq 0$, 
\[ \p(|Q - \E Q| > u) \leq 2 e^{3/2} \exp\bigg(-\frac{u}{2 \sqrt{\mathrm{Var}(Q)}}\bigg). \]
\end{lemma}
\begin{proof}
First, dividing $Q$ by $\sqrt{\mathrm{Var}(Q)}$, we can assume without loss of generality that $\mathrm{Var}(Q) = 1$. Next, note that since any centered Gaussian random vector is a linear function of a vector of independent standard Gaussian random variables, we can assume without loss of generality that $X$ is a vector of $n$ i.i.d. standard Gaussian random variables. Also, since
\[ x_i m_{ij} x_j + x_j m_{ji} x_i = x_i(m_{ij} + m_{ji}) x_j, \]
we can assume without loss of generality that $M$ is symmetric. Let us make all of these assumptions. Let $\lambda_1, \ldots, \lambda_n$ be the eigenvalues of $M$, repeated by multiplicity. Then by the spectral decomposition of $M$, we get that $Q = \sum_{i=1}^n \lambda_i Z_i^2$, where $Z_1, \ldots, Z_n$ are i.i.d. standard Gaussian random variables. This shows that
\[ \E(Q) = \sum_{i=1}^n \lambda_i, ~~ \mathrm{Var}(Q) = \sum_{i=1}^n 2 \lambda_i^2. \]
By our assumption that $\mathrm{Var}(Q) = 1$, this shows that $\sum_{i=1}^n \lambda_i^2 = 1/2$. In particular, $|\lambda_i| \leq 1 / \sqrt{2}$ for each $i$. Thus, 
\[ \E(e^{\frac{1}{2} \lambda_i Z_i^2}) = \int_{-\infty}^\infty \frac{e^{-\frac{1}{2}(1-\lambda_i) x^2}}{\sqrt{2\pi}} dx = \frac{1}{\sqrt{1 - \lambda_i}}. \]
Now note that the second derivative of the map $x \mapsto \log (1 + x)$ is $-1 / (1+x)^2$, which is bounded below by $-12$ when $|x| \leq 1 / \sqrt{2}$. Thus, for $|x| \leq 1 / \sqrt{2}$, Taylor expansion gives
\[ \log(1 + x) \geq x - 6x^2.\]
Multiplying both sides by $-1/2$ and exponentiating, we get
\[ \frac{1}{\sqrt{1 + x}} \leq e^{-\frac{1}{2} x + 3x^2}. \]
Since $|\lambda_i| \leq 1/\sqrt{2}$ for each $i$, and $\sum_{i=1}^n \lambda_i$  = $\E(Q)$ and $\sum_{i=1}^n \lambda_i^2 = 1/2$, this gives
\[\begin{split}
\E (e^{\frac{1}{2}Q}) &= \prod_{i=1}^n \E (e^{\frac{1}{2} \lambda_i Z_i^2}) \leq e^{\frac{1}{2} \sum_{i=1}^n \lambda_i + 3 \sum_{i=1}^n \lambda_i^2} \\
&= e^{\frac{1}{2} \E (Q) + \frac{3}{2}}.
\end{split}\]
Thus, 
\[\begin{split}
\p(Q - \E(Q) > u) &= \p(e^{\frac{1}{2} Q} > e^{\frac{1}{2} (\E(Q) + u)}) \\
&\leq e^{-(1/2) (\E (Q) + u)} \E(e^{\frac{1}{2} Q}) \leq e^{3/2} e^{-u/2}.
\end{split}\]
The left tail may be handled similarly, by first proceeding as before to obtain the bound
\[ \E (e^{-\frac{1}{2} Q}) \leq e^{-\frac{1}{2} \E (Q) + \frac{3}{2}}. \qedhere \] 
\end{proof}

\section{Deterministic PDE proofs}\label{appendix:deterministic-pde-proofs}


We first show that $\mc{Q}_T^\gamma$ (recall Definition \ref{def:Q-T-gamma}) is a Banach space. This will allow us to apply contraction mapping arguments on closed subsets of $\mc{Q}_T^\gamma$.

\begin{lemma}\label{lemma:Q-T-gamma-banach}
For any $\gamma \geq 0$, $T > 0$, $(\mc{Q}_T^\gamma, \|\cdot\|_{\mc{Q}_T^\gamma})$ is a Banach space.
\end{lemma}
\begin{proof}
From the definition of $\|\cdot\|_{\mc{Q}_T^\gamma}$, we see that $\|\cdot\|_{\mc{Q}_T^\gamma}$ is a norm on $\mc{Q}_T^\gamma$. Thus we just need to show that the space is complete. Let $\{A_n\}_{n \geq 1} \sse \mc{Q}_T^\gamma$ be a Cauchy sequence, i.e.,
\[ \lim_{N \toinf} \sup_{m, n \geq N} \|A_n - A_m\|_{\mc{Q}_T^\gamma} = 0. \]
We will construct $A \in \mc{Q}_T^\gamma$ such that $\|A_n - A\|_{\mc{Q}_T^\gamma} \ra 0$, which will show completeness. Towards this end, first observe that for any $T_0 \in (0, T]$, we have that 
\[ \lim_{N \toinf} \sup_{m, n \geq N} \sup_{t \in [T_0, T]} \|A_n(t) - A_m(t)\|_{C^1} = 0 .\]
Thus, using the fact that $\mc{P}_{T - T_0}^1$ is a Banach space (recall Definition \ref{def:P-T-r}), there exists a continuous function $A : [T_0, T] \ra C^1(\torus^d, \lalg^d)$ such that 
\begin{align}\label{anconv}
\lim_{n \toinf} \sup_{t \in [T_0, T]} \|A_n(t) - A(t)\|_{C^1} = 0. 
\end{align}
By taking $T_0 \downarrow 0$, we obtain a continuous function $A : (0, T] \ra C^1(\torus^d, \lalg^d)$. We next show that $A \in \mc{Q}_T^\gamma$ and $\|A_n - A\|_{\mc{Q}_T^\gamma} \ra 0$. Fix $\varep > 0$. By the uniform convergence of $A_n$ to $A$ on $[\varep, T]$ (in the sense of \eqref{anconv}), we have that
\[\begin{split}
\sup_{t \in [\varep, T]} &t^\gamma \|A(t)\|_{C^0} + \sup_{t \in [\varep, T]} t^{(1/2) + \gamma} \|A(t)\|_{C^1} = \\
&\lim_{n \toinf} \biggl(\sup_{t \in [\varep, T]} t^\gamma \|A_n(t)\|_{C^0} + \sup_{t \in [\varep, T]} t^{(1/2) + \gamma} \|A_n(t)\|_{C^1}\biggr) \leq \sup_{n \geq 1} \|A_n\|_{\mc{Q}_T^\gamma}. 
\end{split}\]
Taking $\varep \downarrow 0$, we obtain
\[ \|A\|_{\mc{Q}_T^\gamma} \leq \sup_{n \geq 1} \|A_n\|_{\mc{Q}_T^\gamma} < \infty, \]
where the last inequality follows because Cauchy sequences in $\mc{Q}_T^\gamma$ are norm-bounded. Thus $A \in \mc{Q}_T^\gamma$. A similar argument shows that 
\[ \|A_n - A\|_{\mc{Q}_T^\gamma} \leq \sup_{m \geq n} \|A_n - A_m\|_{\mc{Q}_T^\gamma}.  \]
The fact that $\|A_n - A\|_{\mc{Q}_T^\gamma} \ra 0$ now follows by the assumption that $\{A_n\}_{n \geq 1}$ is Cauchy.
\end{proof}

The following lemma will allow us to obtain estimates on $\rho(A)$, which will be needed for the various contraction arguments that appear later on. Recall the definitions of $X^{(2)}, X^{(3)}$ from Definition \ref{def:X}.

\begin{lemma}\label{lemma:X-2-X-3-estimate}
Let $r \geq 0$ and let $A \in C^r(\torus^d, \lalg^d)$ be a {\oneform}. We have that
\[\begin{split} \|X^{(2)}(A)\|_{C^r} &\leq \const_{r, d} \|A\|_{C^r} \|A\|_{C^{r+1}}, \\ \|X^{(3)}(A)\|_{C^r} &\leq \const_{r, d} \|A\|_{C^r}^3.
\end{split}\]
Additionally, for $A_1, A_2 \in C^r(\torus^d, \lalg^d)$, we have that
\[\begin{split}
\|X^{(2)}(A_1) - X^{(2)}(A_2)\|_{C^r} \leq &~\const_{r, d} \max\{\|A_1\|_{C^r}, \|A_2\|_{C^r}\} \|A_1 - A_2\|_{C^{r+1}} ~+ \\
& \const_{r, d} \max\{\|A_1\|_{C^{r+1}}, \|A_2\|_{C^{r+1}}\} \|A_1 - A_2\|_{C^{r}}, \end{split}\]
\[ \|X^{(3)}(A_1) - X^{(3)}(A_2)\|_{C^r} \leq \const_{r, d} \max\{\|A_1\|_{C^r}^2, \|A_2\|_{C^r}^2\} \|A_1 - A_2\|_{C^r}. \]
\begin{proof}
The first two inequalities follow from the fact (to be proved below) 
that given functions $h, g : \torus^d \ra \lalg$, we have that 
\begin{align}\label{fact}
 \|[h, g]\|_{C^r} \leq \const \|h\|_{C^r} \|g\|_{C^r}.
\end{align}
The last two inequalities follow from this fact, combined with introducing a telescoping sum, i.e., for $i, j, k \in [d]$, we may write
\[ [A_{1, i}, \ptl_j A_{1, k}] - [A_{2, i}, \ptl_j A_{2, k}] = [A_{1, i}, \ptl_j (A_{1, k} - A_{2, k})] + [A_{1, i} - A_{2, i}, \ptl_j A_{2, k}],\]
and similarly for the difference $[A_{1, i},  [A_{1, j}, A_{1, k}]] - [A_{2, i}, [A_{2, j}, A_{2, k}]]$.

To see \eqref{fact}, recalling the notation introduced in Section \ref{section:useful-lemmas}, we may write $h = \sum_{a \in [\lalgdim]} h^a \onbasislalg^a$, $g = \sum_{b \in [\lalgdim]} g^b \onbasislalg^b$, so that
\[ [h, g] = \sum_{c \in [\lalgdim]} \onbasislalg^c \biggl(\sum_{a, b \in [\lalgdim]} f^{abc} h^a g^b\biggr) .\]
It follows that
\[ \|[h, g]\|_{C^r} \leq \const \sum_{a, b \in [\lalgdim]} \|h^a g^b\|_{C^r}, \]
where $h^a, g^b$ are $\R$-valued functions for all $a, b \in [\lalgdim]$. Using that $\|h^a g^b\|_{C^r} \leq \const_r \|h^a\|_{C^r} \|g^b\|_{C^r}$ (when $r$ is integer, this is easy to see; the case for general $r$ follows from combining the integer case with \eqref{eq:C-r-norm-product}), we then obtain the further upper bound
\[ \const_r \sum_{a \in [\lalgdim]} \|h^a\|_{C^r} \sum_{b \in [\lalgdim]} \|g^b\|_{C^r} \leq \const_r \|h\|_{C^r} \|g\|_{C^r},\]
as desired.
\end{proof}

\end{lemma}

We will first need several results for initial data in $C^r$ for some $r \geq 1$. Along the way, we will show Theorem \ref{thm:zdds-existence} and Lemma \ref{lemma:zdds-uniqueness}. To begin, recall that if $A^0 \in C^0(\torus^d, \lalg^d)$, then $t \mapsto e^{t \Delta} A^0$ is the solution to the heat equation with initial data $A^0$, i.e., $\ptl_t e^{t \Delta} A^0 = \Delta e^{t\Delta} A^0$ for all $t > 0$. Moreover, if $A^0$ is smooth, then the function $(t, x) \mapsto (e^{t \Delta} A^0)(x)$ is in $C^\infty([0, \infty) \times \torus^d, \lalg^d)$.

\begin{lemma}
Let $r \geq 0$, $f \in C^r(\torus^d, V)$, where $V = \lalg^d$ or $\R$. For any $t \geq 0$, we have
\beq\label{eq:heat-semigroup-Cr-contraction} \|e^{t \Delta} f\|_{C^r} \leq \|f\|_{C^r}. \eeq
For any $u > 0$, there is a constant $\const_u$ which depends only on $u$, such that for $t \in (0, 1]$, we have
\beq\label{eq:heat-semigroup-Cr-C-r-plus-half} \|e^{t \Delta} f\|_{C^{r+u}} \leq \const_u t^{-u/2} \|f\|_{C^r}. \eeq
\end{lemma}
\begin{proof}
The claim \eqref{eq:heat-semigroup-Cr-contraction} follows because $e^{t \Delta} f$ can be explicitly written as the convolution of $f$ with a smooth density. The claim \eqref{eq:heat-semigroup-Cr-C-r-plus-half} is \cite[Equation (1.13)]{T2011c}.
\end{proof}

\begin{lemma}\label{lemma:heat-semigroup-time-continuity}
Let $f \in C^2(\torus^d, \lalg^d)$. For any $t \geq 0$, we have
\beq\label{eq:heat-semigroup-infinity-norm-time-continuity} \|e^{t \Delta} f - f\|_{C^0} \leq t d \|f\|_{C^2}.\eeq
Additionally, let $r \geq 0$, $t \in (0, 1]$, and $u \in (0, 2)$. For $f \in C^{r + 2 -u}(\torus^d, \lalg)$, we have
\beq\label{eq:heat-semigroup-general-time-continuity} \|e^{t\Delta} f - f\|_{C^{r}} \leq \const_{d, u} \|f\|_{C^{r + 2-u}} t^{1-u/2}. \eeq
\end{lemma}
\begin{proof}
Since $e^{t\Delta} f$ is a solution to the heat equation, we have $\ptl_t e^{t \Delta} f = \Delta e^{t \Delta} f = e^{t \Delta} (\Delta f)$, and thus
\[ e^{t \Delta} f - f = \int_0^t e^{s \Delta} (\Delta f) ds.\]
Thus (using \eqref{eq:heat-semigroup-Cr-contraction} with $r = 0$ in the second inequality below),
\[ \|e^{t \Delta} f - f\|_{C^0} \leq \int_0^t \| e^{s \Delta}  (\Delta f)\|_{C^0} ds \leq t \|\Delta f\|_{C^0} \leq t d\|f\|_{C^2}. \]
For the second claim, note that we have (applying \eqref{eq:heat-semigroup-Cr-C-r-plus-half} in the third inequality below),
\begin{align*}
\|e^{t \Delta} f - f\|_{C^r} &\leq \int_0^t \|\Delta e^{s \Delta} f\|_{C^r} ds \leq d \int_0^t \|e^{s \Delta} f\|_{C^{r+2}} ds \\
&\leq d\const_u \|f\|_{C^{r +2 - u}} \int_0^t s^{-u/2} du  \leq \const_{d, u} \|f\|_{C^{r + 2 -u}} t^{1-u/2}. \qedhere
\end{align*}
\end{proof}

In the following, recall the space $(\mc{P}_T^r, \|\cdot\|_{\mc{P}_T^r})$ from Definition \ref{def:P-T-r}, as well as $\rho$ from Definition \ref{def:rho}.
The next lemma shows that if $A$ is in $\mc{P}_T^r$ for some $T, r$, then $\rho(A)$ is in $\mc{P}_T^{r+1/4}$. That is, we get a regularity improvement.

\begin{lemma}\label{lemma:contraction-estimates}
Let $j \in \{2, 3\}$. Let $r \geq 1$, $T \in (0, 1]$. Let $A \in \mc{P}_T^r$. Then for any $t \in [0, T]$, we have that
\[ \int_0^t \|e^{(t-s)\Delta} X^{(j)}(A(s))\|_{C^{r+1/4}} ds < \infty, \]
and thus $\rho^{(j)}(A) : [0, T] \ra  C^{r+1/4}(\torus^d, \lalg^d)$. Moreover, $\rho^{(j)}(A) \in \mc{P}_T^{r+1/4}$, and if for some $R \geq 0$, we have $\|A\|_{\mc{P}_T^r} \leq R$, then
\[ \|\rho^{(j)}(A)\|_{\mc{P}_T^{r + 1/4}} \leq \const_{r, d} T^{3/8} R^j. \]
Additionally, let $A_1, A_2\in \mc{P}_T^{r}$ be such that $\|A_1\|_{\mc{P}_T^r}, \|A_2\|_{\mc{P}_T^r} \leq R$. Then
\[ \|\rho^{(j)}(A_1) - \rho^{(j)}(A_2)\|_{\mc{P}_T^{r+1/4}} \leq C_{r, d} T^{3/8} R^{j-1} \|A_1 - A_2\|_{\mc{P}_T^r}.\]
Consequently, we have that
\beq\label{eq:rho-A-norm-bound} \|\rho(A)\|_{\mc{P}_T^{r + 1/4}} \leq \const_{r, d} T^{3/8}(R^3 + R^2). \eeq
\beq\label{eq:rho-A1-A2-difference-norm-bound} \|\rho(A_1) - \rho(A_2)\|_{\mc{P}_T^{r+1/4}} \leq C_{r, d} T^{3/8} (R^2 + R) \|A_1 - A_2\|_{\mc{P}_T^r}.\eeq
\end{lemma}
\begin{proof}
The latter two inequalities are direct consequences of the second and third inequalities. We begin by showing the second inequality (and simultaneously, the first inequality). Let $t \in [0, T]$. We have
\[ \|\rho^{(2)}(A)(t)\|_{C^{r + 1/4}} \leq \int_0^t \|e^{(t-s)\Delta} X^{(2)}(A(s))\|_{C^{r+1/4}} ds =: I_1.\]
Applying \eqref{eq:heat-semigroup-Cr-C-r-plus-half} with $u = 5/4$, and then applying Lemma \ref{lemma:X-2-X-3-estimate}, and then using the fact that $\|A\|_{\mc{P}_T^r} \leq R$, we obtain
\[\begin{split}
I_1 &\leq \const \int_0^t (t-s)^{-5/8} \|X^{(2)}(A(s))\|_{C^{r-1}} ds \\
&\leq \const R^2 \int_0^t (t-s)^{-5/8} ds = \const R^2 t^{3/8}.
\end{split}\] 
Taking sup over $t \in [0, T]$, we obtain $\|\rho^{(2)}(A)\|_{\mc{P}_T^{r + 1/4}} \leq \const R^2 T^{3/8}$.
Next, for $j = 3$, note
\[ \|\rho^{(3)}(A)(t)\|_{C^{r+1/4}} \leq \int_0^t \|e^{(t-s)\Delta} X^{(3)}(A(s))\|_{C^{r+1/4}} ds =: I_2.\]
Applying \eqref{eq:heat-semigroup-Cr-C-r-plus-half} with $u = 1/4$, and then proceeding similar to before, we obtain
\[ I_2 \leq \const \int_0^t (t-s)^{-1/8} \|X^{(3)}(A(s))\|_{C^{r}} ds \leq \const R^3 t^{7/8}.\]
We thus obtain $\|\rho^{(3)}(A)\|_{\mc{P}_T^{r + 1/4}} \leq \const R^3 T^{3/8}$ (here we used that $T^{7/8} \leq T^{3/8}$ since $T \in (0, 1]$). We have thus shown the second inequality (as well as the first inequality). The third inequality follows by a similar argument.

Let $j \in \{2, 3\}$. It remains to show that $\rho^{(j)}(A) \in \mc{P}_T^{r+1/4}$, that is, the map $\rho^{(j)}(A) : [0, T] \ra C^{r+1/4}(\torus^d, \lalg^d)$ is continuous. Let $t_0, t_1 \in [0, T]$, $t_0 < t_1$. 
We may bound
\[\begin{split}
\|\rho^{(j)}&(A)(t_1) - \rho^{(j)}(A)(t_0)\|_{C^{r+1/4}} \leq \int_{t_0}^{t_1} \|e^{(t_1 - s) \Delta} X^{(j)}(A(s)) \|_{C^{r+1/4}} ds ~+ \\
&\int_0^{t_0} \|e^{(t_1 - t_0)\Delta} e^{(t_0 - s) \Delta} X^{(j)}(A(s)) - e^{(t_0 - s)\Delta} X^{(j)}(A(s))\|_{C^{r+1/4}} ds.
\end{split}\]
By arguing as in the proof of the second inequality, we obtain
\[ \int_{t_0}^{t_1} \|e^{(t_1 - s) \Delta} X^{(j)}(A(s))\|_{C^{r+1/4}} ds \leq \const R^j (t_1 - t_0)^{3/8}. \]
For the other term, first note that for $s \in [0, t_0]$, we have (applying \eqref{eq:heat-semigroup-general-time-continuity} with $u = 7/4$),
\[\begin{split}
\|e^{(t_1 - t_0)\Delta} e^{(t_0 - s) \Delta} X^{(j)}&(A(s)) - e^{(t_0 - s)\Delta} X^{(j)}(A(s))\|_{C^{r+1/4}} \leq \\
&\const (t_1 - t_0)^{1/8} \|e^{(t_0 - s)\Delta} X^{(j)}(A(s))\|_{C^{r+1/2}}, 
\end{split}\]
and thus we obtain the upper bound
\[ \const (t_1 - t_0)^{1/8} \int_0^{t_0} \|e^{(t_0 - s)\Delta} X^{(j)}(A(s))\|_{C^{r+1/2}} ds. \]
Arguing similar to before (where instead of applying \eqref{eq:heat-semigroup-Cr-C-r-plus-half} with $u = 1/4, 5/4$, we now apply the inequality with $u = 1/2, 3/2$), we obtain that the above is further upper bounded by
\[ \const t_0^{1/4} R^j (t_1 - t_0)^{1/8}. \]
Combining the previous estimates, we obtain that $\rho^{(j)}(A)$ is indeed continuous (in fact, it is H\"{o}lder continuous). 
\end{proof}

\begin{proof}[Proof of Lemma \ref{lemma:rho-in-P-T-0}]
This follows directly from Lemma \ref{lemma:contraction-estimates}.
\end{proof}

\begin{proof}[Proof of Lemma \ref{lemma:rho-A-n-convergence}]
This follows directly from Lemma \ref{lemma:contraction-estimates}.
\end{proof}

Let $T > 0$, $r \geq 1$. Given $A^0 \in C^r(\torus^d, \lalg^d)$, define $W : \mc{P}_T^r \ra \mc{P}_T^r$ by
\beq\label{eq:W-def} W(A)(t) := e^{t \Delta} A^0 + \rho(A)(t), ~~ A \in \mc{P}_T^r, t \in [0, T]. \eeq
For $T > 0$, $r \geq 1$, $R \geq 0$, define
\[ \mc{P}_{T, R}^r := \{A \in \mc{P}_T^r : \|A\|_{\mc{P}_T^r} \leq R\}. \]
Note that $\mc{P}_{T, R}^r$ is a complete metric space (with the metric arising from the norm $\|\cdot\|_{\mc{P}_T^r}$ on $\mc{P}_T^r$).
The following corollary of Lemma \ref{lemma:contraction-estimates} shows that if $W(A) = A$, then $A$ is smooth.

\begin{cor}\label{cor:fixed-point-higher-regularity}
Let $A^0$ be a smooth {\oneform}. Let $W$ be as defined in \eqref{eq:W-def} using $A^0$. Let $T \in (0, 1]$. Suppose $A \in \mc{P}_T^1$ is such that $W(A) = A$. Then $A \in \mc{P}_T^r$ for all $r \geq 1$. Consequently, $A(t)$ is smooth for all $t \in [0, T]$.
\end{cor}
\begin{proof}
Let $A^1(t) := e^{t \Delta} A^0$ for $t \in (0, 1]$. Since $A^0$ is smooth, we have that $A^1 \in \mc{P}_T^r$ for all $r \geq 1$. We show that for all integer $m \geq 0$, $A \in \mc{P}_T^{1 + m/4}$. The base case $m = 0$ follows by assumption. Next, suppose that the $A \in \mc{P}_T^{1 + m/4}$ for some $m \geq 0$.
Then by Lemma \ref{lemma:contraction-estimates}, we have that $\rho(A) \in \mc{P}_T^{1 + (m+1)/4}$. Using that $A = W(A) = A^1 + \rho(A)$, we obtain that $A \in \mc{P}_T^{1 + (m+1)/4}$, and thus the inductive step is proven. The result now follows.
\end{proof}

\begin{lemma}\label{lemma:rho-later-time}
Let $j \in \{2, 3\}$. Let $A : (0, T] \ra C^1(\torus^d, \lalg^d)$ be a continuous function. Suppose that $\rho^{(j)}(A)$ is well-defined for $A$. Then for any $T_0 \in (0, T)$, the following holds. Define $\tilde{A} : [0, T - T_0] \ra C^1(\torus^d, \lalg^d)$ by $\tilde{A}(t) := A(T_0 + t)$. Then for $t \in [0, T - T_0]$, 
\[ \rho^{(j)}(A)(T_0 + t) = e^{t\Delta} (\rho^{(j)}(A)(T_0)) + \rho^{(j)}(\tilde{A})(t).\]
\end{lemma}
\begin{proof}
We have that
\[ \begin{split}
\rho^{(j)}(A)(T_0 + t) &= \int_0^{T_0 + t} e^{(T_0 + t - s)\Delta} X^{(j)}(A(s)) ds \end{split}. \]
The right hand side above is equal to
\[ \int_0^{T_0} e^{t \Delta} e^{(T_0 - s)\Delta} X^{(j)}(A(s)) ds + \int_{T_0}^{T_0 + t} e^{(T_0 + t - s)\Delta} X^{(j)}(A(s)) ds. \]
Since $\rho^{(j)}(A)$ is well-defined for $A$, we have that 
\[ \int_0^{T_0} \|e^{(T_0 - s)\Delta} X^{(j)}(A(s))\|_{C^1} ds < \infty. \]
Combining this with the fact that $e^{t\Delta} : C^1(\torus^d, \lalg^d) \ra C^1(\torus^d, \lalg^d)$ is continuous (by e.g. \eqref{eq:heat-semigroup-Cr-contraction} with $r = 1$), we have that (using \cite[Chapter V.5, Corollary 2]{Yosida1995})
\[ \int_0^{T_0} e^{t \Delta} e^{(T_0 - s)\Delta} X^{(j)}(A(s)) ds = e^{t \Delta} \int_0^{T_0} e^{(T_0 - s)\Delta} X^{(j)}(A(s)) ds. \]
We also have that (changing variables $u = s - T_0$)
\[ \int_{T_0}^{T_0 + t} e^{(T_0 + t - s)\Delta} X^{(j)}(A(s)) ds = \int_0^t e^{(t-u)\Delta} X^{(j)}(\tilde{A}(u)) du. \]
The desired result now follows by combining the previous chain of identities.
\end{proof}

\begin{lemma}\label{lemma:fixed-point-later-time}
Let $A^0$ be a smooth {\oneform}. Let $r \geq 1$, $T > 0$. Define $W$ as in~\eqref{eq:W-def} using $A^0$. Suppose $A \in \mc{P}_T^r$ satisfies $W(A) = A$. Let $T_0 \in (0, T)$. Define $\tilde{A} \in \mc{P}_{T - T_0}^r$ by $\tilde{A}(t) := A(T_0 + t)$. Define $\tilde{W} : \mc{P}_{T - T_0}^r \ra \mc{P}_{T - T_0}^r$ as in \eqref{eq:W-def} using $A(T_0)$ in place of $A^0$. Then $\tilde{W}(\tilde{A}) = \tilde{A}$.
\end{lemma}
\begin{proof}
Since $W(A) = A$, we have that for $t \in [0, T - T_0]$, 
\[ A(T_0 + t) = e^{(T_0 + t) \Delta} A^0 + \rho(A)(T_0 + t). \]
By Lemma \ref{lemma:rho-later-time} (applied with both $j = 2, 3$), we have that
\[ \rho(A)(T_0 + t) = e^{t \Delta} \rho(A)(T_0) + \rho(\tilde{A})(t). \]
We thus obtain
\[\begin{split}
\tilde{A}(t) = A(T_0 + t) &= e^{t \Delta} \big(e^{T_0 \Delta} A^0 + \rho(A)(T_0)\big) + \rho(\tilde{A})(t) \\
&= e^{t \Delta} \tilde{A}(0) + \rho(\tilde{A})(t), 
\end{split}\]
as desired.
\end{proof}

\begin{lemma}\label{lemma:W-is-a-contraction}
There is a continuous non-increasing function $\tau : [0, \infty) \ra (0, 1]$ such that the following holds. Given a smooth {\oneform} $A^0$, let $T_0 = \tau(\|A^0\|_{C^1})$ and $R_0 = 2\|A^0\|_{C^1}$. Let $W$ be defined in terms of $A^0$ by \eqref{eq:W-def}. Then for any $T \leq T_0$, $W$ maps $\mc{P}_{T, R_0}^1$ to $\mc{P}_{T, R_0}^1$, and moreover, it is a $(1/2)$-contraction:
\[ \|W(A) - W(\tilde{A})\|_{\mc{P}_{T}^1} \leq \frac{1}{2} \|A - \tilde{A}\|_{\mc{P}_{T}^1}\ \ \textup{ for all $A, \tilde{A} \in \mc{P}_{T, R_0}^1$}. \]
\end{lemma}
\begin{proof}
Let $\const_{1, d}$ be as in Lemma \ref{lemma:contraction-estimates}. 
By \eqref{eq:heat-semigroup-Cr-contraction} and Lemma \ref{lemma:contraction-estimates}, for any $T \in (0, 1]$, $A \in \mc{P}_{T, R_0}^1$, we have that
\[ \|W(A)\|_{\mc{P}_{T}^1} \leq \|A^0\|_{C^1} + \const_{1, d} T^{3/8}(R_0^3 + R_0^2). \]
Similarly, by Lemma \ref{lemma:contraction-estimates}, for any $A, \tilde{A} \in \mc{P}_{T, R_0}^1$, we have
\[ \|W(A) - W(\tilde{A})\|_{\mc{P}_T^1} \leq \const_{1, d} T^{3/8} (R_0^2 + R_0) \|A- \tilde{A}\|_{\mc{P}_T^1}. \]
We may thus define $\tau(M)$ as the largest $T \in (0, 1]$ such that
\[ \const_{1, d} T^{3/8} ((2 M)^3 + (2M)^2) \leq M, ~~ \const_{1, d} T^{3/8}((2M)^2 + (2M)) \leq \frac{1}{2}. \]
The desired result now follows.
\end{proof}

The next lemma shows that $W$ has a fixed point in $\mc{P}_T^1$ for small enough $T$. 

\begin{lemma}\label{lemma:local-existence-fixed-point}
Let $\tau$ be as in Lemma \ref{lemma:W-is-a-contraction}. Let $A^0$ be a smooth {\oneform}, and let $W$ be defined using $A^0$ as in \eqref{eq:W-def}. Let $T_0 = \tau(\|A^0\|_{C^1})$. Then there exists a unique element $A \in \mc{P}_{T_0}^1$ such that $W(A) = A$. Even more, if for some $T_1 \leq T_0$, we have some $B \in \mc{P}_{T_1}^1$ such that $W(B) = B$, then $A(t) = B(t)$ for all $t \in [0, T_1]$.
\end{lemma}
\begin{proof}
Let $R_0 = 2 \|A^0\|_{C^1}$. By Lemma \ref{lemma:W-is-a-contraction}, we have that $W$ is a strict contraction on $\mc{P}_{T_0, R_0}^1$. Thus by the contraction mapping theorem (\cite[Theorem 2.1]{Tes2012}), we obtain a fixed point $A \in \mc{P}_{T_0, R_0}^1$ of $W$, i.e. $W(A) = A$. This fixed point is unique in the sense that if $\tilde{A} \in \mc{P}_{T_0, R_0}^1$ is such that $W(\tilde{A}) = \tilde{A}$, then $A = \tilde{A}$.

Now suppose that for some $T_1 \leq T_0$, we have some $B \in \mc{P}_{T_1}^1$ such that $W(B) = B$. We proceed to show that $A(t) = B(t)$ for all $t \in [0, T_1]$. Let $R_1 = \max\{2\|B\|_{\mc{P}_{T_1}^1}, R_0\}$. Let $T \in (0, T_1]$ be such that (here $\const_{1, d}$ is as in Lemma~\ref{lemma:contraction-estimates})
\[ \const_{1, d} T^{3/8}(R_1^3 + R_1^2) \leq R_1 /2, ~~ \const_{1, d} T^{3/8} (R_1^2 + R_1) \leq \frac{1}{2}.\]
Then as in the proof of Lemma \ref{lemma:W-is-a-contraction}, by Lemma \ref{lemma:contraction-estimates}, we have for any $A \in \mc{P}_{T, R_1}^1$,
\[ \|W(A)\|_{\mc{P}_T^1} \leq \|A^0\|_{C^1} + \const_{1, d} T^{3/8}(R_1^3 + R_1^2) \leq \frac{R_1}{2} + \frac{R_2}{2} = R_1,\]
and for any $A, \tilde{A} \in \mc{P}_{T, R_1}^1$, we have that
\[ \|W(A) - W(\tilde{A})\|_{\mc{P}_T^1} \leq \frac{1}{2} \|A - \tilde{A}\|_{\mc{P}_T^1}.  \]
Thus $W : \mc{P}_{T, R_1}^1 \ra \mc{P}_{T, R_1}^1$ is a strict contraction.
Thus by the contraction mapping theorem, there is a unique fixed point of $W$ in $\mc{P}_{T, R_1}^1$. Note that both $A, B \in \mc{P}_{T, R_1}^1$ (by the definition of $R_1$). We thus obtain that $A(t) = B(t)$ for all $t \in [0, T]$. If $T = T_1$, then we are done, so let us suppose that $T < T_1$. Define $\tilde{A}(t) := A(T + t)$, $\tilde{B}(t) := B(T + t)$. Note $\tilde{A}, \tilde{B} \in \mc{P}_{T_1 - T, R_1}^1$. Let $\tilde{W}$ be defined as in \eqref{eq:W-def}, but using $\tilde{A}(0)$ ($= \tilde{B}(0)$) instead of $A^0$. By Lemma \ref{lemma:fixed-point-later-time}, we have that $\tilde{W}(\tilde{A}) = \tilde{A}$, $\tilde{W}(\tilde{B}) = \tilde{B}$. Moreover, observe that $\|\tilde{B}(0)\|_{C^1} = \|B(T)\|_{C^1} \leq R_1 /2$. It then follows by the same argument as before that for any $T_2 \leq T$, $\tilde{W} : \mc{P}_{T_2, R_1}^1 \ra \mc{P}_{T_2, R_1}^1$ is a strict contraction, and thus there is a unique fixed point of $\tilde{W}$ in $\mc{P}_{T_2, R_1}^1$. Applying this with $T_2 = \min\{T, T_1 - T\}$, we obtain that $\tilde{A} = \tilde{B}$ on $[0, \min\{T, T_1 - T\}]$, which implies that $A = B$ on $[0, \min\{2T, T_1\}]$. Now by iterating, we can extend this equality to the entire interval $[0, T_1]$.
\end{proof}

\begin{cor}\label{cor:fixed-point-uniqueness}
Let $A^0$ be a smooth {\oneform}, and let $W$ be defined using $A^0$ as in \eqref{eq:W-def}. Let $T > 0$. Suppose we have $A, B \in \mc{P}_T^1$ such that $W(A) = A$, $W(B) = B$. Then $A = B$.
\end{cor}
\begin{proof}
Let $M := \max\{\|A\|_{\mc{P}_T^1}, \|B\|_{\mc{P}_T^1}\}$, and let $T_1 := \tau(M)$, where $\tau$ is as in Lemma \ref{lemma:local-existence-fixed-point}. By Lemma \ref{lemma:local-existence-fixed-point}, we have that $A = B$ on $[0, \min\{T_1, T\}]$. If $T_1 \geq T$, then we are done, so let us assume that $T_1 < T$. Let $\tilde{A} \in \mc{P}_{T - T_1}^1$ be defined by $\tilde{A}(t) := A(T_1 + t)$. Let $\tilde{B} \in \mc{P}_{T - T_1}^1$ be defined similarly. Note that $\tilde{A}(0) = \tilde{B}(0)$. Let $\tilde{W}$ be defined as in \eqref{eq:W-def}, using $\tilde{A}(0)$ in place of $A^0$. By Lemma \ref{lemma:fixed-point-later-time}, we have that $\tilde{W}(\tilde{A}) = \tilde{A}$, $\tilde{W}(\tilde{B}) = \tilde{B}$. Arguing as before, we can then obtain that $\tilde{A} = \tilde{B}$ on $[0, \min\{T_1, T - T_1\}]$, which thus gives $A = B$ on $[0, \min\{2T_1, T\}]$. By iterating, we can obtain that $A = B$ on $[0, T]$, as desired.
\end{proof}

The following shows that a solution to \eqref{eq:ZDDS} is indeed a fixed point of $W$.

\begin{lemma}\label{lemma:classical-solution-is-mild-solution}
Let $A^0$ be a smooth {\oneform}. Define $W$ as in \eqref{eq:W-def}. Let $T > 0$, and let $A \in C^\infty([0, T) \times \torus^d, \lalg^d)$ be a solution to \eqref{eq:ZDDS} on $[0, T)$ with initial data $A(0) = A^0$. Then for all $T_0 \in [0, T)$, we have that $W(A |_{T_0}) = A|_{T_0}$, where $A|_{T_0} \in \mc{P}_{T_0}^1$ is the restriction of $A$ to $[0, T_0]$.
\end{lemma}
\begin{proof}
The proof is essentially given at the end of Section 8.1 in \cite{CG2013}. We reproduce it here. For $t \in [0, T)$, observe that
\[ A(t) - e^{t \Delta} A(0) = \int_0^t \ptl_s \big(e^{(t-s)\Delta} A(s)\big) ds. \]
For any $s \in (0, t)$, we have
\[ \ptl_s \big(e^{(t-s)\Delta} A(s)\big) = e^{(t-s)\Delta} (\ptl_s A(s)) - \Delta e^{(t-s)\Delta} A(s). \]
Using that $\ptl_s A(s) = \Delta A(s) + X(A(s))$ by assumption, we obtain
\[ A(t) - e^{t \Delta} A(0) =  \int_0^t e^{(t-s)\Delta} X(A(s)) ds. \]
Since $A(0) = A^0$, the desired result now follows.
\end{proof}

We are now almost able to prove Theorem \ref{thm:zdds-existence}. Recalling Remark \ref{remark:integral-equation}, the only thing remaining is to show that a solution to the integral equation is a solution to \eqref{eq:ZDDS}. This is the next lemma.

\begin{lemma}\label{lemma:mild-solution-is-classical-solution}
Let $A^0$ be a smooth {\oneform}, and let $W$ be defined as in \eqref{eq:W-def} using $A^0$. Let $T \in (0, 1]$, and let $A \in \mc{P}_T^1$ be such that $W(A) = A$. Then $A \in C^\infty([0, T) \times \torus^d, \lalg^d)$, and it is a solution to \eqref{eq:ZDDS} on $[0, T)$ with initial data $A(0) = A^0$.
\end{lemma}
\begin{proof}
By Corollary \ref{cor:fixed-point-higher-regularity}, for any $T_0 \in [0, T)$ and $r \geq 1$, we have that $A \in \mc{P}_{T_0}^r$, and thus 
\beq\label{eq:A-finite-path-norm} \sup_{0 \leq t \leq T_0} \|A(t)\|_{C^r} < \infty, ~~ T_0 \in [0, T), r \geq 1. \eeq 
Thus in particular, $A(t)$ is smooth for all $t \in [0, T)$. We proceed to show that for any $t \in (0, T)$, we have
\[ \ptl_t A(t) = \Delta A(t) + X(A(t)). \]
Once we have shown this, we automatically obtain that $A \in C^\infty([0, T) \times \torus^d, \lalg^d)$, by combining the fact that $A$ is smooth in the spatial variables along with the above equation to obtain smoothness of $A$ in the time variable.

Fix $t \in (0, T)$. For brevity, let
\[ X(s) := X(A(s)), s \in [0, T). \]
Note we have that $\ptl_t e^{t \Delta} A^0 = \Delta e^{t \Delta} A^0$. Thus since $A(t) = W(A)(t) = e^{t\Delta} A^0 + \int_0^t e^{(t-s)\Delta} X(s) ds$, it remains to show
\[ \ptl_t  \int_0^t e^{(t-s)\Delta} X(s) ds = \Delta \int_0^t e^{(t-s)\Delta} X(s) ds + X(t). \]
Let $t_0 \leq t$, $t_1 \geq t$, $t_0 < t_1$. Define the function $D_{t_0, t_1} : \torus^d \ra \lalg^d$ by
\[ D_{t_0, t_1} := \frac{1}{t_1 - t_0} \bigg(\int_0^{t_1} e^{(t_1 - s) \Delta} X(s) ds - \int_0^{t_0} e^{(t_0 - s)\Delta} X(s) ds \bigg).\]
Split $D_{t_0, t_1} = D_{1, t_0, t_1} + D_{2, t_0, t_1}$, where
\[\begin{split} D_{1, t_0, t_1} &:= \frac{1}{t_1 - t_0} \int_0^{t_0} \big(e^{(t_1 - s)\Delta} X(s) - e^{(t_0 - s) \Delta} X(s)\big) ds, \\
D_{2, t_0, t_1} &:= \frac{1}{t_1 - t_0} \int_{t_0}^{t_1} e^{(t_1 - s)\Delta} X(s) ds. 
\end{split}\]
It suffices to show that as $t_1 - t_0 \downarrow 0$ (which implies $t_0, t_1 \ra t$, since $t_0 \leq t$ and $t_1 \geq t$), we have that $D_{1, t_0, t_1}$ converges to $\Delta \int_0^t e^{(t-s)\Delta} X(s) ds$ in $C^0$, and $D_{2, t_0, t_1}$ converges to $X(t)$ in $C^0$.

For the first term, note that for any $0 < s < u$, we have $\ptl_u e^{(u - s)\Delta} X(s) = \Delta e^{(u- s)\Delta} X(s) = e^{(u - s)\Delta} \Delta X(s)$, and thus
\[ \frac{1}{t_1 - t_0} \big( e^{(t_1 - s) \Delta} X(s) - e^{(t_0-s)\Delta} X(s)\big) = \frac{1}{t_1 - t_0} \int_{t_0}^{t_1} e^{(u - s) \Delta} \Delta X(s) du. \]
Thus we have that
\[ D_{1, t_0, t_1} = \frac{1}{t_1 - t_0} \int_{t_0}^{t_1} \int_0^{t_0} e^{(u-s)\Delta} \Delta X(s) ds du. \]
For any $t_0 \leq u \leq t_1$, we have (applying~\eqref{eq:heat-semigroup-infinity-norm-time-continuity} in the second inequality below and~\eqref{eq:heat-semigroup-Cr-contraction} in the third inequality)
\[\begin{split}
\bigg\| &\int_0^{t_0} e^{(u-s)\Delta} \Delta X(s) ds - \int_0^{t_0} e^{(t_0 - s)\Delta} \Delta X(s) ds \bigg\|_{C^0} \leq \\
& \int_0^{t_0} \big\| e^{(u - t_0) \Delta} e^{(t_0 - s)\Delta} \Delta X(s) - e^{(t_0 - s)\Delta} \Delta X(s) \big\|_{C^0} ds \leq \\
& \const \int_0^{t_0} (u - t_0) \|e^{(t_0 - s)\Delta} \Delta X(s) \|_{C^2} ds \leq \const (u - t_0) t_0 \sup_{0 \leq s \leq t_0} \|X(s)\|_{C^4}.
\end{split}\]
Using this estimate, we obtain that
\[\begin{split}
\bigg\| D_{1, t_0, t_1} &- \int_0^{t_0} e^{(t_0 - s)\Delta} \Delta X(s) ds \bigg\|_{C^0} \leq \\
&\const t_0 \sup_{0 \leq s \leq t_0} \|X(s)\|_{C^4} \frac{1}{t_1 - t_0} \int_{t_0}^{t_1} (u - t_0) du.
\end{split}\]
The right hand side above goes to $0$ as $t_1 - t_0 \downarrow 0$ (recall the definition of $X(s)$ and the inequality~\eqref{eq:A-finite-path-norm}). Thus, to show that $D_{1, t_0, t_1}$ converges uniformly to $\Delta \int_0^{t} e^{(t - s)\Delta} X(s) ds$, it suffices to show that
\[ \bigg\|\int_0^{t} e^{(t - s)\Delta} \Delta X(s) ds - \int_0^{t_0} e^{(t_0 - s)\Delta} \Delta X(s) ds \bigg\|_{C^0} \ra 0 \text{ as $t_1 - t_0 \downarrow 0$}. \]
The left hand side above can be bounded by
\begin{align*}
&\int_{t_0}^t \| e^{(t-s)\Delta} \Delta X(s)\|_{C^0} ds \\
&\qquad \qquad + \int_0^{t_0} \big\|e^{(t - t_0)\Delta} e^{(t_0 - s)\Delta} \Delta X(s) - e^{(t_0 - s)\Delta} \Delta X(s)\big\|_{C^0} ds. 
\end{align*}
By \eqref{eq:heat-semigroup-Cr-contraction}, the definition of $X(s)$, and \eqref{eq:A-finite-path-norm}, we have that the first term above goes to $0$ as $t_1 - t_0 \downarrow 0$ (since this implies $t_0 \ra t$). By \eqref{eq:heat-semigroup-infinity-norm-time-continuity}, \eqref{eq:heat-semigroup-Cr-contraction}, the definition of $X(s)$, and \eqref{eq:A-finite-path-norm}, we have that the second term above goes to $0$ as $t_1 - t_0 \downarrow 0$ (since this implies $t_0 \ra t$). Thus the desired result about $D_{1, t_0, t_1}$ follows.

We now move on to $D_{2, t_0, t_1}$. We have (applying \eqref{eq:heat-semigroup-infinity-norm-time-continuity} in the second inequality)
\begin{align*}
&\bigg\| D_{2, t_0, t_1} - \frac{1}{t_1 - t_0} \int_{t_0}^{t_1} X(s) ds \bigg\|_{C^0} \\
&\leq \frac{1}{t_1 - t_0} \int_{t_0}^{t_1} \|e^{(t_1 - s) \Delta} X(s) - X(s) \|_{C^0} ds \\
&\leq C \sup_{0 \leq s \leq t_1} \|X(s) \|_{C^2} \frac{1}{t_1 - t_0} \int_{t_0}^{t_1} (t_1 - s) ds. 
\end{align*}
As $t_1 - t_0 \downarrow 0$, the right hand side above goes to $0$ (recall the definition of $X(s)$ and \eqref{eq:A-finite-path-norm}). To finish, note that as $t_1 - t_0 \downarrow 0$, we have
\[ \bigg\|\frac{1}{t_1 - t_0} \int_{t_0}^{t_1} X(s) ds - X(t)\bigg\|_{C^0} \leq \frac{1}{t_1 - t_0} \int_{t_0}^{t_1} \|X(s) - X(t)\|_{C^0} ds \ra 0, \]
where the limit follows because $s \mapsto X(s)$ is a continuous function from $[0, T)$ into $C^0(\torus^d, \lalg)$ (recall the definition of $X(s)$ and the fact that $A : [0, T) \ra C^1(\torus^d, \lalg^d)$ is continuous by assumption), combined with the fact that $t_1 - t_0 \downarrow 0$ implies $t_0, t_1 \ra t$.
\end{proof}

\begin{proof}[Proof of Theorem \ref{thm:zdds-existence}]
This follows by Lemma \ref{lemma:local-existence-fixed-point} and Lemma \ref{lemma:mild-solution-is-classical-solution}.
\end{proof}

\begin{proof}[Proof of Lemma \ref{lemma:zdds-uniqueness}]
This follows by Corollary \ref{cor:fixed-point-uniqueness} and Lemma \ref{lemma:classical-solution-is-mild-solution}.
\end{proof}

We now begin to work towards the proof of Theorem \ref{thm:rough-distributional-zdds-local-existence}. For $\gamma \geq 0$, $T > 0$, $R \geq 0$, recall the definitions of $\mc{Q}_T^\gamma$ and $\mc{Q}_{T, R}^\gamma$ (Definition \ref{def:Q-T-gamma}). 
The next lemma shows that $\rho(A)$ is well-defined for $A \in \mc{Q}_T^\gamma$, for small enough $\gamma$. In what follows, we will use the following inequality:
\beq\label{eq:t-minus-sigma-sigma-integral-2} \int_0^t (t-s)^{-\beta} s^{-\gamma} ds \leq  \const_{\beta, \gamma} t^{1 - \beta - \gamma}, ~~~ 0 \leq \beta, \gamma < 1, \eeq
which can be obtained by splitting the integral $\int_0^t = \int_0^{t/2} + \int_{t/2}^t$. 

\begin{lemma}\label{lemma:distributional-contraction-bound}
Let $\gamma \in [0, 1/4)$, $T \in (0, 1]$, $R \geq 0$. Let $A \in \mc{Q}_{T, R}^\gamma$. Then for $j \in \{2, 3\}$, $\rho^{(j)}(A)$ is well-defined for $A$,
and moreover $\rho^{(j)}(A) \in \mc{Q}_T^\gamma$, and
\[ \|\rho^{(j)}(A)\|_{\mc{Q}_T^\gamma} \leq \const_{\gamma, d} T^{(1/2) - \gamma} R^j. \]
Consequently,
\[ \|\rho(A)\|_{\mc{Q}_T^\gamma} \leq \const_{\gamma, d} T^{(1/2) - \gamma} (R^2 + R^3).\]
Additionally, let $A_1, A_2 \in \mc{Q}_{T, R}^\gamma$. Then for $j \in \{2, 3\}$,
\[ \|\rho^{(j)}(A_1) - \rho^{(j)}(A_2)\|_{\mc{Q}_{T}^\gamma} \leq \const_{\gamma, d} T^{(1/2) - \gamma} R^{j-1} \|A_1 - A_2\|_{\mc{Q}_T^\gamma}, \]
and consequently,
\[\|\rho(A_1) - \rho(A_2)\|_{\mc{Q}_{T}^\gamma} \leq \const_{\gamma, d} T^{(1/2) - \gamma} (R + R^2) \|A_1 - A_2\|_{\mc{Q}_T^\gamma}. \]
\end{lemma}
\begin{proof}
The second and fourth inequalities follow immediately from the first and third inequalities. Let $t \in (0, T]$. 
Define
\[ I_1 := \int_0^t \|e^{(t-s)\Delta} X^{(2)}(A(s))\|_{C^1} ds. \]
Using \eqref{eq:heat-semigroup-Cr-C-r-plus-half} with $u = 1$, and then applying Lemma \ref{lemma:X-2-X-3-estimate}, and then using that $A \in \mc{Q}_{T, R}^\gamma$, we obtain
\[\begin{split}
I_1 &\leq \const \int_0^t (t-s)^{-1/2} \|X^{(2)}(A(s))\|_{C^0} ds \\
&\leq \const \int_0^t (t-s)^{-1/2} \|A(s)\|_{C^0} \|A(s)\|_{C^1} ds \\
&\leq \const R^2 \int_0^t (t-s)^{-1/2} s^{-((1/2) + 2\gamma)} ds.
\end{split}\]
We thus obtain (using \eqref{eq:t-minus-sigma-sigma-integral-2} and the fact that $\gamma < 1/4$, so that $(1/2) + 2\gamma < 1$)
\[ t^{(1/2) + \gamma} I_1 \leq \const R^2 t^{(1/2) - \gamma}. \]
This shows that $\rho^{(2)}(A)$ is well-defined for $A$. For the $C^0$ norm, by applying \eqref{eq:heat-semigroup-Cr-contraction} and Lemma \ref{lemma:X-2-X-3-estimate} both with $r = 0$, using that $A \in \mc{Q}_{T, R}^{\gamma}$, and then using \eqref{eq:t-minus-sigma-sigma-integral-2}, we obtain
\[\begin{split}
t^\gamma \|\rho^{(2)}(A)(t)\|_{C^0} &\leq t^\gamma \int_0^t \|e^{(t-s)\Delta} X^{(2)}(A(s))\|_{C^0} ds \\
&\leq \const t^\gamma \int_0^t \|A(s)\|_{C^0} \|A(s)\|_{C^1} ds \\
&\leq \const t^\gamma R^2 \int_0^t s^{-((1/2) +2\gamma)} ds  \\
&\leq \const t^{(1/2) - \gamma} R^2.
\end{split}\]
By combining the previous few estimates, and taking sup over $t \in (0, T]$, we obtain $\|\rho^{(2)}(A)\|_{\mc{Q}_T^{\gamma}} \leq \const R^2 T^{(1/2) - \gamma}$.

Next, define
\[ I_2 := \int_0^t \|e^{(t-s)\Delta} X^{(3)}(A(s))\|_{C^1} ds. \]
Arguing similarly, we obtain for $t \in (0, T]$, (using \eqref{eq:t-minus-sigma-sigma-integral-2} as well as the fact that $\gamma < 1/4$, so that $3\gamma < 1$)
\[
t^{(1/2) + \gamma} I_2 \leq \const R^3 t^{(1/2) + \gamma} \int_0^t (t-s)^{-1/2} s^{-3\gamma} ds \leq \const R^3 t^{1-2\gamma}. \]
This shows that $\rho^{(3)}(A)$ is well-defined for $A$. For the $C^0$ norm, we may bound (applying \eqref{eq:heat-semigroup-Cr-contraction} and Lemma \ref{lemma:X-2-X-3-estimate} both with $r = 0$, using that $A \in \mc{Q}_{T, R}^{\gamma}$, and then using \eqref{eq:t-minus-sigma-sigma-integral-2})
\[ t^\gamma \|\rho^{(3)}(A)(t)\|_{C^0} \leq \const t^\gamma \int_0^t \|A(s)\|_{C^0}^3 ds \leq \const t^\gamma R^3 \int_0^t s^{-3\gamma} ds \leq \const t^{1 -2\gamma} R^3. \]
Combining the previous few estimates, and taking sup over $t \in (0, T]$, we obtain $\|\rho^{(3)}(A)\|_{\mc{Q}_T^\gamma} \leq \const R^3 T^{(1/2) - \gamma}$ (here we used that $T \in (0, 1]$, so that $T^{1 - 2\gamma} \leq T^{(1/2) - \gamma}$).


We move on to the third inequality. For $t \in (0, T]$, we may bound (using \eqref{eq:heat-semigroup-Cr-C-r-plus-half} with $u = 1$, then applying Lemma \ref{lemma:X-2-X-3-estimate}, then using that $A_1, A_2 \in \mc{Q}_{T, R}^\gamma$, and then using \eqref{eq:t-minus-sigma-sigma-integral-2})
\[\begin{split}
&\|\rho^{(2)}(A_1)(t) - \rho^{(2)}(A_2)(t)\|_{C^1} \\
&\leq  \int_0^t \|e^{(t-s)\Delta} (X^{(2)}(A_1(s)) - X^{(2)}(A_2(s)))\|_{C^1} ds \\
&\leq \const \int_0^t (t-s)^{-1/2} \|X^{(2)}(A_1(s)) - X^{(2)}(A_2(s))\|_{C^0} ds \\
&\leq \const R \|A_1 - A_2\|_{\mc{Q}_T^{\gamma}} \int_0^t (t-s)^{-1/2} s^{-((1/2) + 2\gamma)} ds \\
&\leq \const t^{-2\gamma} R \|A_1 - A_2\|_{\mc{Q}_T^\gamma}.
\end{split}\]
This implies that
\[ \sup_{t \in (0, T]} t^{(1/2) + \gamma} \|\rho^{(2)}(A_1)(t) - \rho^{(2)}(A_2)(t)\|_{C^1} \leq \const T^{(1/2) - \gamma} R \|A_1 - A_2\|_{\mc{Q}_T^\gamma}. \]
For the $C^0$ norm, applying \eqref{eq:heat-semigroup-Cr-contraction} and Lemma \ref{lemma:X-2-X-3-estimate} both with $r = 0$, using that $A_1, A_2 \in \mc{Q}_{T, R}^{\gamma}$, and then using \eqref{eq:t-minus-sigma-sigma-integral-2}, we obtain for $t \in (0, T]$
\[\begin{split}
\|\rho^{(2)}(A_1)(t) - \rho^{(2)}(A_2)(t)\|_{C^0} &\leq \const R \|A_1 - A_2\|_{\mc{Q}_T^\gamma} \int_0^t s^{-((1/2) + 2\gamma)} ds \\
&\leq \const t^{(1/2) - 2\gamma}.
\end{split}\]
By combining the previous few estimates, we obtain
\[ \|\rho^{(2)}(A_1) - \rho^{(2)}(A_2)\|_{\mc{Q}_T^\gamma} \leq \const T^{(1/2) - \gamma} R \|A_1 - A_2\|_{\mc{Q}_T^\gamma} .\]
For $j = 3$, again, for $t \in (0, T]$, we may bound (using \eqref{eq:heat-semigroup-Cr-C-r-plus-half} with $u = 1$, then applying Lemma \ref{lemma:X-2-X-3-estimate}, then using that $A_1, A_2 \in \mc{Q}_{T, R}^\gamma$, and then using \eqref{eq:t-minus-sigma-sigma-integral-2})
\[\begin{split}
\|\rho^{(3)}(A_1)(t) - \rho^{(3)}(A_2)(t)\|_{C^1} &\leq \const R^2 \|A_1 - A_2\|_{\mc{Q}_T^\gamma} \int_0^t (t-s)^{-1/2} s^{-3\gamma} ds \\
&\leq \const t^{(1/2) - 3\gamma} R^2 \|A_1 - A_2\|_{\mc{Q}_T^\gamma},
\end{split}\]
and applying \eqref{eq:heat-semigroup-Cr-contraction} and Lemma \ref{lemma:X-2-X-3-estimate} both with $r = 0$, using that $A_1, A_2 \in \mc{Q}_{T, R}^{\gamma}$, and then using \eqref{eq:t-minus-sigma-sigma-integral-2}, we obtain
\[\begin{split}
\|\rho^{(3)}(A_1)(t) - \rho^{(3)}(A_2)(t)\|_{C^0} &\leq \const R^2 \|A_1 - A_2\|_{\mc{Q}_T^\gamma} \int_0^t s^{-3\gamma} ds \\
&\leq \const t^{1 - 3\gamma} R^2 \|A_1 - A_2\|_{\mc{Q}_T^\gamma} .
\end{split} \]
Combining the previous few estimates, and recalling that $T^{1 - 2\gamma} \leq T^{(1/2) - \gamma}$, we obtain
\[ \|\rho^{(3)}(A_1) - \rho^{(3)}(A_2)\|_{\mc{Q}_T^\gamma} \leq \const T^{(1/2) - \gamma} R^2 \|A_1 - A_2\|_{\mc{Q}_T^\gamma}. \]
We have thus shown the third inequality. Let $j \in \{2, 3\}$. It remains to show that $\rho^{(j)}(A) : (0, T] \ra C^1(\torus^d, \lalg^d)$ is continuous. (Note that $\rho(A)$ maps into $C^1(\torus^d, \lalg^d)$ by our previous inequalities.) Fix $T_0 \in (0, T)$. Define $\tilde{A} : [0, T - T_0] \ra C^1(\torus^d, \lalg^d)$ by $\tilde{A}(t) := A(t + T_0)$ for $t \in [0, T - T_0]$. Since $A \in \mc{Q}_T^\gamma$, we have that $\tilde{A} \in \mc{P}_{T - T_0}^1$. Moreover, note that for $t \in [0, T-T_0]$, we have that (by Lemma \ref{lemma:rho-later-time})
\[ \rho^{(j)}(A)(t + T_0) = e^{t \Delta} (\rho^{(j)}(A)(T_0)) + \rho^{(j)}(\tilde{A})(t). \]
By Lemma \ref{lemma:contraction-estimates}, we have that $\rho^{(j)}(\tilde{A}) \in \mc{P}_{T - T_0}^{1 + 1/4} \sse \mc{P}_{T - T_0}^1$. The fact that $t \mapsto e^{t \Delta} (\rho^{(j)}(A)(T_0))$ is continuous on $[0, T - T_0]$ follows because $\rho^{(j)}(A)(T_0) \in C^1(\torus^d, \lalg^d)$. It follows that $\rho^{(j)}(A)$ is continuous on $[T_0, T]$. Since $T_0 \in (0, T)$ was arbitrary, the desired result now follows.
\end{proof}

\begin{proof}[Proof of Lemma \ref{lemma:rho-3-bounds}]
By following the same steps as in the proof of Lemma \ref{lemma:distributional-contraction-bound}, we have that if $A \in \mc{Q}_{T, R}^\gamma$, then for $t \in (0, T]$,
\[ t^{(1/2)+\gamma} \|\rho^{(3)}(A)(t)\|_{C^1} \leq  \const R^3 t^{(1/2) + \gamma} \int_0^t (t-s)^{-1/2} s^{-3\gamma} ds.\]
Multiplying both sides by $t^{-\gamma}$, using \eqref{eq:t-minus-sigma-sigma-integral-2} and the assumption that $\gamma < 1/3$, we obtain
\[ t^{1/2} \|\rho^{(3)}(A)(t)\|_{C^1} \leq \const t^{1 - 3\gamma} R^3.\]
This shows that $\rho^{(3)}(A)$ is well-defined for $A$. Moreover, as in same proof, we have
\[ t^\gamma \|\rho^{(3)}(A)(t)\|_{C^0} \leq \const t^{1 - 2\gamma} R^3, \]
which shows that $\|\rho^{(3)}(A)(t)\|_{C^0} \leq \const t^{1 -3\gamma} R^3$.
Taking sup over $t \in (0, T]$, we thus obtain $\|\rho^{(3)}(A)\|_{\mc{Q}_T^0} \leq \const T^{1 - 3\gamma} R^3$. The bound for $\|\rho^{(3)}(A_1) - \rho^{(3)}(A_2)\|_{\mc{Q}_T^0}$ may be similarly obtained by arguing as in the proof of Lemma \ref{lemma:distributional-contraction-bound}.
\end{proof}

\begin{lemma}\label{lemma:distributional-fixed-point-is-solution}
Let $T \in (0, 1]$, and let $A : (0, T] \ra C^1(\torus^d, \lalg^d)$ be a continuous function. Suppose that for all $T_0 \in (0, T)$, if we define $\tilde{A} : [0, T - T_0] \ra C^1(\torus^d, \lalg^d)$ as $\tilde{A}(t) := A(T_0 + t)$, then for all $t \in [0, T - T_0]$, $\tilde{A}(t) = e^{t \Delta} \tilde{A}(0) + \rho(\tilde{A})(t)$. Then $A \in C^\infty((0, T) \times \torus^d, \lalg^d)$, and moreover, $A$ is a solution to \eqref{eq:ZDDS} on $(0, T)$.
\end{lemma}
\begin{proof}
We claim that $A(T_0)$ is smooth for all $T_0 \in (0, T]$. Given this claim, for $T_0 \in (0, T)$, define $\tilde{A}(t) := A(T_0 + t)$ for $t \in [0, T - T_0]$. Then the assumptions of Lemma \ref{lemma:mild-solution-is-classical-solution} are satisfied, and thus $\tilde{A} \in C^\infty([0, T - T_0) \times \torus^d, \lalg^d)$ is a solution to \eqref{eq:ZDDS} on $(0, T - T_0)$, which implies that $A \in C^\infty((T_0, T) \times \torus^d, \lalg^d)$ is a solution to \eqref{eq:ZDDS} on $(T_0, T)$. Since $T_0$ was arbitrary, the desired result would then follow.

It remains to show the claim. Fix $T_0 \in (0, T]$. Take a sequence $0 < T_1 < T_2 < \cdots < T_0$ such that $T_n \uparrow T_0$. For $n \geq 1$, define $B_n(t) := A(T_n + t)$ for $t \in [0, T - T_n]$. We will show that for all $n \geq 1$, we have that $B_n \in \mc{P}_{T - T_n}^{1 + (n-1)/4}$. Given this, we would then have that for all $n \geq 1$, $A(T_0) = B^n(T_0 - T_n) \in C^{1 + (n-1)/4}(\torus^d, \lalg^d)$, and thus the desired result would follow.

We proceed by induction. The case $n = 1$ follows because $T_1 > 0$ and $A : (0, T] \ra C^1(\torus^d, \lalg^d)$ is a continuous function. Now suppose the claim is true for some $n$. First, note that for $t \in [0, T - T_n]$, we have by assumption that 
\[ B_n(t) = e^{t \Delta} B_n(0) + \rho(B_n)(t). \]
Thus, for $t \in [0, T - T_{n+1}]$, we have that
\[\begin{split}
B_{n+1}(t) &= B_n(t + (T_{n+1} - T_n)) \\
&= e^{t \Delta} e^{(T_{n+1} - T_n) \Delta} B_n(0) + \rho(B_n)(t + (T_{n+1} - T_n)). \end{split}\]
Note that $B_n(0) = A(T_n) \in C^1(\torus^d, \lalg^d)$. Thus, since $T_{n+1} - T_n > 0$, we have that $e^{(T_{n+1} - T_n) \Delta} B_n(0)$ is smooth. By Lemma \ref{lemma:contraction-estimates} and the inductive assumption (i.e., that $B_n \in \mc{P}_{T - T_n}^{1 + (n-1)/4}$), we have that $\rho(B_n) \in \mc{P}_{T - T_n}^{1 + n/4}$. Combining the previous few observations with the above display, the inductive step follows. 
\end{proof}

\begin{definition}\label{def:Y}
Let $A, B \in C^1(\torus^d, \lalg^d)$ be {\oneforms}. Recalling the definition of $X$ (Definition \ref{def:X}), define the {\oneform} $Y(A, B) \in C^0(\torus^d, \lalg^d)$ by:
\[ Y(A, B) := X(A + B) - X(A) - X(B). \]
The point here is that the following identity is satisfied:
\[ X(A + B) = X(A) + X(B) + Y(A, B), \]
i.e., $Y(A, B)$ contains all ``cross-terms" when expanding out $X(A + B)$. Note that $Y$ is a sum of terms like $[A_i, \ptl_j B_k]$, $[B_i, \ptl_j A_k]$, $[A_i, [A_j, B_k]]$, $[A_i, [B_j, B_k]]$, $[B_i, [B_j, A_k]]$, $[B_i, [A_j, A_k]]$, $[A_i, [B_j, A_k]]$, $[B_i, [A_j, B_k]]$ over particular values of $i, j, k \in [d]$.
\end{definition}

The following lemma is an analogue of Lemma \ref{lemma:X-2-X-3-estimate}.

\begin{lemma}\label{lemma:Y-estimates}
Let $A, B \in C^1(\torus^d, \lalg^d)$ be {\oneforms}. We have that
\[\begin{split}
\|Y(A, B)\|_{C^0} \leq \const_d \big(\|A\|_{C^0} &\|B\|_{C^1} + \|A\|_{C^1} \|B\|_{C^0} ~+ \\
&\|A\|_{C^0}^2 \|B\|_{C^0} + \|A\|_{C^0} \|B\|_{C^0}^2 \big). 
\end{split} \]
Additionally, let $A_1, A_2, B_1, B_2 \in C^1(\torus^d, \lalg^d)$. Suppose that 
\[ \|A_1\|_{C^0}, \|A_2\|_{C^0} \leq R_{0, A} ~\text{  and  }~ \|B_1\|_{C^0}, \|B_2\|_{C^0} \leq R_{0, B},\]
\[ \|A_1\|_{C^1}, \|A_2\|_{C^1} \leq R_{1, A} ~\text{  and  }~ \|B_1\|_{C^1}, \|B_2\|_{C^1} \leq R_{1, B}.\]
Then we have that 
\[ \begin{split}
\|Y(A_1, B_1) - Y(A_2, B_2)\|_{C^0} \leq  &~ \const_d \bigg(R_{1, B} \|A_1 - A_2\|_{C^0} + R_{0, A} \|B_1 - B_2\|_{C^1} + ~\\
& R_{1, A} \|B_1 - B_2\|_{C^0} + R_{0, B} \|A_1 - A_2\|_{C^1} +~ \\
& R_{0, A} R_{0, B} (\|A_1 - A_2\|_{C^0} + \|B_1 - B_2\|_{C^0}) ~+ \\
& R_{0, A}^2 \|B_1 - B_2\|_{C^0} + R_{0, B}^2 \|A_1 - A_2\|_{C^0}\bigg).
\end{split}\]
\end{lemma}
\begin{proof}
The proof uses the same argument as in the proof of Lemma \ref{lemma:X-2-X-3-estimate}.
\end{proof}

\begin{definition}\label{def:eta}
Let $T > 0$. Let $A, B : (0, T] \ra C^1(\torus^d, \lalg^d)$ be continuous functions. Suppose that
\[ \int_0^t \|e^{(t-s)\Delta} Y(A(s), B(s))\|_{C^1} ds < \infty \text{ for all $t \in (0, T]$.} \]
Define $\eta(A, B) : (0, T] \ra C^1(\torus^d, \lalg^d)$ by
\beq\label{eq:eta-def} \eta(A, B)(t) := \int_0^t e^{(t-s)\Delta} Y(A(s), B(s)) ds, ~~ t \in (0, T]. \eeq
In this case, we say that $\eta(A, B)$ is well-defined for $A, B$.
\end{definition}


The following lemma shows that for small enough $\gamma_1, \gamma_2$, and $A \in \mc{Q}_T^{\gamma_1}$, $B \in \mc{Q}_T^{\gamma_2}$, we have that $\eta(A, B)$ is well-defined for $A, B$, and moreover $\eta(A, B) \in \mc{Q}_T^{\gamma_2}$ (it will also be in $\mc{Q}_T^{\gamma_1}$, but we won't need this). 

\begin{lemma}\label{lemma:distributional-cross-term-bound}
Let $\gamma_1, \gamma_2 \in [0, 1/2)$, $\gamma_1 + \gamma_2 < 1/2$. Let $T \in (0, 1]$, $R_1, R_2 \in [0, \infty)$. For $A \in \mc{Q}_{T, R_1}^{\gamma_1}$, $B \in \mc{Q}_{T, R_2}^{\gamma_2}$, we have that $\eta(A, B)$ is well-defined for $A, B$, and moreover $\eta(A, B) \in \mc{Q}_T^{\gamma_2}$, and
\[ \|\eta(A, B)\|_{\mc{Q}_T^{\gamma_2}} \leq \const_{\gamma_1, \gamma_2, d} T^{(1/2) - \gamma_1}(R_1 R_2 + R_1^2 R_2 + R_1 R_2^2). \]
For $R \in [0, \infty)$ and $A^1, A^2 \in \mc{Q}_{T, R}^{\gamma_1}$, $B^1, B^2 \in \mc{Q}_{T, R}^{\gamma_2}$, we have that
\[\begin{split}
\|\eta(A^1, B^1) - &\eta(A^2, B^2)\|_{\mc{Q}_T^{\gamma_2}} \leq  \\
&\const_{\gamma_1, \gamma_2, d}  T^{(1/2) - \gamma_1} (R + R^2) \big(\|A^1 - A^2\|_{\mc{Q}_T^{\gamma_1}} + \|B^1 - B^2\|_{\mc{Q}_T^{\gamma_2}}
\big). 
\end{split}\]
\end{lemma}
\begin{proof}
Let $t \in (0, 1]$. We have that
\[ \|\eta(A, B)(t)\|_{C^1} \leq \int_0^t \|e^{(t-s)\Delta} Y(A(s), B(s))\|_{C^1} ds. \]
Applying \eqref{eq:heat-semigroup-Cr-C-r-plus-half} with $u = 1$, and then applying Lemma \ref{lemma:Y-estimates}, and then using the fact that $A \in \mc{Q}_{T, R_1}^{\gamma_1}$, $B \in \mc{Q}_{T, R_2}^{\gamma_2}$, we obtain
\[\begin{split}
\|\eta(A, B)(t)\|_{C^1} \leq \const \int_0^t (t-s)^{-1/2}&(R_1 R_2 s^{-((1/2) + \gamma_1 + \gamma_2)} ~+ \\
&R_1^2 R_2 s^{-(2\gamma_1 + \gamma_2)} + R_1 R_2^2 s^{-(\gamma_1 + 2\gamma_2)}) ds. 
\end{split}\]
Using \eqref{eq:t-minus-sigma-sigma-integral-2} (along with the assumptions that $\gamma_1, \gamma_2 < 1/2$,  $\gamma_1 + \gamma_2 < 1/2$, so that $(1/2) + \gamma_1 + \gamma_2 < 1$, $2\gamma_1 + \gamma_2<1$, and $\gamma_1 + 2\gamma_2 < 1$), we obtain that $\eta(A, B)$ is well-defined for $A, B$, and moreover
\[ t^{(1/2) + \gamma_2} \|\eta(A, B)(t)\|_{C^1} \leq \const (R_1 R_2 t^{(1/2) - \gamma_1} + R_1^2 R_2 t^{1 - 2\gamma_1} + R_1 R_2^2 t^{1 - \gamma_1 - \gamma_2}). \]
Taking sup over $t \in (0, T]$ (and using that $T \in (0, 1]$, $\gamma_1, \gamma_2 < 1/2$, so that $T^{1 - 2\gamma_1} \leq T^{(1/2) - \gamma_1}$ and $T^{1 - \gamma_1 - \gamma_2} \leq T^{(1/2) - \gamma_1}$), we obtain
\[ \sup_{t \in (0, T]} t^{(1/2) + \gamma_2} \|\eta(A, B)(t)\|_{C^1} \leq \const T^{(1/2) - \gamma_1} (R_1 R_2 + R_1^2 R_2 + R_1 R_2^2) . \]
By a similar argument, we may bound the $C^0$ norm
\[\sup_{t \in (0, T]} t^{\gamma_2} \|\eta(A, B)(t)\|_{C^0} \leq \const T^{(1/2) - \gamma_1} (R_1 R_2 + R_1^2 R_2 + R_1 R_2^2). \]
The first desired inequality now follows. The second desired inequality follows by a similar argument.

It remains to show that $\eta(A, B) : (0, T] \ra C^1(\torus^d, \lalg^d)$ is continuous (since this is part of the definition of $\mc{Q}_T^{\gamma_2}$). Towards this end, let $R = \max\{R_1, R_2\}$, and fix $t_0, t_1 \in (0, T]$, $t_0 < t_1$. We have that
\[\begin{split}
\|\eta&(A, B)(t_1) - \eta(A, B)(t_0)\|_{C^1}  \leq \int_{t_0}^{t_1} \|e^{(t_1 - s)\Delta} Y(A(s), B(s))\|_{C^1} ds ~+ \\
&\int_0^{t_0} \|e^{(t_1 - t_0)\Delta} e^{(t_0 - s)\Delta} Y(A(s), B(s)) - e^{(t_0 - s)\Delta} Y(A(s), B(s))\|_{C^1}  ds.
\end{split}\]
Let the two terms on the right hand side be $J_1(t_0, t_1)$, $J_2(t_0, t_1$). 
Arguing similar to before, we bound
\[ J_1(t_0, t_1) \leq \const \big(R^2 t_0^{-((1/2) + \gamma_1 + \gamma_2)} + R^3 t_0^{-(2\gamma_1 + \gamma_2)} + R^3 t_0^{-(\gamma_1 + 2\gamma_2)}\big) (t_1 - t_0)^{1/2}.\]
To bound $J_2(t_0, t_1)$, first apply \eqref{eq:heat-semigroup-general-time-continuity} with $r = 1$, $u = 3/2$ to obtain
\[ J_2(t_0, t_1) \leq \const (t_1 - t_0)^{1/4} \int_0^{t_0} \|e^{(t_0 - s)\Delta} Y(A(s), B(s))\|_{C^{3/2}} ds. \]
Then, apply \eqref{eq:heat-semigroup-Cr-C-r-plus-half} with $r = 0$, $u = 3/2$ to obtain the further upper bound
\[J_2(t_0, t_1) \leq \const (t_1 - t_0)^{1/4} \int_{0}^{t_0} (t_0 - s)^{-3/4} \|Y(A(s), B(s))\|_{C^0} ds.   \]
By Lemma \ref{lemma:Y-estimates}, we may bound
\[ \|Y(A(s), B(s))\|_{C^0} \leq \const\big( R^2 s^{-((1/2) + \gamma_1 + \gamma_2)} + R^3 s^{-(2\gamma_1 + \gamma_2)} + R^3 s^{-(\gamma_1 + 2\gamma_2)}\big). \]
Combining the previous two displays, and using \eqref{eq:t-minus-sigma-sigma-integral-2}, we obtain the following upper bound on $J_2(t_0, t_1)$:
\[ \const \big(R^2 t_0^{-((1/4) + \gamma_1 + \gamma_2)} + R^3 t_0^{(1/4) - (2\gamma_1 + \gamma_2)} + R^3 t_0^{(1/4) - (\gamma_1 + 2\gamma_2)}\big) (t_1 - t_0)^{1/4} .\]
Combining the estimates on $J_1, J_2$, we thus see that $\eta(A, B)$ is continuous (in fact, it is locally H\"{o}lder continuous).
\end{proof}

Observe that (for $T \in (0, 1]$) if $A, B \in \mc{P}_T^1$, then we have that (recalling Definition \ref{def:rho}, Lemma \ref{lemma:rho-in-P-T-0}, equation \eqref{eq:eta-def}, and Definition \ref{def:Y}) $\eta(A, B)$ is well-defined for $A, B$, and moreover
\beq\label{eq:rho-eta-identity} \rho(A + B) = \rho(A) + \rho(B) + \eta(A, B). \eeq
Moreover, extending $\eta(A, B)$ to $[0, T]$ by setting $\eta(A, B)(0) := 0$, we have that $\eta(A, B) \in \mc{P}_T^1$.
We also have the following analogue of Lemma \ref{lemma:rho-later-time}.

\begin{lemma}\label{lemma:eta-later-time}
Let $T > 0$, and let $A, B : (0, T] \ra C^1(\torus^d, \lalg^d)$ be continuous functions. Suppose that $\eta(A, B)$ is well-defined for $A, B$. Then for any $T_0 \in (0, T)$, the following holds. Let $\tilde{A}, \tilde{B} : [0, T - T_0] \ra C^1(\torus^d, \lalg^d)$ be defined by $\tilde{A}(t) := A(T_0 + t)$, $\tilde{B}(t) := B(T_0 + t)$ for $t \in [0, T - T_0]$. Then for all $t \in [0, T - T_0]$, we have that
\[ \eta(A, B)(T_0 + t) = e^{t\Delta} (\eta(A, B)(T_0)) + \eta(\tilde{A}, \tilde{B})(t). \]
\end{lemma}
\begin{proof}
This follows by the same argument as in the proof of Lemma \ref{lemma:rho-later-time}.
\end{proof}

\begin{proof}[Proof of Theorem \ref{thm:rough-distributional-zdds-local-existence}]
Let $\const_{\gamma_1, \gamma_2, d}$ be the maximum of the constants appearing in Lemma \ref{lemma:distributional-contraction-bound} (applied with $\gamma = \gamma_2$) and Lemma \ref{lemma:distributional-cross-term-bound}. For $M \geq 0$, define $\tau_{\gamma_1 \gamma_2}(M)$ to be the largest $T \leq 1$ such that the following inequalities are satisfied:
\beq\label{eq:tau-gamma-ineq-def}\begin{split} 
\const_{\gamma_1, \gamma_2, d}\max\{T^{(1/2) - \gamma_1}, T^{(1/2) - \gamma_2}\} ((3M)^2 + 2 (3M)^3) &\leq M, \\
\const_{\gamma_1, \gamma_2, d} \max\{T^{(1/2) - \gamma_1}, T^{(1/2) - \gamma_2}\} (3M + (3M)^2) &\leq \frac{1}{4}. \end{split}\eeq
Now, let $T = \tau_{\gamma_1 \gamma_2}(R)$, and define $V : \mc{Q}_T^{\gamma_2} \ra \mc{Q}_T^{\gamma_2}$ by $V(B) := B^1 + \rho(B) + \eta(A^1, B)$ ($V$ indeed maps into $\mc{Q}_T^{\gamma_2}$ by Lemmas \ref{lemma:distributional-contraction-bound} and \ref{lemma:distributional-cross-term-bound}). Moreover, by the definition of $\tau_{\gamma_1 \gamma_2}$, combined with Lemmas \ref{lemma:distributional-contraction-bound} and \ref{lemma:distributional-cross-term-bound}, we have that $V : \mc{Q}_{T, 3R}^{\gamma_2} \ra \mc{Q}_{T, 3R}^{\gamma_2}$ is a (1/2)-contraction. Thus by the contraction mapping theorem (using that $\mc{Q}_{T, 3R}^{\gamma_2}$ is nonempty, since $B^1 \in \mc{Q}_{T, 3R}^{\gamma_2}$, and that $\mc{Q}_{T, 3R}^{\gamma_2}$ is a complete metric space, which follows by Lemma \ref{lemma:Q-T-gamma-banach}),
we obtain a fixed point $B \in \mc{Q}_{T, 3R}^{\gamma_2}$ of $V$, i.e., $V(B) = B$.

Now let $A = A^1 + B$. For $T_0 \in (0, T)$, $t \in [0, T - T_0]$, let $\tilde{A}(t) := A(T_0 + t)$, $\tilde{A}^1(t) := A^1(T_0 + t)$, $\tilde{B}(t) := B(T_0 + t)$,  $\tilde{B}^1(t) := B^1(T_0 + t)$. One can show that for all $t \in [0, T - T_0]$, we have that 
\[ \tilde{A}(t) = e^{t\Delta} \tilde{A}(0) + \rho(\tilde{A})(t). \]
Now let $A = A^1 + B$. For $T_0 \in (0, T)$, $t \in [0, T - T_0]$, let $\tilde{A}(t) := A(T_0 + t)$, $\tilde{A}^1(t) := A^1(T_0 + t)$, $\tilde{B}(t) := B(T_0 + t)$,  $\tilde{B}^1(t) := B^1(T_0 + t)$. We claim that for all $t \in [0, T - T_0]$, we have that 
\[ \tilde{A}(t) = e^{t\Delta} \tilde{A}(0) + \rho(\tilde{A})(t). \]
Given this claim, we would have by Lemma \ref{lemma:distributional-fixed-point-is-solution} that $A \in C^\infty((0, T) \times \torus^d, \lalg^d)$ is a solution to \eqref{eq:ZDDS} on $(0, T)$.

Thus to complete the proof of the first part of the theorem, it remains to show the claim. The claim follows upon combining, for any $t \in [0, T - T_0]$, the following sequence of identities:
\[ \tilde{A}^1(t) = e^{t\Delta} (\tilde{A}^1(0)), \]
\[ \tilde{B}(t) = \tilde{B}^1(t) + \rho(B)(T_0 + t) + \eta(A^1, B)(T_0 + t),\]
\[ \tilde{B}^1(t) = e^{t\Delta} B^1(T_0) + \rho(\tilde{A}^1)(t), \]
\[ \rho(B)(T_0 + t) = e^{t\Delta}( \rho(B)(T_0)) + \rho(\tilde{B})(t), \]
\[ \eta(A^1, B)(T_0 + t) = e^{t\Delta} (\eta(A^1, B)(T_0)) + \eta(\tilde{A}^1, \tilde{B})(t), \]
\[ \tilde{B}(0) = B(T_0) = B^1(T_0) + \rho(B)(T_0) + \eta(A^1, B)(T_0), \]
\[ \rho(\tilde{A}^1)(t) + \rho(\tilde{B})(t) + \eta(\tilde{A}^1, \tilde{B})(t) = \rho(\tilde{A}^1 + \tilde{B})(t) = \rho(\tilde{A})(t).\]
The first identity follows by the assumption on $A^1$, the second identity follows because $B = V(B)$, the third identity follows because $B^1$ is by assumption a first nonlinear part for $A^1$, the fourth identity follows by Lemma \ref{lemma:rho-later-time}, the fifth identity follows by Lemma \ref{lemma:eta-later-time}, the sixth identity follows since $B = V(B)$, and the seventh identity follows by \eqref{eq:rho-eta-identity}. To see how these identities give the claim, note that using the first two identities, we have
\[\begin{split}
\tilde{A}(t) &= \tilde{A}^1(t) + \tilde{B}(t) \\
&= e^{t\Delta}(\tilde{A}^1(0)) + \tilde{B}^1(t) + \rho(B)(T_0 + t) + \eta(A^1, B)(T_0 + t). 
\end{split}\]
The third identity then gives that the above is equal to
\[ e^{t\Delta}(\tilde{A}^1(0) + B^1(T_0)) + \rho(\tilde{A}^1)(t) + \rho(B)(T_0 + t) + \eta(A^1, B)(T_0 + t) .\]
Combining the fourth, fifth, and sixth identities, we obtain that the above is equal to
\[ e^{t\Delta} (\tilde{A}^1(0) + \tilde{B}(0)) + \rho(\tilde{A}^1)(t) + \rho(\tilde{B})(t) + \eta(\tilde{A}^1, \tilde{B})(t).\]
Now using that $\tilde{A}(0) = \tilde{A}^1(0) + \tilde{B}(0)$ (because $A = A^1 + B$) and the seventh identity, we obtain that
\[ \tilde{A}(t) = e^{t\Delta} \tilde{A}(0) + \rho(\tilde{A})(t), \]
as desired.

For the second part of the theorem, for $n \geq 1$ let $V_n : \mc{Q}_{T_n, 3R_n}^{\gamma_2} \ra \mc{Q}_{T_n, 3R_n}^{\gamma_2}$ be defined by $V_n(B) := B^1_n + \rho(B) + \eta(A^1_n, B)$. By the first part of the theorem, for each $n \geq 1$, we obtain a fixed point $B_n \in \mc{Q}_{T_n, 3R_n}^{\gamma_2}$ of $V_n$, such that $A_n = A^1_n + B_n \in C^\infty((0, T_n) \times \torus^d, \lalg^d)$ is a solution to \eqref{eq:ZDDS} on $(0, T_n)$. Now fix $T_0 \in (0, T)$. For large enough $n$, we have that $T_n > T_0$ (since $T_n \ra T$). Thus by Lemmas \ref{lemma:distributional-contraction-bound} and \ref{lemma:distributional-cross-term-bound}, along with the definition of $\tau_{\gamma_1 \gamma_2}$, we obtain
\[ \|\rho(B_n) - \rho(B)\|_{\mc{Q}_{T_0}^{\gamma_2}} \leq \frac{1}{4} \|B_n - B\|_{\mc{Q}_{T_0}^{\gamma_2}}, \]
\[ \|\eta(A^1_n, B_n) - \eta(A^1, B)\|_{\mc{Q}_{T_0}^{\gamma_2}} \leq \frac{1}{4} (\|A^1_n - A^1\|_{\mc{Q}_{T_0}^{\gamma_1}} + \|B_n - B\|_{\mc{Q}_{T_0}^{\gamma_2}}). \]
From this, it follows that
\[\begin{split}
\|B_n - B\|_{\mc{Q}_{T_0}^{\gamma_2}} &= \|V_n(B_n) - V(B)\|_{\mc{Q}_{T_0}^{\gamma_2}} \\
&\leq \|B_n^1 - B^1\|_{\mc{Q}_{T_0}^{\gamma_2}} + \frac{1}{2}\|B_n - B\|_{\mc{Q}_{T_0}^{\gamma_2}} + \frac{1}{4} \|A_n^1 - A^1\|_{\mc{Q}_{T_0}^{\gamma_1}}, 
\end{split}\]
and thus we obtain that $\|B_n - B\|_{\mc{Q}_{T_0}^{\gamma_2}} \ra 0$, as desired. The convergence of $A_n$ to $A$ then follows upon noting that
\[ \|A_n - A\|_{\mc{Q}_{T_0}^{\max(\gamma_1, \gamma_2)}} \leq \|A_n^1 - A^1\|_{\mc{Q}_{T_0}^{\gamma_1}} + \|B_n - B\|_{\mc{Q}_{T_0}^{\gamma_2}}. \qedhere \]
\end{proof}

\begin{proof}[Proof of Lemma \ref{lemma:tau-gamma-bound}]
Let $\gamma = \max\{\gamma_1, \gamma_2\}$. From the definition of $\tau_{\gamma_1 \gamma_2}$ (recall the inequalities \eqref{eq:tau-gamma-ineq-def}), we have that for $R \geq 0$,
\[(\tau_{\gamma_1 \gamma_2}(R))^{(1/2) - \gamma} \geq \min\bigg\{1, \frac{1}{\const R^2}\bigg\}, \]
for some big enough $\const$ (which depends on $\gamma_1, \gamma_2, d$). From this, we obtain
\[ (\tau_{\gamma_1, \gamma_2}(R))^{-1} \leq \const + \const R^{2 / ((1/2) - \gamma)},\]
which gives the desired result.
\end{proof}



\begin{proof}[Proof of Lemma \ref{lemma:zdds-smooth-initial-data-distributional-local-existence}]
From the proof of Theorem \ref{thm:rough-distributional-zdds-local-existence},  we obtain $B \in \mc{Q}_{T, 3R}^{\gamma_2}$ such that $B = B^1 + \rho(B) + \eta(A^1, B)$, and moreover $A^1 + B \in C^\infty((0, T) \times \torus^d, \lalg^d)$ is a solution to \eqref{eq:ZDDS} on $(0, T)$. We claim that if we extend $B$ to $[0, T]$ by setting $B(0) := 0$, then we have that $B \in \mc{P}_T^1$. Given this claim, we would then be able to write (recall \eqref{eq:rho-eta-identity}):
\[ B = \rho(A^1) + \rho(B) + \eta(A^1, B) = \rho(A^1 + B), \]
which would imply
\[ A^1 + B = A^1 + \rho(A^1 + B). \]
The desired result would then follow by Lemma \ref{lemma:mild-solution-is-classical-solution}. 

Thus it just remains to show the claim. We need to show that $\lim_{t \downarrow 0} \|B(t)\|_{C^1} = 0$. Since $A^0$ is smooth, so is $A^1$, and thus by Lemma \ref{lemma:contraction-estimates}, we have that $B^1 = \rho(A^1)$ is such that $\lim_{t \downarrow 0} \|B^1(t)\|_{C^1} = 0$. Thus, it suffices to show that 
\[ \lim_{t \downarrow 0} \|\rho(B)(t)\|_{C^1} = 0, ~~~ \lim_{t \downarrow 0} \|\eta(A^1, B)(t)\|_{C^1} = 0. \]
Towards this end, first note that it follows from Lemma \ref{lemma:distributional-contraction-bound} (applied with $\gamma = \gamma_2$) that for any $s \in (0, T]$, we have that
\[ \begin{split}
s^{\gamma_2} \|\rho(B)(s)\|_{C^0} + s^{(1/2) + \gamma_2} \|\rho(B)(s)\|_{C^1} &\leq \|\rho(B)\|_{\mc{Q}_s^{\gamma_2}}\\
&\leq \const_{\gamma_2, d} s^{(1/2) - \gamma_2} ((3R)^2 + (3R)^3),
\end{split}\]
which implies
\[ \|\rho(B)(s)\|_{C^0} + s^{1/2} \|\rho(B)(s)\|_{C^1} \leq \const_{\gamma_2, d} s^{(1/2) - 2\gamma_2} ((3R)^2 + (3R)^3).\]
Let $t \in (0, T]$. Taking sup over $s \in (0, t]$, we obtain
\[ \|\rho(B)\|_{\mc{Q}_t^0} \leq \const_{\gamma_2, d} t^{(1/2) - 2\gamma_2} ((3R)^2 + (3R)^3).\]
Similarly, it follows from Lemma \ref{lemma:distributional-cross-term-bound}, applied with $\gamma_1 = 0$ and 
\[
\tilde{R} := \max\{R, \|A^1\|_{\mc{Q}_T^0}\}
\]
in place of $R$, that for any $s \in (0, T]$, we have that
\begin{align*}
&s^{\gamma_2} \|\eta(A^1, B)(s)\|_{C^0} + s^{(1/2) + \gamma_2} \|\eta(A^1, B)(s)\|_{C^1} \\
&\leq \const_{\gamma_2, d} s^{1/2} ((3\tilde{R})^2 + 2(3\tilde{R})^3), 
\end{align*}
which implies
\[ \|\eta(A^1, B)(s)\|_{C^0} + s^{1/2} \|\eta(A^1, B)(s)\|_{C^1} \leq \const_{\gamma_2, d} s^{(1/2) - \gamma_2} ((3\tilde{R})^2 + 2(3\tilde{R})^3).\]
Let $t \in (0, T]$. Taking sup over $s \in (0, t]$, we obtain
\[ \|\eta(A^1, B)\|_{\mc{Q}_t^0} \leq \const_{\gamma_2, d} t^{(1/2) - \gamma_2} ((3\tilde{R})^2 + 2(3\tilde{R})^3). \]
Using that $B = \rho(A^1) + \rho(B) + \eta(A^1, B)$ and the assumption that $\gamma_2 < 1/4$, we thus obtain that $B \in \mc{Q}_T^0$, and moreover $\lim_{t \downarrow 0} \|B\|_{\mc{Q}_t^0} = 0$. Now applying Lemma~\ref{lemma:distributional-contraction-bound} again (this time with $\gamma = 0$), we obtain
\[ t^{1/2} \|\rho(B)(t)\|_{C^1} \leq \|\rho(B)\|_{\mc{Q}_t^0} \leq \const_d t^{1/2} (\|B\|_{\mc{Q}_t^0}^2 + \|B\|_{\mc{Q}_t^0}^3), ~~ t \in (0, T], \]
from which we obtain
\[ \|\rho(B)(t)\|_{C^1} \leq \const_d (\|B\|_{\mc{Q}_t^0}^2 + \|B\|_{\mc{Q}_t^0}^3), ~~ t \in (0, T].\]
Thus we obtain $\lim_{t \downarrow 0} \|\rho(B)(t)\|_{C^1} = 0$. Similarly, applying Lemma \ref{lemma:distributional-cross-term-bound} (with $\gamma_1 = \gamma_2 = 0$), we obtain for $t \in (0, T]$,
\[ \|\eta(A^1, B)(t)\|_{C^1} \leq \const (\|A^1\|_{\mc{Q}_t^0} \|B\|_{\mc{Q}_t^0} + \|A^1\|_{\mc{Q}_t^0}^2 \|B\|_{\mc{Q}_t^0} + \|A^1\|_{\mc{Q}_t^0}\|B\|_{\mc{Q}_t^0}^2). \]
Using that $\sup_{t \in (0, T]} \|A^1\|_{\mc{Q}_t^0} = \|A^1\|_{\mc{Q}_T^0} < \infty$ (since $A^1$ is smooth) and that $\lim_{t \downarrow 0} \|B\|_{\mc{Q}_t^0} = 0$, we thus obtain that $\lim_{t \downarrow 0} \|\eta(A^1, B)(t)\|_{C^1} = 0$. The desired result follows. 
\end{proof}

\begin{proof}[Proof of Corollary \ref{cor:B-limit-B-n}]
This follows immediately from the proof of Theorem \ref{thm:rough-distributional-zdds-local-existence}, and in particular the fact that $B$ was constructed as a fixed point of the contraction map $V$. Thus the proof is omitted.
\end{proof}

\section{Proof of Lemma \ref{lemma:tau-1-tau-2-bound}}\label{appendix:tau1-tau2-bound}

\begin{lemma}\label{lemma:technical-sum-bound}
Let $\alpha > d/2$. For any $n \in \Z^d$, we have that
\[ \sum_{\substack{k \in \Z^d \\ k \neq 0, n}} \frac{1}{|k|^\alpha |n - k|^\alpha} \leq \const \min\{1, |n|^{-(2\alpha - d)}\}. \]
Here, $\const$ depends only on $d, \alpha$.
\end{lemma}
\begin{proof}
Fix $n  \in \Z^d$. We split into four cases: (1) $|k| > 2|n|$, (2) $|k| \leq 2|n|$, $|k| > |n|/4$, $|n-k| > |n| /4$, (3) $|k| \leq 2|n|$, $|k| \leq |n| /4$, $|n-k| > |n|/4$, (4) $|k| \leq 2|n|$, $|k| > |n|/4$, $|n-k| \leq |n|/4$ (note that the case $|k|\le |n|/4$, $|n-k| \leq |n|/4$ is impossible). Let the contributions from these four cases be $I_1, I_2, I_3, I_4$, respectively. We proceed to bound $I_1, I_2, I_3, I_4$ individually. First, note that when $|k| > 2|n|$, we have that $|n-k| \geq |k|/2$, and thus
\[ I_1 \leq \const \sum_{\substack{k \in \Z^d \\ |k| > 2|n|}} |k|^{-2\alpha} \leq \const \sum_{r=2|n|}^\infty r^{d-1-2\alpha} \leq \const \min\{1, |n|^{-(2\alpha - d)}\}.\]
(Here we used that $\alpha > d/2$ to ensure that the infinite sum is summable.) Next, note that for $n = 0$, we have that $I_2 = I_3 = I_4 = 0$, due to the restriction $k \neq 0$. Thus, assume that $n \neq 0$ in what follows. We have that
\[ I_2 \leq \const |n|^{-2\alpha} \sum_{\substack{k \in \Z^d \\ |k| \leq 2|n|}} 1 \leq \const |n|^{d-2\alpha},\]
and that
\[ I_3 \leq \const |n|^{-\alpha} \sum_{\substack{k \in \Z^d \\ |k| \leq 2|n| \\ k \neq 0}} |k|^{-\alpha} \leq \const |n|^{-\alpha} \sum_{r=0}^{2|n|} r^{d-1-\alpha} \leq \const |n|^{d-2\alpha}.\]
The bound for $I_4$ may be handled similarly, by changing variables $k \mapsto n-k$. The desired result now follows by combining the previous few estimates.
\end{proof}

\begin{lemma}\label{lemma:greens-function-product-smooth-bound-sum}
Let $\alpha \in (d/2, d)$. Let $t \in (0, 1]$, $u \in [0, 2t]$. We have that
\[ \sum_{\substack{n^1, n^2 \in \Z^d \\ n^1, n^2 \neq 0}} e^{-\pi^2 |n^1 + n^2|^2 (2t-u)} e^{-\pi^2 (|n^1|^2 + |n^2|^2) u} \frac{1}{|n^1|^\alpha |n^2|^\alpha} \leq \const t^{-(d-\alpha)}. \]
Here, $\const$ depends only on $d, \alpha$.
\end{lemma}
\begin{proof}
We may rewrite the sum as
\[ \sum_{n \in \Z^d} e^{-\pi^2 |n|^2 (2t-u)}\bigg( \sum_{\substack{n^1 \in \Z^d \\n^1 \neq 0, n}} e^{-\pi^2 (|n^1|^2 + |n - n^1|^2)u} \frac{1}{|n^1|^\alpha |n - n^1|^\alpha}\bigg).  \]
Using that $|n|^2 / 2 \leq |n^1|^2 + |n-n^1|^2$ for all $n, n^1 \in \Z^d$, we obtain the upper bound
\[ \sum_{n \in \Z^d} e^{-\pi^2 |n|^2 (2t-(u/2))} \bigg(\sum_{\substack{n^1 \in \Z^d \\n^1 \neq 0, n}}  \frac{1}{|n^1|^\alpha |n - n^1|^\alpha}\bigg).\]
Applying Lemma \ref{lemma:technical-sum-bound}, we obtain the further upper bound (using $u \leq t$ in the first inequality)
\[\begin{split}
\const \sum_{n \in \Z^d} &e^{-\pi^2 |n|^2 (2t - (u/2))} \min\{1, |n|^{-(2\alpha - d)}\} \\
&\leq \const \sum_{n \in \Z^d} e^{-\pi^2 |n|^2 t} \min\{1, |n|^{-(2\alpha - d)}\}  \leq \const  + \const \sum_{r = 1}^\infty r^{2d - 1 -2\alpha} e^{-\pi^2 r^2 t}.
\end{split} \]
Now by arguing as in the proof of Lemma \ref{lemma:fractional-greens-function-heat-kernel}, we obtain (using that $\alpha < d$)
\[ \sum_{r = 1}^\infty r^{2d - 1 -2\alpha} e^{-\pi^2 r^2 t} \leq \const t^{-(d-\alpha)}.\]
The desired result now follows by combining the previous few estimates (and using that $t \in (0, 1]$ to bound $\const \leq \const t^{-(d-\alpha)}$).
\end{proof}

\begin{lemma}\label{lemma:greens-function-smooth-product-bound}
Let $\alpha \in (d/2, d)$. Let $t \in (0, 1]$ and $u \in [0, 2t]$. We have that
\[ \|e^{(2t - u)\Delta} \big((e^{(u/2) \Delta} G_0^\alpha)^2\big)\|_{C^0} \leq  \const t^{-(d-\alpha)}.\]
Here, $\const$ depends only on $d, \alpha$.
\end{lemma}
\begin{proof}
Recall that $G_0^\alpha$ has Fourier coefficients $\hat{G}_0^\alpha(0) = 0$, $\hat{G}_0^\alpha(n) = |n|^{-\alpha}$, $n \in \Z^d \setminus \{0\}$. Thus, 
\[ (e^{(u/2)\Delta} G_0^\alpha)^2(x) = \sum_{\substack{n^1, n^2 \in \Z^d \\n^1, n^2 \neq 0}} e_{n^1 + n^2}(x) e^{-2\pi^2 (|n^1|^2 + |n^2|^2)u} \frac{1}{|n^1|^\alpha |n^2|^\alpha}. \]
Therefore, 
\[\begin{split}
\big(e^{(2t - u)\Delta}&\big((e^{(u/2)\Delta} G_0^\alpha)^2 \big)\big)(x) = \\
&\sum_{\substack{n^1, n^2 \in \Z^d \\n^1, n^2 \neq 0}} e_{n^1 + n^2}(x) e^{-4\pi^2 |n^1 + n^2|^2 (2t-u)} e^{-2\pi^2 (|n^1|^2 + |n^2|^2)u}  \frac{1}{|n^1|^\alpha |n^2|^\alpha}.
\end{split}\]
We thus have that
\[ \begin{split}
\|e^{(2t - u)\Delta}& \big((e^{(u/2) \Delta} G_0^\alpha)^2\big)\|_{C^0} \leq \\
&\sum_{\substack{n^1, n^2 \in \Z^d \\n^1, n^2 \neq 0}} e^{-4\pi^2 |n^1 + n^2|^2 (2t-u)} e^{-2\pi^2 (|n^1|^2 + |n^2|^2)u}  \frac{1}{|n^1|^\alpha |n^2|^\alpha}.
\end{split}\]
The desired result now follows by Lemma \ref{lemma:greens-function-product-smooth-bound-sum} (note that $e^{-4\pi^2 |n^1 + n^2|^2 (2t-u)} \leq e^{-\pi^2 |n^1 + n^2|^2 (2t-u)}$ and $e^{-2\pi^2 (|n^1|^2 + |n^2|^2)u} \leq e^{-\pi^2 (|n^1|^2 + |n^2|^2)u}$).
\end{proof}

\begin{lemma}\label{lemma:greens-function-smooth-product-bound-2}
Let $\alpha \in (d/2, d)$. Let $t \in (0, 1]$ and $u \in [0, 2t]$. For any $\const_1 \geq 0$, we have that
\[ \|e^{(2t - u)\Delta} \big(\big(e^{(u/2) \Delta} (G_0^\alpha + \const_1)\big)^2\big)\|_{C^0} \leq  \const_2 t^{-(d-\alpha)}.\]
Here, $\const_2$ depends only on $d$, $\alpha$, and $\const_1$.
\end{lemma}
\begin{proof}
This follows by Lemmas \ref{lemma:fractional-greens-function-heat-kernel} and \ref{lemma:greens-function-smooth-product-bound} and the fact that $t \in (0, 1]$.
\end{proof}

\begin{proof}[Proof of Lemma \ref{lemma:tau-1-tau-2-bound}]
Fix $k \geq 0$. Recall that $\hat{\tau}(n) \geq 0$ for all $n \in \Z^d$ (by Lemma \ref{lemma:fourier-coeff-uncorrelated}). Observe that
\[\begin{split}
I((n^1, n^2), t)^2 &= \int_0^t \int_0^t e^{-4\pi^2 |n^1 + n^2|^2(2t - (s_1 + s_2))} e^{-4\pi^2 (|n^1|^2 + |n^2|^2)(s_1 + s_2)} ds_1 ds_2  \\
&= \int_0^{2t} h_t(u) e^{-4\pi^2 |n^1 + n^2|^2 (2t - u)} e^{-4\pi^2 (|n^1|^2 + |n^2|^2) u} du,
\end{split}\]
where $h_t$ is the function given by
\[ h_t(u) = \begin{cases} u & u \in [0, t] \\ 2t - u & u \in [t, 2t] \end{cases}.\]
Thus it suffices to bound 
\[\begin{split}
J:=  \sum_{\substack{n^1, n^2 \in \Z^d \\ \max\{|n^1|_\infty, |n^2|_\infty\} \geq N}} &\int_0^{2t} h_t(u)  |n^1+n^2|^k e^{-4\pi^2 |n^1 + n^2|^2 (2t-u)}~\times \\
&e^{-4\pi^2 (|n^1|^2 + |n^2|^2) u} (|n^1|^2 + |n^2|^2) \hat{\tau}(n^1) \hat{\tau}(n^2) du. 
\end{split}\]
We may split the integral $\int_0^{2t} = \int_0^t + \int_t^{2t}$. Let the two contributions be $J_1, J_2$, respectively. In the following, we will use that for $\beta \geq 0$,
\beq\label{eq:sup-poly-exp} \sup_{x \geq 0} x^{\beta} e^{-2\pi^2 x^2 s} \leq \const_\beta s^{-\beta /2}, ~~ s > 0. \eeq
For $N \geq 0$, $u \in [0, 2t]$, define 
\begin{align}
F_{N, t}(u) &:= \sum_{\substack{n^1, n^2 \in \Z^d \\ \max\{|n^1|_\infty, |n^2|_\infty\} \geq N}}e^{-4\pi^2 |n^1 + n^2|^2 (2t-u)} \times \notag\\
&\qquad \qquad \qquad \qquad \qquad \qquad  e^{-2\pi^2 (|n^1|^2 + |n^2|^2)u} \hat{\tau}(n^1) \hat{\tau}(n^2). \label{eq:F-N-t-def}
\end{align}
(Note that the the exponent in the second exponential is $2\pi$ instead of $4\pi$, unlike for $J$.)
Applying \eqref{eq:sup-poly-exp} with $\beta = k$, $s = 2t-u$, as well as with $\beta = 2$, $s = u$, we may bound $J_1$ by 
\[\begin{split}
J_1 \leq \const_k \sum_{\substack{n^1, n^2 \in \Z^d \\ \max\{|n^1|_\infty, |n^2|_\infty\} \geq N}} \int_0^t &h_t(u)(2t-u)^{-k/2} e^{-2\pi^2 |n^1 + n^2|^2 (2t-u)} ~\times\\
& u^{-1} e^{-2\pi^2 (|n^1|^2 + |n^2|^2)u} \hat{\tau}(n^1) \hat{\tau}(n^2) du.
\end{split} \]
Using that $(2t-u)^{-k/2} \leq t^{-k/2}$ and $h_t(u) = u$ for $u \in [0, t]$, the definition of $F_{N, t}$, and that $e^{-2\pi^2 (|n^1|^2 + |n^2|^2)u} \leq e^{-\pi^2 (|n^1|^2 + |n^2|^2)u}$, we obtain
\[ J_1 \leq \const_k t^{-k/2} \int_0^t F_{N, t/2}(u/2) du .\]
To bound $J_2$, first note that we may bound
\[ |n^1 + n^2|^k \leq \const_k (|n^1|^k + |n^2|^k), \]
and thus we may bound (using Young's inequality, i.e. $xy \leq (x^p / p) + (y^q / q)$ for $x, y \geq 0$, $(1/p) + (1/q) = 1$),
\[ |n^1 + n^2|^k (|n^1|^2 + |n^2|^2) \leq \const_k (|n^1|^{k+2} + |n^2|^{k+2}). \]
Then applying \eqref{eq:sup-poly-exp} with $\beta = k+2$, $s = u$, we obtain
\[\begin{split}
J_2 \leq \const_k  \sum_{\substack{n^1, n^2 \in \Z^d \\ \max\{|n^1|_\infty, |n^2|_\infty\}}} \int_t^{2t} h_t(u) &e^{-4\pi^2 |n^1 + n^2|^2 (2t-u)} u^{-(k+2)/2} ~\times \\
&e^{-2\pi^2 (|n^1|^2 + |n^2|^2) u} \hat{\tau}(n^1) \hat{\tau}(n^2) du. \end{split} \]
Using that $u^{-(k+2)/2} \leq t^{-(k+2)/2}$ and $h_t(u) = (2t-u)$ for $u \in [t, 2t]$, as well as the definition of $F_{N, t}(u)$, we further obtain
\[ J_2 \leq \const_k t^{-(k+2)/2} \int_t^{2t} (2t-u) F_{N, t}(u) du. \]
We have thus shown
\beq\label{eq:S-N-k-bound}\begin{split}
S(N, k, a_1, a_2, j_1, j_2,& t) \leq \const_k t^{-k/2} \int_0^t F_{N, t/2}(u/2) du ~+ \\
&\const_k t^{-(k+2)/2} \int_t^{2t} (2t-u) F_{N, t}(u) du. 
\end{split} \eeq
Recalling the definition \eqref{eq:F-N-t-def} of $F_{N, t}$, note that for $u \in [0, 2t]$, we have that
\[\begin{split}
F_{0, t}(u) = \big(e^{(2t-u) \Delta} \big((e^{(u/2) \Delta} \tau)^2\big)\big)(0). \end{split}\]
Using that $|\tau| \leq \constgf (G_0^\alpha + \constgf)$ (by Assumption \ref{assumption:bounded-by-fractional-greens-function}) combined with \eqref{eq:heat-semigroup-monotonicity}, we obtain that $|e^{(u/2) \Delta} \tau| \leq \constgf e^{(u/2) \Delta}(G_0^\alpha + \constgf)$, and thus also (applying \eqref{eq:heat-semigroup-monotonicity} once more)
\[\big|e^{(2t-u) \Delta} \big((e^{(u/2) \Delta} \tau)^2\big)\big| \leq (\constgf)^2 e^{(2t-u) \Delta} \big(\big(e^{(u/2) \Delta} (G_0^\alpha + \constgf)\big)^2 \big). \]
Combining this with Lemma \ref{lemma:greens-function-smooth-product-bound-2}, we obtain that
\beq\label{eq:intermediate-sum-tau-product-bound} F_{0, t}(u) \leq \const t^{-(d-\alpha)}, ~~ t \in (0, 1], u \in [0, 2t].\eeq
It follows that (using \eqref{eq:S-N-k-bound})
\[ S(0, k, a_1, a_2, j_1, j_2, t) \leq \const_k t^{-(d-1 + (k/2) - \alpha)}. \]
Thus we may set $\delta_{0, k}^{\ref{lemma:tau-1-tau-2-bound}} \equiv 1$.

To define $\delta_{N, k}^{\ref{lemma:tau-1-tau-2-bound}}$ for general $N$, observe that for any $N \geq 0$, we have that
\[ S(N, k, a_1, a_2, j_1, j_2, t) \leq S(0, k, a_1, a_2, j_1, j_2, t).\]
Thus for $N > 0$, we may define
\[ \delta_{N, k}^{\ref{lemma:tau-1-tau-2-bound}}(t_0) := \sup_{t \in [t_0, 1]} \max_{\substack{a_1, a_2 \in [\lalgdim] \\ j_1, j_2 \in [d]}} \frac{S(N, k, a_1, a_2, j_1, j_2, t)}{\max\{1, S(0, k, a_1, a_2, j_1, j_2, t)\}}.\]
This ensures that $\delta_{N, k}^{\ref{lemma:tau-1-tau-2-bound}}$ maps into $[0, 1]$ and is non-increasing. To show the pointwise convergence to $0$, it suffices to show that for any $a_1, a_2 \in [\lalgdim]$, $j_1, j_2 \in [d]$, $t_0 \in (0, 1]$, we have that 
\[ \lim_{N \toinf} \sup_{t \in [t_0, 1]} S(N, k, a_1, a_2, j_1, j_2, t) = 0.\]
Fix $a_1, a_2, j_1, j_2$. Recalling \eqref{eq:S-N-k-bound}, it suffices to show that
\beq\label{eq:sup-int-0-t-limit} \lim_{N \toinf} \sup_{t \in [t_0, 1]} \int_0^t F_{N, t/2}(u/2) du = 0, \eeq
and 
\beq\label{eq:sup-int-t-2t-limit} \lim_{N \toinf} \sup_{t \in [t_0, 1]}\int_t^{2t}(2t-u) F_{N, t}(u) du = 0.\eeq
For $N \geq 0$, $u \geq 0$, define
\[ G_N(u) := \sum_{\substack{n \in \Z^d \\ |n|_\infty \geq N}} e^{-2\pi^2 |n|^2 u} \hat{\tau}(n).  \]
For $u > 0$, note that $G_N(u) \leq G_0(u) = (e^{(u/2) \Delta} \tau)(0) < \infty$ (the last inequality follows by Lemma \ref{lemma:heat-semigroup-covariance}), and thus by dominated convergence, we have that for $u > 0$, $\lim_{N \toinf} G_N(u) = 0$. 

We claim that for any $t \in (0, 1]$, $u \in (0, 2t]$, we have that
\[ F_{N, t}(u) \leq \const \min\big\{t^{-(d-\alpha)}, u^{-(1/2)(d-\alpha)} G_N(u) \big\}.\]
Given this claim, we may bound
\[ \sup_{t \in [t_0, 1]} \int_0^t F_{N, t/2}(u/2) du \leq \const  \int_0^1 \min\{t_0^{-(d-\alpha)}, u^{-(1/2) (d-\alpha)} G_N(u/2)\} du .\]
By dominated convergence and the fact that $\lim_{N \toinf} G_N(u) = 0$ for all $u \in (0, 1]$, we obtain \eqref{eq:sup-int-0-t-limit}. By a similar argument, we may obtain \eqref{eq:sup-int-t-2t-limit}.

It remains to show the claim. First, note that by \eqref{eq:intermediate-sum-tau-product-bound}, $F_{N, t}(u) \leq F_{0, t}(u) \leq \const t^{-(d-\alpha)}$. Next, recalling the definition \eqref{eq:F-N-t-def} of $F_{N, t}$, we have that (using that $u \in (0, 2t]$, so that $u \leq 2t$)
\[ F_{N, t}(u) \leq \sum_{\substack{n^1, n^2 \in \Z^d \\ \max\{|n^1|_\infty, |n^2|_\infty\} \geq N}} e^{-2\pi^2 |n^1|^2 u} \hat{\tau}(n^1) e^{-2\pi^2 |n^2|^2 u} \hat{\tau}(n^2). \]
Since $\max\{|n^1|_\infty, |n^2|_\infty\} \geq N$ implies that at least one of $|n^1|_\infty, |n^2|_\infty$ is at least $N$, we may further bound the above by
\[\begin{split}
F_{N, t}(u) &\leq 2 G_N(u) \sum_{n \in \Z^d} e^{-2\pi^2 |n|^2 u} \hat{\tau}(n) \\
&= 2 G_N(u) (e^{(u/2)\Delta} \tau)(0) \\
&\leq \const G_N(u) u^{-(1/2)(d-\alpha)},
\end{split}\]
where we used Lemma \ref{lemma:heat-semigroup-covariance} in the last inequality. The claim now follows.
\end{proof}

\bibliographystyle{plainnat}

\end{document}